\documentclass[11pt]{amsart}
\usepackage{amscd,amsmath,amssymb,amsfonts,verbatim}
\usepackage[cmtip, all]{xy}
\usepackage{MnSymbol}
\usepackage{comment}

\newtheorem{thm}{Theorem}[section]
\newtheorem{prop}[thm]{Proposition}
\newtheorem{lem}[thm]{Lemma}
\newtheorem{cor}[thm]{Corollary}

\theoremstyle{definition}
\newtheorem{ques}[thm]{Question}

\newtheorem{defn}[thm]{Definition}
\theoremstyle{remark}
\newtheorem{remk}[thm]{Remark}
\newtheorem{remks}[thm]{Remarks}

\newtheorem{exm}[thm]{Example}
\newtheorem{exms}[thm]{Examples}
\newtheorem{notat}[thm]{Notation}
\numberwithin{equation}{section}

{\hfill$\square$\end{defn}}
{\hfill$\square$\end{remk}}
{\hfill$\square$\end{remks}}
{\hfill$\square$\end{exm}}
{\hfill$\square$\end{exms}}
{\hfill$\square$\end{notat}}

\newcommand{\thmref}{Theorem~\ref}
\newcommand{\propref}{Proposition~\ref}
\newcommand{\corref}{Corollary~\ref}

\newcommand{\lemref}{Lemma~\ref}

\newcommand{\sD}{{\mathcal D}}

\newcommand{\sF}{{\mathcal F}}

\newcommand{\sI}{{\mathcal I}}

\newcommand{\sK}{{\mathcal K}}

\newcommand{\sO}{{\mathcal O}}
\newcommand{\sP}{{\mathcal P}}
\newcommand{\sQ}{{\mathcal Q}}
\newcommand{\sR}{{\mathcal R}}

\newcommand{\sU}{{\mathcal U}}
\newcommand{\sV}{{\mathcal V}}

\newcommand{\Q}{{\mathbb Q}}

\newcommand{\Z}{{\mathbb Z}}

\newcommand{\fm}{{\mathfrak m}}

\newcommand{\fp}{{\mathfrak p}}
\newcommand{\fq}{{\mathfrak q}}

\newcommand{\fk}{{\mathfrak k}}

\newcommand{\fl}{{\mathfrak l}}
\newcommand{\ff}{{\mathfrak f}}

\newcommand{\Irr}{{\rm Irr}}
\newcommand{\Ker}{{\rm Ker}}
\newcommand{\gr}{{\rm gr}}

\newcommand{\CH}{{\rm CH}}

\newcommand{\surj}{\twoheadrightarrow}
\newcommand{\inj}{\hookrightarrow}
\newcommand{\red}{{\rm red}}

\newcommand{\Pic}{{\rm Pic}}
\newcommand{\Div}{{\rm Div}}
\newcommand{\Hom}{{\rm Hom}}

\newcommand{\Spec}{{\rm Spec \,}}

\newcommand{\as}{{\rm as}}
\newcommand{\ms}{{\rm ms}}
\newcommand{\bk}{{\rm bk}}

\newcommand{\ab}{\rm ab}
\newcommand{\nr}{\rm nr}
\newcommand{\Gal}{{\rm Gal}}
\newcommand{\divf}{{\rm div}}

\newcommand{\sHom}{{\mathcal{H}{om}}}

\newcommand{\pfd}{{\operatorname{\mathbf{Pfd}}}} 
\newcommand{\Sch}{{\operatorname{\mathbf{Sch}}}}

\newcommand{\Sets}{{\mathbf{Sets}}}

\renewcommand{\max}{{\operatorname{\rm max}}}

\newcommand{\Ab}{{\mathbf{Ab}}}

\newcommand{\et}{{\text{\'et}}}

\newcommand{\ds}{{/\kern-3pt/}}

\renewcommand{\log}{{\operatorname{log}}}

\newcommand{\Br}{{\operatorname{Br}}}

\newcommand{\tr}{{\operatorname{tr}}}

\newcommand{\un}{\underline}
\newcommand{\ov}{\overline}

\renewcommand{\dim}{\text{\rm dim}}

\newcommand{\tuborg}{\left\{\begin{array}{ll}}
\newcommand{\sluttuborg}{\end{array}\right.}

\newcommand{\nis}{{\rm nis}}

\newcommand{\ord}{{\rm ord}}

\newcommand{\reg}{{\rm reg}}
\newcommand{\local}{{\rm loc}}

\newcommand{\Sw}{{\rm sw}}
\newcommand{\Ar}{{\rm ar}}
\newcommand{\adiv}{{\rm adiv}}

\newcommand{\wt}{\widetilde}
\newcommand{\wh}{\widehat}
\newcommand{\cont}{{\rm cont}}

\newcommand{\Aut}{{\rm Aut}}
\newcommand{\coker}{{\rm Coker}}
\newcommand{\Fil}{{\rm fil}}

\newcounter{elno}

\newcounter{elno-abc}   

\newcounter{elno-abc-prime}

\begin{document}
\title{Reciprocity for Kato-Saito idele class group with modulus}
\author{Rahul Gupta, Amalendu Krishna}
\address{Fakult\"at f\"ur Mathematik, Universit\"at Regensburg, 
93040, Regensburg, Germany.}
\email{Rahul.Gupta@mathematik.uni-regensburg.de}
\address{Department of Mathematics, Indian Institute of Science, Bangalore, 560012, India.}
\email{amalenduk@iisc.ac.in}



\maketitle

\begin{quote}\emph{Abstract.}  
We introduce an {\'e}tale fundamental group with modulus and construct a
reciprocity homomorphism from the Kato-Saito idele class group with
modulus to this fundamental group. This is the $K$-theoretic
analogue of the reciprocity for the cycle-theoretic
idele class group with modulus due to Kerz-Saito. It plays a central role in 
showing the isomorphism between the two idele class groups and in proving
Bloch's formula for the Chow group of 0-cycles with modulus.
%It also provides a new interpretation of the already known {\'e}tale fundamental group with modulus due to Deligne and Laumon.
\end{quote}

\setcounter{tocdepth}{1}

\tableofcontents

\section{Introduction}\label{sec:Intro}
\subsection{Background}\label{sec:Background}
The aim of the class field theory in the geometric case is the description of 
the abelian {\'e}tale coverings of a finite type scheme over a finite field
in terms of some arithmetic or geometric invariants (such as the Chow groups of 
0-cycles) intrinsically associated to the scheme.
For smooth projective schemes, this was achieved by Kato and Saito 
\cite{Kato-Saito-1} who showed that for such a scheme 
$X$, there is an injective reciprocity map
$\CH_0(X) \to \pi^{\ab}_1(X)$ which has dense image.
If $X$ is smooth but not projective over a finite field, the 
class field theory is a much more intricate problem.  

Let $X$ be an integral normal projective scheme of dimension
$d \ge 1$ over a finite field $k$ and let $U \subset X$ be a regular dense 
open subscheme.
In order to establish the class field theory for $U$, Kato and Saito 
\cite{Kato-Saito-2} used a $K$-theoretic
invariant $C(U)$ which is the limit of the 
Nisnevich cohomology groups $H^d_\nis(X, \sK^M_{d, (X,D)})$ as $D$ runs through
all closed subschemes of $X$ which are supported on $X \setminus U$.  
The main result of \cite{Kato-Saito-2} is that there is a reciprocity
homomorphism $\rho \colon C(U) \to \pi^{\ab}_1(U)$ which is an isomorphism upon the
profinite completion of $C(U)$. In particular, $\rho$ induces an isomorphism between
the degree zero subgroups $C(U)_0$ and $\pi^{\ab}_1(U)_0$ (see sections
~\ref{sec:Notns} and ~\ref{sec:ICGC} for definitions of degree zero subgroups).
The cohomology $H^d_\nis(X, \sK^M_{d, (X,D)})$ has been known as the
`Kato-Saito idele class group with modulus' of the pair $(X,D)$.
 
Unfortunately, $C(U)$ and the Kato-Saito idele class group
are difficult objects to understand and compute as
they are  made up of cohomology of $K$-theory sheaves.
Until now, there was no version of the fundamental group with modulus which
could be directly related to the Kato-Saito idele class group.

Few years ago,
Kerz and Saito \cite{Kerz-Saito-2} made a new breakthrough
in the class field theory for $U$. They
discovered the notion of Chow group with
modulus $\CH_0(X|D)$ and constructed a reciprocity homomorphism
from this group to a quotient $\pi^{\ab}_1(X,D)$ of $\pi^{\ab}_1(U)$ which 
classifies abelian covers of $U$ with ramification bounded by $D$ in the
sense of Deligne and Laumon \cite{Laumon}. The main result of 
\cite{Kerz-Saito-2} is that there is an injective reciprocity homomorphism
$\rho'_{X|D} \colon \CH_0(X|D) \to \pi^{\ab}_1(X,D)$ with dense image
if ${\rm Char}(k) \neq 2$. A much simplified proof 
was recently found by Binda-Krishna-Saito \cite{BKS} without any
restriction on ${\rm Char}(k)$.
Taking the limit over $D$, this provides a cycle-theoretic (analogous
to \cite{Kato-Saito-1}) solution to the class field theory for $U$.

Now, one of the main questions in the ramified class field theory that 
remains to be settled is whether the solutions 
provided by Kato-Saito \cite{Kato-Saito-2} and Kerz-Saito \cite{Kerz-Saito-2}
are equivalent when $X$ is regular.
It is worth noting here that the Kato-Saito idele class group
is a much more subtle invariant 
than the cycle-theoretic idele class group because its study involves 
understanding the ramification theory of Galois 
extensions of discrete valuation fields with imperfect residue   
fields.

In order to answer the above question,
one needs to address the following precise question.

\begin{ques}\label{ques:Main-0}
Given an effective Cartier divisor $D \subset X$, is there an isomorphism
\[
cyc_{X|D} \colon \CH_0(X|D) \xrightarrow{\cong} H^d_\nis(X, \sK^M_{d, (X,D)})
\]
such that upon taking the limit over $D$ supported away from $U$, one has 
that $\varprojlim \ \rho'_{X|D} = \rho \circ (\varprojlim \ cyc_{X|D})$?
\end{ques}

A positive answer to the above question will not only make the
class field theories of Kato-Saito and Kerz-Saito equivalent, 
it will simultaneously establish
Bloch's formula for the Chow group of 0-cycles with modulus. The latter is
an essential
tool for connecting cycles with modulus to algebraic $K$-theory. 
The present work is the first of the two papers (see \cite{Gupta-Krishna-CFT})
which together positively settle this question.

\subsection{Main result}\label{sec:Result}
The main obstacle in solving Question~\ref{ques:Main-0} is that
unlike $\rho'$, the reciprocity map $\rho$ does not a priori
descend to a reciprocity map $\rho_{X|D} \colon  H^d_\nis(X, \sK^M_{d, (X,D)}) \to
\pi^{\ab}_1(X,D)$ for a fixed divisor $D$.
In order to get around this problem, we introduce a new {\'e}tale 
fundamental group with modulus called $\pi^\adiv_1(X,D)$ in this paper. 
The major contribution of this paper is to show that $\rho$ descends to a reciprocity homomorphism
\[
\rho_{X|D} \colon  H^d_\nis(X, \sK^M_{d, (X,D)}) \to
\pi^{\adiv}_1(X,D).
\]

In \cite{Gupta-Krishna-CFT}, we shall show that $\rho_{X|D}$ is injective with
dense image. We shall also show that there is an
isomorphism $\pi^{\adiv}_1(X,D) \cong \pi^{\ab}_1(X,D)$ which is compatible
with $\rho_{X|D}$ and $\rho'_{X|D}$. This solves
Question~\ref{ques:Main-0}. This also provides a new
interpretation of the already known group $\pi^{\ab}_1(X,D)$, and is
expected to play a key role in ramification theory.

\vskip .3cm

In precise terms, the main result of this paper is the following.

\begin{thm}\label{thm:Main-1}
Let $X$ be an integral normal scheme which is finite
type over a finite field and has dimension $d \ge 1$. 
Let $D \subset X$ be a closed subscheme
of pure codimension one such that $U= X \setminus D$ is regular.
Then
%the reciprocity map $\rho\colon C(U) \to \pi_1^{\ab}(U)$ 
%of \cite{Kato-Saito-2} induces a continuous  homomorphism 
 there exists  a continuous reciprocity homomorphism 
\[
\rho_{X|D} \colon H^d_\nis(X, \sK^M_{d, (X,D)})
\to \pi^{\adiv}_1(X,D)
\]
whose image is dense.
Moreover, upon taking the limit over $D$ supported on $X\setminus U$, we have 
 $\varprojlim \ \rho_{X|D} = \rho \colon C(U) \to \pi_1^{\ab}(U)$. 
\end{thm}

\subsection{Comparison of \thmref{thm:Main-1} with class field
  theory of Kerz-Zhao}\label{sec:KZ}
In \cite{Kerz-Zhao},
Kerz and Zhao consider a variant of the Kato-Saito idele class group with modulus
(the source of $\rho_{X|D}$ in \thmref{thm:Main-1}) when $X$ is smooth and
the support of $D$ is a strict normal crossing divisor. They denote this
variant by $H^d_{\nis}(X, \sK^M_{d, X|D})$. Their main result is that
there is a canonical isomorphism ${H^d_{\nis}(X, \sK^M_{d, X|D})}/m \xrightarrow{\cong}
{H^d_{\et}(X, \sK^M_{d, X|D})}/m$ (see \cite[Corollary~3.3.4]{Kerz-Zhao}).
  Kerz and Zhao define their {\'e}tale fundamental group with modulus
  with finite coefficients ${C(X,D)}/m$ as ${H^d_{\et}(X, \sK^M_{d, X|D})}/m$
  so that one gets ${H^d_{\nis}(X, \sK^M_{d, X|D})}/m \cong {C(X,D)}/m$
  (see \cite[Definition~3.3.7, Corollary~3.3.8]{Kerz-Zhao}). This isomorphism
  appears similar to \thmref{thm:Main-1}, but it is actually different
  as we explain below.

First, we should note that the sheaf $\sK^M_{d, X|D}$ is different from
    $\sK^M_{d, (X,D)}$, is usually much simpler to deal with, and is useful
    when $X$ is smooth and the support of $D$ is a strict normal crossing divisor.
    Hence, the class group
    with modulus used by Kerz-Zhao is in general different from the Kato-Saito idele
    class group. Second, Kerz and Zhao define their
    {\'e}tale fundamental group with modulus as a sheaf cohomology and they do not
    explain if it describes finite {\'e}tale coverings of $X \setminus D$ with
    certain ramifications along $D$. On the other hand, the group $\pi^{\adiv}_1(X,D)$
    used in \thmref{thm:Main-1} is tailor-made to describe such coverings.
    Third, Kerz and Zhao work with finite coefficients unlike \thmref{thm:Main-1}, which
    is an integral statement. Finally, Kerz and Zhao work under a strong assumption
    that $X$ is smooth and the support of $D$ is a strict normal crossing divisor.
    \thmref{thm:Main-1} makes none of these assumptions.

\subsection{Organization of the paper}\label{sec:Outline}
We give a brief outline of this paper.
Since the cohomology group $H^d_\nis(X, \sK^M_{d, (X,D)})$ is very difficult
to understand, we use a recent result of Kerz \cite{Kerz11} in which he
defines a more explicit idele class group $C(X,D)$ and shows that it
coincides with $H^d_\nis(X, \sK^M_{d, (X,D)})$ under the hypothesis of
\thmref{thm:Main-1}. That one could use $C(X,D)$ instead of
the Kato-Saito idele class group is a key observation in our construction.
Kato's theory of reciprocity for higher local fields and the technique of Parshin chains
allow one us to define a reciprocity map for $C(X,D)$ under a weaker hypothesis than
in \thmref{thm:Main-1}. Its construction goes into the
following steps.

In \S~\ref{sec:PChain*}, we recall the theory of Parshin chains and their
Milnor $K$-theory. In the next two sections, we recall the definition of
$C(X,D)$ from \cite{Kerz11} and prove several properties such as 
functoriality and continuity. In \S~\ref{sec:Rec}, we construct the
reciprocity for the limit group $C_{U/X}$ using Kato's reciprocity for
higher local fields. In \S~\ref{sec:EFG-div}, we introduce our
new {\'e}tale fundamental group with modulus called $\pi^\adiv_1(X,D)$ and
prove several properties. In \S~\ref{sec:R-fixed-D}, we prove our main
result using the ramification theories of Abbes-Saito \cite{Abbes-Saito},
Kato \cite{Kato89} and Matsuda \cite{Matsuda}. The required results from
ramification theory are proven in \S~\ref{sec:Filt*}.

\vskip .3cm

\subsection{Notation}\label{sec:Notns}
In this paper, $k$ will be the base field of characteristic
$p \ge 0$. For most parts, this
will be a finite field in which case we shall let $q = p^s$ be the
order of $k$ for some integer $s \ge 1$. 
%We shall let $\ov{k}$ denote a fixed separable closure of $k$.
A {\sl scheme} will usually
mean a separated and essentially of finite type $k$-scheme.
We shall denote the category of such schemes by $\Sch_k$.
%The product $X \times_{\Spec(k)} Y$ in $\Sch_k$ will be written as
%$X \times Y$. 
If $k \subset k'$ is an extension of fields and
$X \in \Sch_k$, we shall let $X_{k'}$ denote $X \times \Spec(k')$. 
%We shall write $X_{\ov{k}}$ as $\ov{X}$.
For a subscheme $D \subset X$, we shall let $|D|$ denote the set of points
lying on $D$.

For a morphism $f \colon X' \to X$ of schemes and
$D \subset X$ a subscheme, we shall write $D \times_X X'$ as $f^*(D)$.
For a point $x \in X$,
we shall let $\fm_x$ denote  the maximal ideal of $\sO_{X,x}$ and
$k(x)$ the residue field of $\sO_{X,x}$. We shall let $\sO^h_{X,x}$ (resp.
$\sO^{sh}_{X,x}$) denote the henselization (resp. strict henselization)
of $\sO_{X,x}$. We shall let $\ov{\{x\}}$ denote
the closure of $\{x\}$ with its integral subscheme structure.  
%We shall let $d_x$ denote the dimension of $\ov{\{x\}}$.
%We shall let $k(X)$ denote the total ring of quotients of $X$.

If $X$ is a Noetherian scheme, we shall consider sheaves on $X$
with respect to its Nisnevich topology unless mentioned
otherwise. In particular, all sheaf cohomology groups will be considered
with respect to the Nisnevich topology. We shall say that $X$ is
local (resp. semilocal) if it is the spectrum of a local (resp. semilocal)
ring. We shall let $X_{(q)}$ (resp. $X^{(q)}$) denote the set of points on 
$X$ of dimension (resp. codimension) $q$.

For a field $K$ of characteristic $p > 0$ 
with a discrete valuation $\lambda$, we shall let
$K_\lambda$ denote the fraction field of the henselization of the
ring of integers associated to $\lambda$. 
For a field $K$, we shall let  $\ov{K}$ denote a fixed separable closure of 
$K$ and we shall assume all separable algebraic extensions of $K$ to be 
inside $\ov{K}$.

For a Galois extension (possibly infinite) $K \subset L$ of fields,
we shall let $G(L/K)$ denote the Galois group of $L$ over $K$.
We shall let $G_K$ denote the absolute Galois group of $K$.
All Galois groups will be considered as topological abelian groups 
with their profinite topology.
For a Noetherian scheme $X$, its abelianized {\'e}tale fundamental
group will be denoted by $\pi^{\ab}_1(X)$. We shall consider
$\pi^{\ab}_1(X)$ as a topological abelian group with its profinite
topology. If $X \in \Sch_k$, we shall let $\pi^{\ab}_1(X)_0$
be the kernel of the natural map $\pi^{\ab}_1(X) \to \pi^{\ab}_1(k)
\cong G_k$ with the subspace topology. 
%This map is surjective if $X$ is geometrically connected
%over $k$. 

A ring will always mean a unital commutative ring. If $R$ is a reduced ring,
we shall let $R_n$ denote the normalization of $R$. 
We shall let $Q(R)$ denote the total ring of quotients of $R$.
If $A$ is an abelian group, we shall let $A\{p'\}$ denote the
subgroup of elements of $A$ which are torsion of order prime to $p$.
We let $A\{p\}$ denote the $p$-primary torsion subgroup of $A$.

\section{Parshin chains and their Milnor $K$-theory}
\label{sec:PChain*}
As mentioned previously, the key idea for constructing the reciprocity
map for the Kato-Saito idele class group is to replace it with $C(X,D)$. 
We shall therefore work with the latter group in most parts of this paper. 
In this section, we recall Parshin chains and their Milnor $K$-groups
from \cite[\S~1]{Kato-Saito-2}. We also recall the topologies
on these Milnor $K$-groups which will play a very useful role in this
paper. Ideally, this section should have been shorter as it is
mostly recollection from \cite{Kato-Saito-2}. But we noticed that apart from the
new lemmas that we prove in this section, many of the
terms of this section are technical for a general reader
and are very heavily used throughout this paper as
well as in our follow-up papers \cite{Gupta-Krishna-CFT}, \cite{Gupta-Krishna-Duality}
and \cite{GKR}.
We therefore decided to write these down in one section instead of
referring the reader to \cite{Kato-Saito-2} too many times, for sake of clarity of our
presentation, without claiming any originality.

\subsection{The dimension function}\label{sec:Dim}
Let $X$ be an excellent integral scheme of Krull dimension $d(X)$.
A dimension function on $X$ is an assignment
$d_X \colon X \to \Z_+$ which satisfies the catenary condition:
if $x \in \ov{\{y\}}$ is a point of codimension one 
($\dim(\sO_{\ov{\{y\}}, x}) = 1$), then $d_X(y) = d_X(x) +1$.
We let $d_m(X) = {\rm min}\{d_X(x)| x \in X\}$
and $X_i = \{x \in X| d_X(x) = i + d_m(X)\}$ for $i \ge 0$.
Note that an excellent scheme is necessarily Noetherian and
a Noetherian scheme with a dimension function is
necessarily catenary.

We let $\sD_\dim$ be the category of pairs $(X,d_X)$, where $X$ is an
excellent integral scheme endowed with a dimension function $d_X$.
A morphism in $\sD_\dim$ is a quasi-finite morphism of
schemes $f \colon X' \to X$ such that $d_{X'} = d_X \circ f$. In this case,
we say that $d_{X'}$ is a pull-back of (or induced by) $d_X$.
%If $(X,d_X)$ is given, we shall simply write $d_X$ as $d$.

\begin{exm}\label{exm:Dim-func}
We give an example which we shall use very often in this paper.
Let $X$ be an excellent integral scheme. We let
$d_X(x)$ be the Krull-dimension of $\ov{\{x\}}$.
Since $X$ is catenary (because it is excellent),
it follows that $d_X$ is a dimension function on $X$.
We call it the canonical dimension function on $X$.
Here, $d_m(X) = 0$. 

We now let $P \in X$ be a point of codimension $r \ge 0$
and let $X_P = \Spec(\sO^h_{X,P})$. We let $x_P$ denote the closed
point of $X_P$. 
%and $k(P)$ its residue field. 
%We let $X'_P =X_P \setminus \{x_P\}$. 
We have the canonical morphisms of schemes
$X_P \setminus \{x_P\} \inj X_P \to X$. We endow $X_P$ and 
$X_P \setminus \{x_P\}$ with the
dimension function induced by that of $X$.
Then we have
\[
d_m(X_P) = d(X) - r = d_X(P), \ d(X_P) = r, 
\]
\[\
d_m(X_P \setminus \{x_P\}) = d_m(X_P) + 1, \ d(X_P \setminus \{x_P\}) = 
d(X_P) - 1 = r-1.
\]
\end{exm}

\subsection{Parshin chains}\label{sec:Pchain}
Let $X$ be an excellent integral scheme with a dimension
function $d_X$. 
Let $C \subset X$ be a reduced nowhere dense 
closed subscheme of $X$ and let $U = 
X \setminus C$. Let $K$ denote the function field of $X$.
We recall the following from \cite[\S~1.6]{Kato-Saito-2}
(or \cite[Definition~3.1]{Kerz11}).

\begin{defn}\label{defn:Chain}
\begin{enumerate}
\item
An $s$-chain on $X$ is a sequence of points $P = (p_0, \ldots, p_s)$
on $X$ such that $\ov{\{p_i\}} \subset \ov{\{p_{i+1}\}}$ for
$0 \le i \le  s-1$. 
\item
We say that $P$ is a Parshin chain ($P$-chain) if $d_X({p_i}) = i + d_m(X)$ 
(equivalently, $p_i \in X_i$) for $0 \le i \le s$.
\item
A Parshin chain on the pair $(U \subset X)$ is a Parshin chain
$P = (p_0, \ldots , p_s)$ on $X$ with $s \ge 0$ 
such that $p_i \in C$ for $0 \le i < s$ and $p_s \in U$. 
\item
The dimension of a chain $P = (p_0, \ldots , p_s)$ on $X$ is $d_X(p_s)$,
and it is denoted by $d_X(P)$.
\item
We say that $P = (p_0, \ldots , p_s)$ is a maximal Parshin chain if
it is a Parshin chain such that
$p_s = \eta$, where $\eta$ is the generic point of $X$.
\item 
If $P$ is a maximal Parshin chain and $D \subset X$
is an effective Weil divisor with support $C$,
we let $m_P(D)$ (often written only as $\fm_P$) denote the multiplicity of
$\ov{\{p_{s-1}\}}$ in $D$. Note that $m_P(D) = 0$ if $p_{s-1} \notin C$.
\end{enumerate}
\end{defn}

If $P = (p_0, \ldots , p_s)$ is an $s$-chain, we shall write
$P' = (p_0, \ldots , p_{s-1})$. This convention will be followed
throughout the paper.

Note that a Parshin $0$-chain on the pair $(U \subset X)$ is a point 
$p_0 \in U$ such that $d_X(p_0) = d_m(X)$.
In particular, no such chain exists if $X$ is local.
%and $U \neq X$
If $P_1 = (p_0, \ldots , p_s)$ is an $s$-chain and
$P_2 = (p_0, \ldots , p_s, p_{s+1}, \ldots , p_r)$ is an $r$-chain, 
then we say that $P_2$ dominates $P_1$ and write $P_1 \le P_2$. 

Let $P = (p_0, \ldots , p_s)$ be an $s$-chain on $X$. We define the 
ring $\sO^h_{X,P}$ as follows.
We let $\sO^h_{X,p_0}$ be the henselization of $\sO_{X,p_0}$.
Suppose now that $s \ge 1$ and we have defined $\sO^h_{X,P'}$. 
We let $R = \sO^h_{X,P'}$ and let $T$ be the set of all 
prime ideals of $R$ lying over $p_s$ under the canonical map
$\sO_{X,p_{s-1}} \to R$. We let $\sO^h_{X,P} = {\underset{\fq \in T}\prod}
R^h_{\fq}$. It is easily seen by induction on $s$ that $T$ is a finite set
(e.g., see \cite[Lemme~18.6.9.2]{EGA-IV}).
We let $J^h_{X,P}$ be the Jacobson radical of $\sO^h_{X,P}$ and
$k(P) = {\sO^h_{X,P}}/{J^h_{X,P}}$. 
Then $k(P)$ is the product of the residue fields of $R^h_{\fq}$
(equivalently, of $R_{\fq}$), where 
$\fq$ runs through $T$. Note that if $X$ is regular at $p_i$ for some
$i \ge 0$, then $\sO^h_{X,P}$ is a regular semilocal ring. 
We call $k(P)$ the residue ring of the chain $P$. We write
$X_P = \Spec(\sO^h_{X,P})$.

\begin{lem}\label{lem:Inv-ambient}
Let $P = (p_0, \ldots , p_s)$ be an $s$-chain on $X$ and let 
$Y = \ov{\{p_s\}}$. 
Then $P$ is an $s$-chain on $Y$ and there are canonical isomorphisms
${\sO^h_{X,P}}/{J^h_{X,P}} \xrightarrow{\cong} \sO^h_{Y,P}
\xrightarrow{\cong} k(P)$.
\end{lem}
\begin{proof}
There is nothing to prove if $s = 0$, so we can assume that $s \ge 1$.
We write $P = (P',p_s)$. By replacing $p_s$ by one of the 
points in the inverse image of $p_s$ under the canonical map $X_{P'} \to X$,
we see that it suffices to prove the lemma when $X$ is a Henselian
local scheme, $P = (p_0, p_1)$, where $p_0$ is the closed point of $X$
and $Y = \ov{\{p_1\}}$. 

If $Y = X$, there is nothing to prove, so we can assume that
$Y$ has codimension at least one in $X$. In this case, the lemma
is equivalent to the statement that if $A$ is a Henselian local ring
and $\fp \subset A$ is a prime ideal of height at least one with
$B = {A}/{\fp}$, then there is an isomorphism
${(A_{\fp})^h}/{J^h_\fp} \cong Q(B)$, where $J^h_\fp$ is the maximal 
ideal of $(A_{\fp})^h$.
However, this is clear using two simple observations.
First, the quotient map
$A \surj B$ induces an isomorphism  ${A_\fp}/{J_\fp} \xrightarrow{\cong} Q(B)$,
where $J_\fp$ is the maximal ideal of $A_\fp$. Second,
the henselization map $A_{\fp} \to (A_\fp)^h$ induces an isomorphism
on the quotients ${A_\fp}/{J_\fp} \xrightarrow{\cong} 
{(A_{\fp})^h}/{J^h_\fp}$.
\end{proof} 

Suppose that $P = (p_0, \ldots , p_s)$ is a maximal Parshin chain
and $R = \sO^h_{X, P'}$.
Then $R$ is a finite product of one-dimensional reduced local rings
such that the canonical injection $R \inj R_n$ between reduced rings
induces a one-to-one correspondence between the minimal primes.
Furthermore, $\sO^h_{X,P}$ is the total quotient ring of $R$
(and of $R_n$). In particular, the map $\sO^h_{X,P} \surj k(P)$ is a 
bijection and every factor of $k(P)$ has a unique discrete valuation
whose ring of integers is the normalization of the quotient of $R$ by
a unique minimal prime.
If $\sI_D \subset \sO_X$ is the ideal sheaf of a closed subscheme
defining an effective Weil divisor $D \subset X$ with support $C$,
then it is easily seen that $\sI_DR_n = (\sI_{\ov{\{p_{s-1}\}}}R_n)^{m_P}$.

\subsection{Residue fields of Parshin chains via valuation rings}
\label{sec:valuation}
In order to study the reciprocity for the $K$-theoretic idele class groups, we
need to express the residue fields of Parshin chains in terms of
the fraction fields of 
valuations rings so that Kato's theory of reciprocity maps for
higher local fields can be utilized.
We recall the machinery of doing this from \cite[\S~3.1]{Kato-Saito-2}

Let $V$ be a valuation ring with quotient field $K$ whose value
group is $\Z^r$ with the lexicographic order. Let $v \colon 
K^{\times} \to \Z^r$ be the valuation. In this case, one says
(in the terminology of \cite{Kato-Saito-2}) that
$V$ is an $r$-DV.
If we let $V_i = \{0\} \cup v^{-1}((\Z_{\ge 0})^{r-i} \times \Z^i)$ for 
$1 \le i \le r$,
then each $V_i$ is a valuation ring with quotient field $K$ and there
are inclusions $V = V_0 \subset \cdots \subset V_{r-1} \subset V_r = K$. 
There is a saturated chain of prime ideals $\fp_0 \supsetneq \fp_1 
\supsetneq \cdots \supsetneq \fp_r = 0$ in $V$ such that $V_i = V_{\fp_i}$.
If $k(\fp_i)$ denotes the residue field of $V_i$, then
the image of the composite map $V_i \inj V_{i+1} \surj k(\fp_{i+1})$
is a discrete valuation ring with quotient field $k(\fp_{i+1})$ and
residue field $k(\fp_i)$ for each $0 \le i \le r-1$.
In particular, $V_{r-1}$ is a discrete valuation ring with quotient field $K$.

Suppose $V$ is an $r$-DV. Assume $r \ge 1$ and let $V'$ be the image of $V$
in $k(\fp_{r-1})$. Then $V'$ is an $(r-1)$-DV with quotient field $k(\fp_{r-1})$.
Let $E$ be the quotient field of $V'^h$ so that there is an
extension of fields $k(\fp_{r-1}) \inj E$.
Let $\wt{V}_{r-1}$ be the inductive limit of a system of {\'e}tale 
local rings over $V_{r-1}$ such that $\wt{V}_{r-1}$ is a Henselian 
discrete valuation ring with residue field $E$.
It is then known (e.g., see \cite{Ribenboim}) 
that $V^h$ is the inverse image of $V'^h$ under
the quotient $\wt{V}_{r-1} \surj E$. 
Furthermore, the following elementary lemma says that $V^h$ is
a Henselian $r$-DV in $Q(V^h)$ such that $\wt{V}_{r-1} = (V^h)_{\fp_{r-1}}$.

\begin{lem}\label{lem:r-dv}
Let 
\[
\xymatrix@C.8pc{
A \ar@{->>}[r]^-{\alpha} \ar@{^{(}->}[d] & B \ar@{^{(}->}[d] \\
A' \ar@{->>}[r]^-{\alpha'} & B'}
\]
be a Cartesian square of Noetherian local integral domains such that
$B' = Q(B)$, and let $\fp = \ker(\alpha)$. Then $A' = A_{\fp}$.
\end{lem}

We now keep the notations of \S~\ref{sec:Pchain}. In particular,
$X$ is an excellent integral scheme with function field $K$ and dimension function $d_X$.
Let $P = (p_0, \ldots , p_d)$ be a maximal Parshin chain on $X$. 
We let $V \subset K$ be a $d$-DV.
We say that $V$ dominates $P$ if the local ring $V_i$ dominates the local 
ring $\sO_{X,p_i}$
for each $0 \le i \le d$. Let $\sV(P)$ be the set of $d$-DV's in $K$
which dominate $P$.
Let $\nu \colon X_n \to X$ be the normalization map and let 
$\nu^{-1}(P)$ be the set of all Parshin chains
$P_1= (p'_0, \ldots , p'_d)$ on $X_n$ such that $\nu(P_1) = P$.
Equivalently, $\nu(p'_i) = p_i$ for each $0 \le i \le d$.

\begin{lem}\label{lem:P-valuation}
There are canonical isomorphisms
\begin{equation}\label{eqn:P-valuation-0}
k(P) \xrightarrow{\cong} {\underset{V \in \sV(P)}\prod} Q(V^h)
\xleftarrow{\cong} {\underset{P_1 \in \nu^{-1}(P)}\prod} k(P_1).
\end{equation}
\end{lem}
\begin{proof}
The first isomorphism is \cite[Proposition~3.3]{Kato-Saito-2}.
The second isomorphism follows by
combining the first isomorphism for $X$ and $X_n$ and using the 
fact that a valuation ring is integrally closed, and hence
$\sV(P) =  {\underset{P_1 \in \nu^{-1}(P)}\bigcup} \sV(P_1)$.
\end{proof}

\subsection{Milnor $K$-theory of $s$-chains}\label{sec:MPchain}
For a commutative ring $R$, we let $K^M_*(R)$ be the quotient of the
tensor algebra $T^*(R^{\times})$ by the two-sided ideal generated by
%$a \otimes (1-a)$ with $a, 1-a \in R^{\times}$. 
%Defn of Milnor $K$-groups has been changed to the one in 
%Kato-Saito. 
the elements $a_1 \otimes \cdots \otimes a_n$ such that $a_i +a_j =1$
for some $i \neq j$. 
We let the image
of $a_1 \otimes \cdots \otimes a_n$ in $K^M_n(R)$ be denoted by
$\{a_1, \ldots , a_n\}$.
If $(U_i)_{1 \le i \le n}$ are subgroups of $R^\times$, we shall
let $\{U_1, \ldots , U_n\}$ be the subgroup of $K^M_n(R)$ generated by
$\{a_1, \ldots , a_n\}$ such that $a_i \in U_i$. 

If $I \subset R$ is an ideal, we let $K^M_*(R, I)$ be the
kernel of the canonical map $K^M_*(R) \to K^M_*(R/I)$.
If $R$ is either a local ring (see \cite[Lemma~1.3.1]{Kato-Saito-2})
or a semilocal ring which contains a field of cardinality at least three 
(see \cite[Lemma~3.2]{Gupta-Krishna-2}), then the canonical maps
$K^M_n(R) \to K^M_n(R/I)$ and 
 \begin{equation} \label{eqn:MPchain*-1}
 \bigoplus_{1\leq i \leq n} \left( K^M_{i-1}(R) \otimes K^M_1(R, I) 
\otimes K^M_{n-i}(R) \right) \to K^M_n(R, I)
 \end{equation}
 are surjective. We shall
denote by $K^M_*(R|I)$ the image of $K^M_*(R, I)$ under the 
canonical map $K^M_*(R) \to K^M_*(F)$, where $F$ is the ring of total 
fractions of $R$.

If $R$ is a one-dimensional 
excellent regular semilocal integral domain with
Jacobson radical $\fm$ and quotient field $F$, we shall write 
$K^M_n(R|\fm^r)$ also as $K^M_n(F,r)$. 
If $n =1$, we shall also use the alternative notation $U_r(F)$ for
$K^M_1(F,r)$.
If $X$ is an excellent integral scheme with a dimension
function $d_X$ and if $P = (p_0, \ldots , p_s)$ is an $s$-chain on $X$,
we define the Milnor $K$-theory of $P$ to be the Milnor $K$-theory
of its residue field $k(P)$.

\subsection{Topologies on Milnor $K$-groups}\label{sec:Top}
Let us first assume that $R$ is a one-dimensional 
excellent regular semilocal integral domain with
Jacobson radical $\fm$ and quotient field $F$.
We consider $K^M_n(F)$ a 
topological group whose topology is the smallest
such that the subgroups $\{K^M_n(F,r)\}_{r \ge 0}$ form a basis of
open neighborhoods of 0. This is called the canonical
topology of $K^M_n(F)$.

Let $R$ now be a one-dimensional semilocal excellent integral domain with 
fraction field $F$ and normalization $R'$
(we have digressed from our notation $R_n$ for normalization to
avoid a conflict).  
Let $\fm$ and $\fm'$ denote the Jacobson radicals of 
$R$ and $R'$, respectively. The $\fm$-adic topology of $R$
induces a topology on $R^{\times}$ generated by subgroups
$U_r(R) := (1 + \fm^r)$ with $r \ge 0$, where we let $1 + \fm^0 = R^{\times}$.
The same holds for $(R')^{\times}$ too.
Let us consider following topologies $\tau_i$ on $K^M_n(F)$ for $1 \le i \le 3$.
We let $\tau_1$ be the smallest topology 
for which the fundamental
system of neighborhoods of 0 is given by the subgroups
$\{U_1, \ldots , U_n\} \subset K^M_n(F)$, where each $U_i$
is open in $R^\times$. 

We let $\tau_2$ be the smallest topology 
for which the fundamental
system of neighborhoods of 0 is given by the subgroups
$\{1 + I^r, R^\times , \ldots , R^{\times}\} \subset K^M_n(F)$, where $I \subset 
\fm$ is an ideal such that $\sqrt{I} = \fm, \ r \ge 0$ and 
$1 + I^0 = R^{\times}$.
If $R$ is either local (see \cite[Lemma~1.3.1]{Kato-Saito-2})
or contains a field of cardinality at least three 
(see \cite[Lemma~3.2]{Gupta-Krishna-2}), then
$\{1 + I^r, R^\times , \ldots , R^{\times}\}= K^M_n(R| I^r) \subset K^M_n(F)$.
We let $\tau_3$ be the smallest topology 
for which the fundamental
system of neighborhoods of 0 is given by the subgroups
$\{U_r(R), F^\times , \ldots , F^{\times}\} \subset K^M_n(F)$, where $r \ge 0$.

\begin{lem}\label{lem:Kato-Kerz}
Each of the topologies $\tau_1, \tau_2$ and $\tau_3$ coincides with
the canonical topology of $K^M_n(F)$.
\end{lem}
\begin{proof}
The assertion that $\tau_1$ and $\tau_3$ coincide
follows from \cite[Proposition~2 (2)]{Kato86}.
It is clear that every fundamental neighborhood of 0 in the
canonical topology (see \S~\ref{sec:MPchain}) 
contains a fundamental neighborhood in the
$\tau_1$-topology. We prove the reverse inclusion.

Let $G$ be the subgroup $\{U_1, \ldots , U_n\} \subset K^M_n(F)$, 
where each $U_i$ is open in $R^\times$. It follows from 
\cite[Proposition~2 (2)]{Kato86} that
$G$ contains an open subgroup of the form $G_1 = 
\{U_1, F^{\times}, \ldots, F^{\times}\}$, where $U_1 \subset R^{\times}$ is open.
Since the conductor ideal of the normalization map $R \to R'$
contains an $\fm'$-primary ideal, it follows that  $U_1$ contains an open 
subgroup $U_2$ of
the form $U_r(R')$ for some $r \ge 1$. We are already done if $n = 1$.
Otherwise, as $\{U_2, F^\times, \ldots ,
F^{\times}\}$ contains $\{U_2, (R')^{\times}, \ldots ,
(R')^{\times}\} = K^M_n(F,r)$, we conclude that $G$ contains $K^M_n(F,r)$ for
some $r \ge 1$. In particular, $G$ is open in the canonical topology of
$K^M_n(F)$. This proves the agreement between $\tau_1$ and the canonical
topology.

Note now that $\tau_2$ does not change if we replace $I$ by $\fm$.
It is now clear that $\tau_3 \subset \tau_2 \subset \tau_1$ and we have already
 shown above that
$\tau_3 = \tau_1$. The lemma is proven.
\end{proof}

\subsection{The direct sum topology}\label{sec:DST}
Recall from \cite[\S~6]{Kerz-Schmidt-JNT} that if 
$G = {\underset{i \in I}\bigoplus} G_i$ is the direct sum of 
abelian groups $\{G_i\}_{i \in I}$ such that each $G_i$ is a topological abelian group,
then the direct sum topology on $G$
is the inductive limit of the product topologies on finite direct sums.
It is easy to see that if each $G_i$ is generated by a family of open subgroups
$\{G_{i, \lambda_i}\}_{\lambda_i \in I_i}$
(in which case one says that $G_i$ has the subgroup topology), then
the family of direct sums ${\underset{i \in I}\bigoplus} G_{i, \lambda_i}$,
where $\lambda_i$ runs through $I_i$ for each $i \in I$,
generates the direct sum topology of $G$.
The direct sum topology has another property which we shall use often.

\begin{lem}\label{lem:DST-univ}
Let $G$ be the direct sum of abelian groups 
$\{G_{i}\}_{i \in I}$. Assume that each $G_i$ is
a topological abelian group with subgroup topology.
Then $G$, with its direct sum topology, is  
a coproduct in the category of topological abelian groups with 
the subgroup topology.
\end{lem}
\begin{proof}
Let 
$\{f_i \colon G_i \to H\}$ be a family of continuous homomorphisms
to a topological abelian group $H$. 
Then the lemma asserts that there is a unique
continuous homomorphism $f \colon G \to H$ such that $f|_{G_i} = f_i$
for all $i \in I$. 

Now, it is clear that a homomorphism of groups $f$ exists with 
the above uniqueness property as a group homomorphism. 
We need to show that $f$ is continuous.
Since $G$ has the inductive limit topology, it suffices to prove
the case when $I$ is a finite set.
We can therefore let $I = \{1, \ldots , n\}$.

We now let $H' \subset H$ be an open subgroup.
By \lemref{lem:Cont-tp-grp}, it suffices to show that $f^{-1}(H')$ is open. 
We let $B_i = f^{-1}_i(H')$ and let $B = {\underset{1 \le i \le n}\prod} B_i
\subset G$. It is clear that $B \subset G$ is open.
We are left with showing that $f(B) \subset H'$.
But this is clear because $f^{-1}(H')$ is a subgroup of $G$ and
$f^{-1}_i(H') = f^{-1}(H') \cap G_i$ via the inclusions $G_i \subset G$.
\end{proof}

The following useful result may be well known, but we write down a proof
as we could not find a suitable reference.

\begin{lem}\label{lem:Cont-tp-grp}
Let $f \colon G \to H$ be a group homomorphism between two
topological abelian groups with subgroup topologies. Then $f$ is continuous
if and only if the inverse image of every generating open subgroup via $f$ is
open. 
\end{lem}
\begin{proof}
Since $G$ and $H$ have subgroup topologies,
$f$ will be continuous if the inverse image of every coset of
a generating open subgroup is open. So we let $a \in H$ be an arbitrary 
element and $B \subset H$ a generating open subgroup.
If $f^{-1}(a + B)$ is empty, we are done. Otherwise, we let 
$a' \in G$ be such that $f(a') = a + b$ for some $b \in B$.
By our hypothesis, there is an open subgroup $B' \subset f^{-1}(B)$.
If $b' \in B'$, then $f(a' + b') = a + (b + b') \in a + B$.
This implies that $f(a' + B') \subset a + B$. Since $a' + B'$ is open
in $G$, the lemma follows.
\end{proof}

\section{The idele class group {\`a} la Kerz}
\label{sec:Kerz}
In the construction of the reciprocity homomorphism, 
the idele class groups defined by Kato-Saito
\cite{Kato-Saito-2} and Kerz \cite{Kerz11} will play the key roles.
The goal of this section is to recall the idele class group defined
by Kerz and study some properties which we shall use in the proof of
the main results.

Let $X$ be an excellent integral scheme with a dimension
function $d_X$. We let $U \subset X$ be an open immersion 
whose complement $C$ is a nowhere dense
closed subset  with the reduced subscheme structure.
We let $\sP_{U/X}$ denote the set of Parshin chains on the pair $(U \subset X)$
and let $\sP^{\max}_{U/X}$ be the subset of $\sP_{U/X}$ consisting of
maximal Parshin chains. 
If $P$ is a Parshin chain on $X$, then we shall consider 
$K^M_{d_X(P)}(k(P))$ as a topological abelian group with its canonical
topology if $P$ is maximal. Otherwise, we shall consider
$K^M_{d_X(P)}(k(P))$ as a topological abelian group with its discrete topology.

\subsection{The idele groups}\label{sec:Idele*}
We let
\begin{equation}\label{eqn:Idele-U-0}
I_{U/X} = {\underset{P \in \sP_{U/X}}\bigoplus} K^M_{d_X(P)}(k(P)),
\end{equation}
and call it the ($K$-theoretic) idele group.
We consider $I_{U/X}$ as a topological group with its direct sum topology.
It is clear that the topology of $I_{U/X}$ is generated by the open 
subgroups ${\underset{P \in \sP^{\max}_{U/X}}\bigoplus} 
K^M_{d_X(P)}(k(P), m_P(D(P)))$,
where $D(P)$ runs through the set of all effective Weil divisors whose 
supports coincide with $C$.
Note that $K^M_{d_X(P)}(k(P), m_P(D(P)))$ is defined
since each factor $k(P)_i$ of $k(P)$ is the quotient field of the unique 
discrete valuation ring $R$ whose maximal ideal is the radical of the
extension of $\sI_{\ov{\{p_{s-1}\}}}$ in $R \subset k(P)_i$ under the canonical 
map $\sO_{X} \to k(P)_i$.
It also follows that $I_{U/X}$ is discrete if $\dim(C) \le \dim(X) -2$. 

If $D \subset X$ is a closed subscheme with support $C$, we let 
\begin{equation}\label{eqn:Idele-D}
I(X,D) = {\rm Coker}\left({\underset{P \in \sP^{\max}_{U/X}}\bigoplus}
K^M_{d_X(P)}(\sO^h_{X, P'}, I_D) \to I_{U/X}\right),
\end{equation}
where $I_D$ is the extension of the ideal sheaf $\sI_D \subset \sO_X$
to $\sO^h_{X,P'}$ and the map on the right is induced by the composition
$K^M_{d_X(P)}(\sO^h_{X, P'}, I_D) \to K^M_{d_X(P)}(k(P)) \to I_{U/X}$ for
$P \in \sP^{\max}_{U/X}$.
We call this the ($K$-theoretic) idele group of $X$ with modulus $D$.
We consider $I(X,D)$ a topological group with its quotient topology.

\subsection{The $C$-topology on idele groups}
\label{sec:D-top}
Let $U \subset X$ be as above. 
The $C$-topology on $I_{U/X}$ is the smallest topology generated by the
open subgroups each of which is of the form
$I_{U/X}(D) := {\underset{P \in \sP^{\max}_{U/X}}\bigoplus}
K^M_{d_X(P)}(\sO^h_{X, P'}| I_D)$, where $D \subset X$ is
a closed subscheme with support $C$. 
The direct sum topology on $I_{U/X}$ is finer than the $C$-topology.
To see this, it suffices to show that if $D \subset X$ is a closed
subscheme with support $C$ and if $C$ has codimension one in $X$, then
for every maximal Parshin chain $P$ on $X$, the
subgroup $K^M_{d_X(P)}(\sO^h_{X, P'}| I_D)$ is open in $K^M_{d_X(P)}(k(P))$.
But this follows from \lemref{lem:Kato-Kerz}.
We shall use the comparison between the direct sum and $C$-topology
in some of our proofs.

It is easy to see that the direct sum topology and the $C$-topology
on $I_{U/X}$ coincide if there are only finitely many maximal Parshin chains on
$(U \subset X)$. However, one can check that this is not the case 
otherwise. In this paper, we shall consider $I_{U/X}$ as a topological
group with its direct sum topology unless we specify something else.

\subsection{The idele class groups}\label{sec:ICG}
Let $U \subset X$ be as above with dimension function $d_X$.
A $Q$-chain on $(U \subset X)$ is a chain 
$(p_0, \ldots , p_{s-2}, p_s)$ with $s \ge 1$ such that the following hold:
\begin{enumerate}
\item
$p_s \in U$;
\item 
$p_i \in C$ for $0 \le i \le s-2$;
\item
$P = (p_0, \ldots , p_{s-2}, p_{s-1}, p_s)$ is a Parshin chain on $X$
for some $p_{s-1} \in X$;
\end{enumerate}
We let $\sQ_{U/X}$ denote the set of all $Q$-chains on $(U \subset X)$.

Let $Q = (p_0, \ldots , p_{s-2}, p_s)$ be a $Q$-chain on $(U \subset X)$
and let $P = (p_0, \ldots , p_{s-2}, p_{s-1}, p_s)$  be a Parshin chain
on $X$. If $p_{s-1} \in C$, then we have the canonical extension
of rings $k(Q) \inj k(P)$ and this yields a map
$\partial_{Q \to P} \colon K^M_{d_X(Q)}(k(Q)) \to K^M_{d_X(Q)}(k(P)) = 
K^M_{d_X(P)}(k(P))$.
If $p_{s-1} \in U$, then $P' = (p_0, \ldots , p_{s-1})$ is a
Parshin chain on $(U \subset X)$. We therefore have the 
residue map $\partial_{Q \to P} \colon K^M_{d_X(Q)}(k(Q)) \to
K^M_{d_X(P')}(k(P'))$. We thus get a map
$\partial_{Q} \colon K^M_{d_X(Q)}(k(Q)) \to I_{U/X}$.
We let $\partial_{U/X}$ denote the sum of the maps $\partial_Q$ where
$Q$ runs through  $\sQ_{U/X}$.
We let 
\begin{equation}\label{eqn:IC-0}
C_{U/X} = {\rm Coker}\left({\underset{Q \in \sQ_{U/X}}\bigoplus}
K^M_{d_X(Q)}(k(Q)) \xrightarrow{\partial_{U/X}} I_{U/X}\right)
\end{equation}
and call it the ($K$-theoretic) 
{\sl idele class group} of the pair $(U \subset X)$.
This is a topological group with its quotient topology. 
We shall see in \propref{prop:Degree} that $C_{X/X}$ coincides with
$\CH^F_0(X)$ (Chow group in the sense of \cite{Fulton})
when $d_X$ is the canonical dimension function.

If $D \subset X$ is a closed subscheme with support $C$, we let
\begin{equation}\label{eqn:IC-1}
C(X,D) = {\rm Coker}\left({\underset{Q \in \sQ_{U/X}}\bigoplus}
K^M_{d_X(Q)}(k(Q)) \to I(X,D)\right)
\end{equation}
and call it the {\sl idele class group} of $X$ with modulus $D$.
This is a topological group with its quotient topology.

If $d_X$ is the canonical dimension function on $X$, then 
for a closed point $x \in U$, we let $[x]$ denote the image of $1$ 
under the homomorphism $K_0^M(k(x)) \to C(X, D)$, and call it the
`cycle class' of $x$. 
It is clear that for every $D \subset X$ as above, there is a
canonical surjection $C_{U/X} \surj C(X,D)$ which is continuous.
The following result says more about the topology of $C(X,D)$.

\begin{lem}\label{lem:Disc}
The topologies of $I(X,D)$ and $C(X,D)$ are discrete. 
If the codimension of $C$ is
more than one in $X$, then $I_{U/X}$ and $C_{U/X}$ are also discrete. 
\end{lem}
\begin{proof}
If the codimension of $C$ is more than one in $X$, then there are 
no maximal Parshin chains on $(U \subset X)$
whose Milnor $K$-groups define the topology of all idele and idele
class groups. Hence, all these topological groups are discrete.

Suppose now that $C$ has an irreducible component of codimension
one in $X$. It is clear from its  definition that $I(X,D)$
is discrete with respect to the $C$-topology. The discreteness of 
$I(X,D)$ now follows because the direct sum topology of $I_{U/X}$ is finer 
than its $C$-topology (see \S~\ref{sec:D-top}).
Since $C(X,D)$ is a quotient of $I(X,D)$ with its quotient topology,
it must also be discrete.
\end{proof}

The $Q$-chains have the following property that will be often useful
to us.
Let $A$ be a Henselian equidimensional reduced local ring of Krull 
dimension $d$.
Let $\fp$ be a minimal prime of $A$ and let $X = \Spec(A/{\fp})$.
Note that $X$ is a Henselian local integral scheme of Krull dimension $d$.
Let $p_0$ and $p_d$ denote the closed point and the generic point of
$X$, respectively. Let $Q = (p_0, \ldots , p_{d-2}, p_d)$ be a  
$Q$-chain on $X$.
Let $\nu \colon X_n \to X$ be the normalization map.
Then $X_n$ is also a local integral scheme with the unique point $p'_0$ 
lying over $p_0$. Let $\sV(Q)$ be the set of $Q$-chains on $X_n$ lying over 
$Q$.

Let $B(Q)$ denote the set of all Parshin chains 
$P = (p_0, \ldots, p_{d-2}, p_{d-1}, p_d)$
on $X$. For every Parshin chain $P$ on $X$, let $\nu^{-1}(P)$ denote the
set of all Parshin chains $P_1$ on $X_n$ such that $\nu(P_1) = P$.
%Note that any $P_1 \in \nu^{-1}(P)$ must be of the form $(p'_0, p'_1, p_2)$
%where $\nu(p'_1) = p_1$.
For any Parshin chain $P \in B(Q)$, we have the inclusion of
fields $\iota_P \colon k(Q) \inj k(P)$.
Let $\iota^Q \colon k(Q) \to {\underset{P \in B(Q)}\prod} k(P)$ be the
induced map to the product.
\begin{lem}\label{lem:Q-chain-nor}
There is a commutative diagram
\begin{equation}\label{eqn:Q-chain-nor-0}
\xymatrix@C.8pc{
k(Q) \ar[d]_-{\cong} \ar[r]^-{\iota^Q} & {\underset{P \in B(Q)}\prod} k(P)
\ar[d]^-{\cong} \\
{\underset{Q' \in \sV(Q)}\prod} k(Q') 
\ar[r]^-{\iota^{Q'}} & {\underset{P \in B(Q)}\prod} 
{\underset{P_1 \in \sV(P)}\prod} k(P_1)}
\end{equation} 
such that the vertical arrows are isomorphisms.
\end{lem}
\begin{proof}
The commutativity of the diagram is clear and the isomorphism of the
right vertical arrow follows from ~\lemref{lem:P-valuation}, whose
proof also shows that the left vertical arrow is an isomorphism too.
\end{proof}

\subsection{Functoriality of the idele class group}
\label{sec:Functorial}
Let $X$ be an excellent integral scheme of dimension $d = d(X)$ endowed with a 
dimension function $d_X$. Let $C \subset X$ be a reduced nowhere dense
closed subscheme of $X$ with $U = X \setminus C$.
Let $f \colon X' \to X$ be a finite morphism from an integral scheme
such that the support of $f^{-1}(C)$ is contained in a 
nowhere dense reduced closed subscheme $C' \subset X'$ with
complement $U'$. We endow $X'$ with the dimension function induced by $d_X$. 
If these properties are satisfied, we shall say that
$f \colon (X', U') \to (X,U)$ is an admissible morphism.

\begin{prop}\label{prop:Funtorial*}
Assume that $X$ is essentially of finite type over a field.
Then the map $f$ induces a push-forward map
\[
f _* \colon C_{{U'}/{X'}} \to C_{U/X}.
\]
\end{prop}
\begin{proof}
Let $P  = (p_0, \ldots , p_s)$ be a Parshin chain on $(U' \subset X')$.
Since $f$ is finite, $f(P) = (f(p_0), \ldots , f(p_s))$ 
is a Parshin chain on $X$. But it may not be a Parshin chain
on $(U \subset X)$ unless $C' = f^{-1}(C)$. We let $s' \le s$ be the
largest integer such that $P'' := (f(p_0), \ldots , f(p_{s'}))$ is a 
Parshin chain on $(U \subset X)$.  We let $P_1 = (p_0, \ldots , p_{s'})$.

Since $P_1 \le P$, there is a composition of residue homomorphisms
$\partial^{P_1 \le P} \colon K^M_{d_{X'}(P)}(k(P)) \to K^M_{d_{X'}(P_1)}(k(P_1))$.
We also have the finite map between products of fields
$k(P'') \to k(P_1)$ and this yields the norm map
$N^{P_1 \to P''} \colon K^M_{d_{X'}(P_1)}(k(P_1)) \to K^M_{d_X(P'')}(k(P''))$.
We let $f^{P \to P''} = N^{P_1 \to P''} \circ \partial^{P_1 \le P}
\colon  K^M_{d_{X'}(P)}(k(P)) \to K^M_{d_X(P'')}(k(P''))$.
Composing this with the canonical inclusion $K^M_{d_X(P'')}(k(P'')) \inj
I_{U/X}$, we get a map
$f^{P}_* \colon  K^M_{d_{X'}(P)}(k(P)) \to I_{U/X}$.
Taking the direct sum of these maps over all Parshin chains
on $(U' \subset X')$, we get our push-forward map
\[
f_* \colon I_{U'/X'} \to I_{U/X}.
\]

We now show that $f_*$ preserves the relations that define the
idele class groups.
If $Q = (p_0, \ldots , p_{s-2}, p_s)$ is a $Q$-chain on $(U' \subset X')$,
then $\partial (K^M_{d_{X'}(Q)}(k(Q)))$ is the direct sum of
images under the henselization along the discrete valuations rings
coming from the irreducible components of $C'$ (if the latter is a
divisor) 
and the residue maps along points on $U'$ (see \S~\ref{sec:ICG}).
Since $f_*$ is defined by composing residue and the norm maps,
the assertion that it preserves the relations will follow if we 
show that the residue maps commute with henselization and the
norm maps commute with henselization as well as the residue maps.

Since a residue map between the Milnor $K$-groups of fields commutes
with field extensions (see \cite{Kato86}), 
it is clear that it commutes with henselization.
Similarly, since the norm map between Milnor $K$-groups commutes with
field extensions by the standard properties of norm (see \cite{Bass-Tate}), 
it commutes with henselization along discrete valuations.
The only nontrivial thing therefore is to check that the norm
and the residue maps between Milnor $K$-groups of fields commute.
But this follows from \lemref{lem:Norm-residue}.   
\end{proof}

\begin{lem}\label{lem:Norm-residue}
Let $A$ be discrete valuation ring with quotient field $E$ and 
residue field $k$. Let $F$ be a finite field extension of $E$ and $B$ the 
integral closure of $A$ in $F$. Let $k_1, \ldots , k_r$ be the
residue fields of $F$. Assume that $A$ is essentially of finite type
over a field. Then the diagram
\begin{equation}\label{eqn:Norm-residue-0}
\xymatrix@C.8pc{
K^M_n(F) \ar[d]_-{N_{F/E}} \ar[r]^-{\sum_i \partial_{F/{k_i}}} &
{\underset{1 \le i \le r}\bigoplus} K^M_{n-1}(k_i) \ar[d]^-{\sum_i N_{{k_i}/k}} \\
K^M_n(E) \ar[r]^-{\partial_{E/k}} & K^M_{n-1}(k)}
\end{equation}
is commutative.
\end{lem}
\begin{proof}
This result was proven by Kato \cite[\S~1, Lemma~16]{Kato80} when
$E$ is a complete discrete valuation field and $F/E$ is a normal
extension whose degree is a prime number.
We shall prove the general case using Bloch's higher Chow groups.

For an equidimensional scheme $X$ which is essentially of finite type over a 
field, let $\CH^p(X,q)$ denote Bloch's higher Chow 
groups of $X$ \cite{Bloch86}. These groups are equipped with proper
push-forward, flat pull-back and localization sequence.
By Totaro \cite{Totaro} and Nesterenko-Suslin \cite{Nesterenko-Suslin}, 
there are isomorphisms $\theta_K \colon K^M_n(K) \to \CH^n(K, n)$ for a field 
$K$. These isomorphisms satisfy some nice properties.

We now consider the diagram
\begin{equation}\label{eqn:Norm-residue-1}
\xymatrix@C.8pc{
K^M_n(F) \ar[rr]^-{\sum_i \partial_{F/{k_i}}} \ar[dd]_-{N_{F/E}} 
\ar[dr]_-{\theta_F} & & {\underset{1 \le i \le r}\bigoplus} K^M_{n-1}(k_i) 
\ar[dd]^>>>>>>>{\sum_i N_{{k_i}/k}} \ar[dr]^-{\sum_i \theta_{k_i}} & & \\
& \CH^n(F,n) \ar[rr]^-{\sum_i \partial'_{F/{k_i}}} 
\ar[dd]_>>>>>>>>>>>>>>>>{N'_{F/E}} & &
{\underset{1 \le i \le r}\bigoplus} \CH^{n-1}(k_i, n-1) 
\ar[dd]^-{\sum_i N'_{{k_i}/k}} & \\
K^M_n(E) \ar[rr]_>>>>>>>{\partial_{E/k}}  \ar[dr]_-{\theta_E} & & 
K^M_{n-1}(k) \ar[dr]^-{\theta_k} & & \\
& \CH^n(E,n) \ar[rr]^-{\partial'_{E/k}} & & \CH^{n-1}(k, n-1),}
\end{equation}
where $N'$ denotes the push-forward and
$\partial'$ the boundary map in the localization sequence
for higher Chow groups.
Since all diagonal arrows are isomorphisms, it suffices to show that
all faces (except possibly the back face) of ~\eqref{eqn:Norm-residue-1}
commute to prove the lemma.

The front face commutes because of the known property of 
higher Chow groups that the localization sequences
are compatible with the push-forward maps \cite{Bloch86}.
The two side faces commute by \cite[Lemma~4.7]{Nesterenko-Suslin}.
To show that the top and the bottom faces commute, it suffices to
show that if $R$ is an equicharacteristic discrete valuation ring with 
quotient field $K$ and residue field $\mathfrak{f}$, then
the diagram
\begin{equation}\label{eqn:Norm-residue-2}
\xymatrix@C.8pc{
K^M_n(K) \ar[r]^-{\partial_{K/\mathfrak{f}}} \ar[d]_-{\theta_K} & 
K^M_{n-1}(\mathfrak{f}) \ar[d]^-{\theta_{\mathfrak{f}}} \\
\CH^n(K,n) \ar[r]^-{\partial'_{K/{\mathfrak{f}}}} & \CH^{n-1}(\mathfrak{f},n-1)}
\end{equation} 
is commutative.

Now, it is well known and elementary to see 
(using the Steinberg relations) that $K^M_*(K)$ is generated by $K^M_1(K)$
 as a 
 $K^M_*(R)$-module. Furthermore, $\partial_{K/\mathfrak{f}}$ is 
$K^M_*(R)$-linear (see \cite[Chapter~1, Proposition~4.5]{Bass-Tate}).
Since the localization sequence 
\[
\CH^n(R,n) \to \CH^n(K,n) \xrightarrow{\partial'_{K/{\mathfrak{f}}}} 
\CH^{n-1}(\mathfrak{f},n-1) \to 0
\]
is $\CH^*(R,*)$-linear, it is also $K^M_*(R)$-linear via $\theta_R$.  
Here, the action of $\CH^*(R,*)$ on $\CH^{\bullet-1}(\mathfrak{f},\bullet-1)$ is given 
by the pull-back ring homomorphism $\CH^*(R,*) \to \CH^*(\mathfrak{f},*)$.
Since the map $\theta_{(-)}$ is multiplicative, 
it follows that all arrows in 
~\eqref{eqn:Norm-residue-2} are $K^M_*(R)$-linear. 
It therefore suffices to prove the commutativity of
~\eqref{eqn:Norm-residue-2} for $n =1$. But in this case, both 
$\partial_{K/\mathfrak{f}}$ and $\partial'_{K/\mathfrak{f}}$ are simply the 
valuation map of $K$ corresponding to $R$.
The proof of the lemma is now complete.
\end{proof}

We now study the functoriality of the idele class groups with modulus.
Existence of push-forward map for the idele class group with modulus is a
very nontrivial problem in general. We shall prove this here only for 
finite dominant maps. The existence of push-forward in general will appear
in \cite{Gupta-Krishna-CFT}.

\begin{prop}\label{prop:PF-mod}
Assume that $X$ is essentially of finite type over a field and is regular
at the generic points of $C$.
Let $f \colon (X', U') \to (X,U)$ be a dominant admissible morphism.
Let $D \subset X$ (resp. $D' \subset X'$) be a closed subscheme with support 
$C$ (resp. $C'$) such that $f^*(D) \subset D'$.
Then $f$ induces a push-forward map
\[
f _* \colon C(X',D') \to C(X,D).
\]
\end{prop}
\begin{proof}
In view of \propref{prop:Funtorial*}, we only need to show that
$f_* \colon I_{U'/X'} \to I_{U/X}$ preserves the relative Milnor $K$-groups of
maximal Parshin chains. Since $f$ is finite and dominant, we have
$\dim(X') = \dim(X) = d$.
Let $P_1 = (p'_0, \ldots , p'_d)$ be a maximal Parshin
chain on $(U' \subset X')$ and let $P_2 = f(P_1) = (p_0, \ldots , p_d)$,
where $p_i = f(p'_i)$. Let $0 \le s \le d$ be the largest integer such that 
$(p_0, \ldots , p_s)$ is a Parshin chain on $(U,X)$.

Consider first the case that $s \le d-1$. In this case, the map
$K^M_d(\sO^h_{X',P'_1}, I_{D'}) \to I_{U/X}$ has a factorization
\[
K^M_d(\sO^h_{X',P'_1}, I_{D'}) \to K^M_d(k(P_1)) \xrightarrow{\partial}
K^M_{d-1}(k(P'_1)) \xrightarrow{f_*} I_{U/X}.
\]
Hence, it will suffice to show that if $A$ is a one-dimensional
local integral domain with quotient field $F$ and residue field $\ff$, then
$K^M_{n}(A) \to K^M_n(F) \xrightarrow{\partial} K^M_{n-1}(\ff)$ is a complex
for all $n \ge 1$.
However, by the construction of the residue map $\partial$, it suffices to
prove this assertion when $A$ is normal in which case this is well known.

We now consider the remaining case $s = d$. In this case, we need to show
that the image of $K^M_d(\sO^h_{X',P'_1}| I_{D'})$ under the norm
map $N \colon K^M_d(k(P_1)) \to K^M_d(k(P_2))$ lies in
$K^M_d(\sO^h_{X,P'_2}|I_{D})$.
Since these groups depend only on 
the codimension one subschemes, we can assume that $D$ (resp. $D'$)
has codimension one in $X$ (resp. $X'$) and $p_{d-1}$ (resp. $p'_{d-1}$)
is a generic point of $D$ (resp. $D'$). We can therefore replace
$D'$ by $f^*(D)$ and  assume that $I_{D'} =
I_D \sO^h_{X',P'_1}$ via the (finite) map $\sO^h_{X,P'_2} \to \sO^h_{X',P'_1}$.
The desired assertion now follows from \cite[Lemma~4.3]{Kato-Saito-2}
because $\sO^h_{X,P'_2}$ is a Henselian discrete valuation ring as 
a consequence of our hypothesis. 
\end{proof}

\begin{remk}\label{remk:PF-mod-0}
The proof of \propref{prop:PF-mod} shows that the hypothesis that $X$ is 
regular at the generic points of $C$ is not required
if $f$ is the identity map of $X$. Hence, for any $D \subset D'$, we have the
canonical map $C(X,D') \to C(X,D)$.
\end{remk}

\subsection{Nice vs. regular schemes}\label{sec:nice-reg}
In \cite[Definition~2.2]{Kato-Saito-2}, a scheme is called nice if
each of its local rings $\sO_{X,x}$ has the property that it is
ind-{\'e}tale over a ring which is smooth over a field or a Dedekind domain.  
In \cite{Kato-Saito-2}, the authors need a scheme to be nice in
a dense open subset in many of their results. But they remark after 
their Corollary~2.4 that the only reason for assuming this 
is to be able to apply their Theorem~2.3. However,
it was shown by Kerz in \cite{Kerz09} and \cite{Kerz10} that
the latter theorem holds if $X$ is a regular Noetherian scheme over a field.  
It follows that all results of \cite{Kato-Saito-2} which are valid for
nice schemes are also valid for regular schemes over a field.
In particular, we have the following result due to Kato
\cite[Theorem~2]{Kato86} and Kerz \cite[Proposition~10]{Kerz10}, 
where this property is
mostly used. We shall use this result very often and sometimes
without referring to it.

\begin{lem}\label{lem:Nice-Kerz}
Let $X \in \Sch_k$ for some field $k$ and let $x \in X^{(q)}_\reg$. Then the
Gersten complex for the Milnor $K$-theory gives rise to a canonical isomorphism
\[
\nu_x \colon H^q_x(X, \sK^M_{n,X}) \xrightarrow{\cong} K^M_{n-q}(k(x))
\]
for any integer $n \ge 0$.
\end{lem}

\subsection{Relation between $C(X,D)$ and 
Kato-Saito idele class group}\label{sec:KKS}
Let $X$ be an excellent integral scheme of Krull dimension $d$ 
endowed with its canonical dimension
function $d_X$. 
For an integer $n \ge 0$, we let $\sK^M_{n, X}$ be the Nisnevich sheaf on $X$ 
whose stalk at a point
$x \in X$ is the Milnor $K$-group $K^M_n(\sO^h_{X,x})$.
We define the sheaf of relative Milnor $K$-groups $\sK^M_{n, (X,D)}$
in a similar way using the relative Milnor $K$-theory of rings and
ideals defined in \S~\ref{sec:MPchain}.
For any Nisnevich sheaf $\sF$ on $X$ and subscheme $Y \subset X$ 
(not necessarily closed), let $H^*_Y(X, \sF)$ denote the cohomology with 
support in $Y$. If $Y \subset X$  is closed, we have the
canonical map $H^*_Y(X, \sF) \to H^*(X, \sF)$, which we shall refer to
as `the forget support map'.

For a closed subscheme $D \subset X$, we let
\begin{equation}\label{eqn:KS-idele-**}
C_{KS}(X, D) = H^{d}_\nis(X, \sK^M_{d, (X, D)})
\end{equation}
and call it the `Kato-Saito idele class group'.
We shall write $C_{KS}(X, \emptyset)$ as $C_{KS}(X)$.

Let $U \subset X$ be a dense open subscheme with
complement $C$. For any point $x \in X$, we let
$X_x = \Spec(\sO^h_{X,x})$. For any subscheme $Y \subset X$,
we let $Y_x$ be the pull-back of $Y$ under the canonical map $X_x \to X$.
For an $s$-chain $P$ on $X$, we let $Y_P$ be 
the pull-back of $Y$ to $X_P$, where 
recall that $X_P = \Spec(\sO^h_{X,P})$. The dimension function on
$X_x$ and $X_P$ will always be the one induced by $d_X$.
We let $U_x = X_x \setminus C_x$ and define $U_P$ similarly.
%We let $X'_{x} = X_x \setminus \{x\}$ and define $X'_P$ similarly.
We let $x_P$ denote the set of all closed points of $X_P$. Note that $x_P$ is 
a finite closed subset of $X_P$.

Let $P = (p_0, \ldots , p_s)$ be an $s$-chain on $X$. We let $d_P$ denote
the Krull dimension of $X_P$ so that $d_m(X_P) = d - d_P$. 
We shall also write $d_m(X_P)$ as $d^P_m$. 
We let 
\begin{equation}\label{eqn:KS-idele}
C_{KS}(X_P, D_P) = H^{d_P}_{x_P}(X_P, \sK^M_{d, (X_P, D_P)}).
\end{equation}
%Here we can replace this with regular scheme over a field. 
If $U_P$ is a regular $k$-scheme for some field $k$ and $D_P = \emptyset$, then
it follows from \lemref{lem:Nice-Kerz} that
$C_{KS}(X_P, D_P) \cong K^M_{d-d_P}(k(P))$.

%x can be taken as a regular point of U. 
If $U$ is a $k$-scheme for some field $k$ and
$x \in U$ is a regular closed point, we have the canonical map
$\Z \cong C(X_x, D_x) \cong K^M_0(k(x)) \to C(X,D)$, where the last map
is the canonical one, used in the definition of $C(X,D)$ because
$\{x\}$ is a Parshin chain on $(U \subset X)$.
On the other hand, we also have the forget support map
$\Z \cong C_{KS}(X_x, D_x) \to C_{KS}(X,D)$. Hence, we get a
diagram
\begin{equation}\label{eqn:Global-0}
\xymatrix@C.8pc{
& {\underset{x \in (U_\reg)_0} \bigoplus} \Z \ar[dr] \ar[dl] & \\
C(X,D) \ar@{.>}[rr] & & C_{KS}(X,D).}
\end{equation}

\begin{thm}$($\cite[Theorem~8.2]{Kerz11}$)$\label{thm:Kerz-global}
Assume that $X \in \Sch_k$ for some field $k$ 
and $U$ is regular. Let $D \subset X$ be a closed subscheme 
with support $C$. Then there is an isomorphism 
\[
\psi_{X|D} \colon C(X,D) \xrightarrow{\cong} C_{KS}(X,D)
\]
such that the above diagram commutes.
\end{thm}
\begin{proof}
By the coniveau spectral sequence for Nisnevich cohomology, there is 
an exact sequence
\begin{equation}\label{eqn:Kerz-global-0}
{\underset{y \in X_{(1)}} \bigoplus} H^{d-1}_y(X_y, \sK^M_{d, (X_y, D_y)})
\xrightarrow{\partial_1} 
{\underset{x \in X_{(0)}} \bigoplus} H^{d}_x(X_x, \sK^M_{d, (X_x, D_x)}) \to
C_{KS}(X,D) \to 0.
\end{equation}

On the other hand, it follows essentially by definition and a local
version of the theorem (see \cite[Theorem~8.1]{Kerz11})
that there is an exact sequence
\begin{equation}\label{eqn:Kerz-global-2}
{\underset{y \in U_{(1)}} \bigoplus} H^{d-1}_y(X_y, \sK^M_{d, (X_y, D_y)})
\xrightarrow{\partial_1} 
{\underset{x \in X_{(0)}} \bigoplus} H^{d}_x(X_x, \sK^M_{d, (X_x, D_x)}) \to
C(X,D) \to 0.
\end{equation}
The proof of the theorem is therefore reduced to showing that
the boundary maps in ~\eqref{eqn:Kerz-global-0} and ~\eqref{eqn:Kerz-global-2}
have same image. But this is shown in \cite[Theorem~8.1]{Kerz11}.

There are only couple of remarks to be made here. The first is
that the proof of \cite[Theorem~8.1]{Kerz11} (which is by induction on 
$\dim(X)$) in the dimension one case requires a correction, namely, the correct
diagram to use is
\begin{equation}\label{eqn:Kerz-global-3}
\xymatrix@C.8pc{
K^M_{d^P_m +1}(\sO^h_{X,P}, I_{D_P}) \ar[r] \ar[d]_-{\cong} &
K^M_{d^P_m +1}(k(\eta)) \ar[r] \ar[d]_-{\cong} & C^{\local}(X_P, D_P) \ar@{.>}[d] 
\ar[r] & 0 \\
H^0(X_P, \sK^M_{d^P_m +1, (X_P, D_P)}) \ar[r] & 
H^0(\Spec(\eta), \sK^M_{d^P_m +1, (X_P, D_P)}) \ar[r] & 
C_{KS}(X_P, D_P) \ar[r] & 0,}
\end{equation}
if $P$ is a lifted chain on $X$ with $\dim(X_P) = 1$
and $C^{\local}(X_P, D_P)$ is as defined in \S~\ref{sec:Lchain}.

The second remark is that Kerz works under the set-up where $X$ is projective
over a finite field and $D$ is an effective Cartier divisor, but his proof of 
this theorem uses none of these assumptions.
\end{proof}

\begin{remk}\label{remk:KKS-CE}
If $U$ is not regular, then \thmref{thm:Kerz-global} does not always hold.
For example, take $X$ to the projective nodal curve over a field and
take $D = \emptyset$. Then $C_{KS}(X, D) \cong H^1_\nis(X, \sO^{\times}_X)
\cong \Pic(X) \cong k^{\times} \oplus \Z$. On the other hand, it follows
from \propref{prop:Degree} that $C(X,D) \cong \CH^F_0(X) \cong \Z$.
\end{remk}

\begin{cor}\label{cor:Codim-2}
Assume in \thmref{thm:Kerz-global} that $\dim(C) \le d-2$.
Then $C(X,D) \cong C_{KS}(X)$. In particular, $C(X,D)$ does not depend on $D$.
\end{cor}
\begin{proof}
It is immediate from \thmref{thm:Kerz-global} using the fact that 
the finite push-forward is an exact functor in Nisnevich topology and
the Nisnevich cohomological dimension agrees with the Krull dimension
for a Noetherian scheme.
\end{proof}

\section{Continuity of the push-forward map}
\label{sec:PF-cont}
We shall now prove the continuity of the push-forward map
between the idele class groups.
Let $X$ be an excellent integral scheme of Krull dimension $d$ 
endowed with a dimension function $d_X$. Let $\eta$ denote the generic
point of $X$.
Let $C \subset X$ be a nowhere dense reduced closed subscheme with
complement $U$.

We first consider the easy case.
\begin{lem}\label{lem:Cont-dominant}
Let $f \colon (X',U') \to (X,U)$ be a dominant admissible morphism.
Then the map $f_* \colon C_{{U'}/{X'}} \to C_{U/X}$ is continuous. 
If $U$ is regular, then $f_*$ satisfies the following. 
\begin{enumerate}
\item If $f$ is birational, then $f_*$ is surjective.
\item If $f \colon X' \to X$ is an isomorphism, then $f_*$ is  
a topological quotient map. 
\end{enumerate}
\end{lem}
\begin{proof}
For the first part, it suffices to show that the map
$f_* \colon I_{{U'}/{X'}} \to I_{U/X}$ is continuous. 
By virtue of the direct sum topology of $I_{{U'}/{X'}}$ (see \S~\ref{sec:Idele*}) and 
  \lemref{lem:DST-univ}, it suffices to show that
the map $f_* \colon K^M_{d_{X'}(P)}(k(P)) \to I_{U/X}$ is continuous
for every Parshin chain $P$ on $(U' \subset X')$.
If $P$ is not maximal, then $f^{-1}_*(B)$ is zero for
every proper open subgroup $B$ of $I_{U/X}$. Since the
topology of $I_{{U'}/{X'}}$ restricts to the discrete topology on
$K^M_{d_{X'}(P)}(k(P))$, this case follows.
We shall now assume that $P$ is maximal.

Let $P = (p_0, \ldots , p_d)$, where $d = d_X(\eta) - d_m$.
Let $s \le d$ be the largest integer such that $(f(p_0), \ldots,
f(p_s))$ is a Parshin chain on $(U \subset X)$. We let
$P_1 = (p_0, \ldots , p_s) \le P$. Then $f_*$ is the composition of the 
residue 
map $\partial \colon K^M_d(k(P)) \to K^M_{d_{X'}(P_1)}(k(P_1))$ and the norm map
$N \colon K^M_{d_{X'}(P_1)}(k(P_1)) \to K^M_{d_{X}(P'')}(k(P''))$, where
$P'' = f(P_1)$.

There are two cases to consider. Suppose first that $s = d$ so that
$P'' = f(P)$ is a maximal Parshin chain on $(U \subset X)$.
In this case, $f_*$ is simply the norm map $N \colon
K^M_d(k(P')) \to K^M_d(k(P''))$. Since the inclusion
$K^M_d(k(P'')) \inj I_{U/X}$ is clearly continuous as $P''$ is maximal,
we need to show that $N$ is continuous. But this follows from
\cite[Proposition~4.2]{Kato-Saito-2}.

Suppose now that $s < d$. Then $f_*$ has a factorization
\[
K^M_d(k(P)) \xrightarrow{\partial} K^M_{d_{X'}(P_1)}(k(P_1)) 
\xrightarrow{N} K^M_{d_{X}(P'')}(k(P'')) \inj I_{U/X}.
\]
If we take a fundamental open subgroup $G \subset I_{U/X}$, then its
intersection with $K^M_{d_{X}(P'')}(k(P''))$ is zero because  $P''$ is a 
non-maximal Parshin chain on $(U \subset X)$.
In particular, we have that $f^{-1}_*(G) = 
\partial^{-1}({\rm Ker}(N))$. Hence, it suffices to show that
$\partial^{-1}(H)$ is open in $K^M_d(k(P))$ for every subgroup 
$H \subset K^M_{d_{X'}(P_1)}(k(P_1))$. To show this,
we can argue inductively and reduce to the case that
$\dim(P_1) = d-1$. It is enough to check in this case that
${\rm Ker}(\partial)$ is open in $K^M_d(k(P))$. 
We can assume that $k(P)$ is a field.
In this case, $k(P_1)$ is the residue field of the discrete valuation
ring $A$ which is the normalization of $\sO^h_{X', P_1}$ in $k(P)$.
We are now done because the sequence
\[
K^M_d(A) \to K^M_d(k(P)) \xrightarrow{\partial} K^M_{d-1}(k(P_1)) \to 0
\]
is exact (e.g., see \cite[Proposition~2.7]{Dahlhausen})
and the image of $K^M_d(A)$ is clearly open in $K^M_d(k(P))$
(see \lemref{lem:Kato-Kerz}).

Suppose now that $f$ is birational and $U$ is regular. It suffices to show the
stronger assertion that $f_* \colon I_{U'/X'} \to I_{U/X}$ is surjective.
Since $f$ is finite
and $U$ is regular, it follows that is $f$ an isomorphism over $U$.
We can therefore assume that $U' \subset U$. Suppose now that 
$P = (p_0, \ldots , p_s)$ is a Parshin chain on $(U \subset X)$.
Let $\sP$ be the set of Parshin chains
$P_1= (p'_0, \ldots , p'_s)$ on $X'$ such that $f(P_1) = P$. Since $f$ is 
finite and $p'_s \in U$, 
it follows from \cite[Lemma~3.3.1]{Kato-Saito-2} that $ \prod_{P_1 \in \sP} 
k(P_1) \cong k(P)$ 
and therefore 
the norm map $\oplus_{P_1 \in \sP} K^M_{d_{X'}(P_1)}(k(P_1)) \to 
K^M_{d_{X}(P)}(k(P))$ is identity. 

We let $s' \ge s$ be the smallest integer such that
$P'' = (P_1, p'_{s+1}, \ldots , p'_{s'})$ is a Parshin chain on
$(U', X')$. 
It is then enough to show that the
residue homomorphism $K^M_{d_{X'}(P'')}(k(P'')) \to K^M_{d_{X'}(P_1)}(k(P_1))$
is surjective.
To prove this, one can again argue inductively. Since $p'_s, \ldots , p'_{s'} 
\in U$ and $U$ is regular, one can 
assume that $k(P'')$ is a discrete valuation field with ring of integers $A$ 
and residue field $k(P_1)$.
In this case, one knows that $\partial(\{\pi, u_1, \ldots , u_n\})
= \{\ov{u}_1, \ldots , \ov{u}_n\}$ for $n \ge 1$
if $u_i \in A^{\times}$ for all $i$
and $\pi$ is a uniformizer of $k(P'')$.

Suppose now that $f$ is an isomorphism and $U$ is regular. We show the
stronger assertion that $f_* \colon I_{U'/X'} \to I_{U/X}$ is a topological 
quotient. 
Since we have already shown that $f_*$ is 
continuous and surjective, it suffices to show that it takes
a fundamental open subgroup to an open subgroup.
But this is obvious from the definition of topology on the idele groups
and the fact that $f$ preserves maximality of Parshin chains and 
$f_*$ is identity on the Milnor $K$-theory of maximal Parshin 
chains. This finishes the proof.
\end{proof}

\subsection{Local idele class groups}\label{sec:Lchain}
Let $X$ be an excellent integral scheme endowed with a dimension
function $d_X$. Let $U \subset X$ be a dense open subscheme with
complement $C$. A lifted chain on $X$ is a sequence $P = (p_0, \ldots , p_s)$
such that $p_0 \in X$ and $p_{i+1} \in \Spec(\sO^h_{X, (p_0, \ldots , p_{i})})$
for $i \ge 0$. We let $X_P = \Spec(\sO^h_{X,P})$.
Let $\fm_P$ be the maximal ideal of $\sO^h_{X,P}$ and $k(P)$ its residue field.
We denote the closed point of $X_P$ by $x_P$.
We remark that a lifted chain is different from an $s$-chain. 
If $P$ is a lifted chain, then $X_P$ is always a local (Henselian) scheme
but this may not be the case for an $s$-chain on $X$. In fact, 
for any $s$-chain $P_1$ on $X$, the scheme $X_{P_1}$ is a disjoint union
of the spectra of finitely many lifted chains.

We need to define a local idele class group of a lifted chain.
Let $P = (p_0, \ldots , p_s)$ be a lifted chain on $X$ and let 
$D_P$ (resp. $U_P$) be the pull-back of $D$ (resp. $U$) under the canonical 
map $X_P \to X$. 
We endow $X_P$ with the dimension function induced by $d_X$. 
We let $I(X_P,D_P)$ be exactly as in ~\eqref{eqn:Idele-D}. However, there is one
difference in $Q$-chains on $(U_P \subset X_P)$. 
We say that $Q = (p_0, \ldots , p_{s-2}, p_s)$
is a $Q$-chain on $(U_P \subset X_P)$ if it satisfies conditions (1) - (3) in
the beginning of \S~\ref{sec:ICG}, together with an additional condition
that $p_0$ is the unique closed point of $X_P$.
We let $\sQ^{\local}_{U/X}$ denote the set of these special $Q$-chains.
We define the local idele class group of $(X_P, D_P)$ by
\begin{equation}\label{eqn:L-idele}
C^{\local}(X_P,D_P) = \coker\left({\underset{Q \in \sQ^{\local}_{U/X}}\bigoplus}
K^M_{d_X(Q)}(k(Q)) \xrightarrow{\partial_{U/X}} I(X_P, D_P)\right).
\end{equation}

The following is straightforward and shows that one can ignore
the above new definition in higher dimensions.
But this is not the case in dimension one.

\begin{lem}\label{lem: L-idele-0}  
Assume that $\dim(X_P) \ge 2$.
Let $X'_P = X_P \setminus x_P$ and $D'_P = D_P \setminus x_P$.
Then there is a canonical isomorphism
\[
C^{\local}(X_P,D_P) \xrightarrow{\cong} C(X'_P, D'_P).
\]
\end{lem}

\subsection{Continuity of push-forward in general}
\label{sec:PF-cont-*}
To prove the continuity of more general push-forward maps
between the idele class groups, we need some lemmas.

\begin{lem}\label{lem:Cont-gen-0}
Let $X$ be an excellent integral scheme equipped with a dimension
function $d_X$. Let $U \subset X$ be a dense open subscheme with
complement $C$. Let $P$ be a Parshin chain on $(U \subset X)$
whose codimension in $X$ is at most one. Let  $K^M_{d_X(P)}(k(P))$ be
endowed with its canonical topology.
Then the canonical map $K^M_{d_{X}(P)}(k(P)) \to C_{U/X}$ is continuous
if and only if the composite map 
$K^M_{d_{X}(P)}(k(P)) \to C_{U/X} \surj C(X,D)$ is
continuous for all closed subschemes $D$ with support $C$.
\end{lem}
\begin{proof} One of the implications is obvious. For the nontrivial 
implication, we first note that 
if $P$ is maximal, then the inclusion map
$K^M_{d_{X}(P)}(k(P)) \to I_{U/X}$ is already continuous, as is easily seen.
We can therefore assume that $P$ has codimension one in $X$. 
We let  $P = (p_0, \ldots , p_s)$ and $d' = d_X(P) $.
Let $\eta$ be the generic point of $X$ and $d = d_X(\eta)$.
Let $G = {\underset{P_1 \in \sP^{\max}_{U/X}}\bigoplus} 
K^M_{d}(k(P_1), m_{P_1}(D_{P_1}))$
be a basic open subgroup of $I_{U/X}$.
We need to show that there is an open subgroup $H \subset K^M_{d'}(k(P))$
with the canonical topology such that 
$H \subset G + {\rm Image}(\partial_{U/X})$ inside $I_{U/X}$, where
$\partial_{U/X}$ is described in ~\eqref{eqn:IC-0}.

To show the last assertion, we first observe that there is a unique 
$Q = (p_0, \ldots , p_{s-1}, \eta) \in \sQ_{U/X}$ such that 
$\partial_{U/X}(K^M_{d}(k(Q)))$ can nontrivially intersect
$K^M_{d'}(k(P))$. 
We next observe that there are only finitely many 
Parshin chains on $(U\subset X)$ of the form 
$\wt{P}_i=  (p_0, \ldots , p_{s-1}, q_i, \eta)$
because $C$ has only finitely 
many irreducible components. We let $S$
be the set of these finitely many Parshin chains.
It then  follows that the map
$\partial_{U/X} \colon K^M_{d}(k(Q)) \to I_{U/X}$ factors through 
\begin{equation}\label{eqn:Cont-gen-0-0}
K^M_{d}(k(Q)) \xrightarrow{\partial_{U/X}} 
\left({\underset{P_1 \in S}\oplus} K^M_{d}(k(P_1))\right) \bigoplus
K^M_{d'}(k(P)) \bigoplus \left({\underset{P'' \in \sP''_{U/X}}\oplus}
K^M_{d_X(P'')}(k(P''))\right) \inj I_{U/X},
\end{equation}
where $ \sP''_{U/X}$ is the set of all Parshin chains of codimension 
at least one, excluding $P$.

Since the composite map $K^M_{d'}(k(P)) \inj  I_{U/X} \to C(X,D)$ is
continuous by our hypothesis, and since $C(X,D)$ is discrete by
\lemref{lem:Disc}, it follows from ~\eqref{eqn:Cont-gen-0-0} 
that there is an open 
subgroup $H \subset K^M_{d'}(k(P))$ such that 
$H \subset \left({\underset{P_1 \in S}\oplus} 
K^M_{d}(\sO^h_{X, {P}'_1}| I_D) \right) +  {\rm Image}(\partial_{U/X})$.

Since 
\[
{\underset{P_1 \in S}\bigoplus} 
K^M_{d}(\sO^h_{X, P'_1}| I_D) \subset
\left({\underset{P_1 \in S}\oplus} 
K^M_{d}(\sO^h_{X, P'_1}| I_D) \right) \bigoplus
\left({\underset{P_1 \in \sP^{\max}_{U/X} \setminus S}\oplus} 
K^M_{d}(k(P_1), m_{P_1}(D_{P_1}))\right),
\]
and since for each $1\leq i \leq n$, we have 
\[
K^M_{d}(\sO^h_{X, \wt{P}'_i}| I_D) = 
K^M_{d}(\sO^h_{X, \wt{P}'_i}| I_{m_{\wt{P}_i} D_{\wt{P}_i}}) \subset
 K^M_{d}(k(\wt{P}_i), m_{\wt{P}_i}(D_{\wt{P}_i})),
 \]
we conclude that $H \subset G +  {\rm Image}(\partial_{U/X})$. This proves the 
desired assertion.
\end{proof}

For the rest of \S~\ref{sec:PF-cont-*}, we assume the following.
Let $k$ be a field and $X \in \Sch_k$ an integral scheme equipped 
with a dimension function $d_X$.
Let $U \subset X$ be a dense open subscheme. 
Let $C$ be the complement of $U$ with reduced
closed subscheme structure.

We let $y \in U$ be a point contained in the regular locus of $X$. 
Let $D \subset X$ be a closed subscheme
whose support is $C$. Let $Y = \ov{\{y\}}$ 
with integral closed subscheme structure. We let $d' = d_X(y)$.
For any {\'e}tale map $X' \to X$ equipped with the dimension function induced
by $d_X$ and any points $w, y' \in X'$ such that $y'$ lies over $y$ and
$w$ has codimension one in $\ov{\{y'\}}$, there are canonical maps
(see \cite[\S~8]{Kerz11})
\[
K^M_{d'}(k(y')) \cong C^{\local}(X'_{y'}, D'_{y'}) \to 
C^{\local}(X'_{(w,y')}, D'_{(w,y')}) 
\xrightarrow{\partial} C^{\local}(X'_w, D'_w),
\]
where we let $D' = D \times_X X'$. 
Here, the first map is the canonical map from the Milnor $K$-theory
of the Parshin chain $\{y'\}$ of $X'_{y'}$ to the idele class group
and the second is induced by the residue maps.
Let $Y'$ denote the closure of $\{y'\}$ in $X'$ with integral scheme
structure.
The following lemma is inspired by \cite[Proposition~2.7]{Kato-Saito-2}.

\begin{lem}\label{lem:KS-2.7}
There exists a closed subscheme $E \subset Y$ such that for any
{\'e}tale map $X' \to X$ and any pair of points $w, y' \in X'$ as
above, the composite map $K^M_{d'}(k(y')) \to  C^{\local}(X'_w, D'_w)$ 
annihilates the image of $K^M_{d'}(\sO_{Y',w}, \sI_E\sO_{Y',w})$.
\end{lem}
\begin{proof}
We shall prove the lemma by induction on the codimension of $Y$ in $X$.
Let this be $t$.
Suppose $t = 0$ so that $y$ is the generic point of $X$. 
In this case, we have an exact sequence
\begin{equation}\label{eqn:Curve-1}
K^M_{{d'}}(\sO^h_{X',w}, I_{D'_w}) \to K^M_{{d'}}(k(\eta')) \to 
C^{\local}(X'_w, D'_w) \to 0
\end{equation}
by the definition of $C^{\local}(X'_w, D'_w)$, where $\eta'$ is the
generic point of $X'$.
We can therefore take $E = D$ to satisfy the assertion of the lemma. 
We now assume that $t \ge 1$.

Since $\sO_{X,y}$ is regular of dimension at least one by our assumption,
we can write $\fm_{X,y} = (\pi, x_1, \ldots , x_{t-1})$. We let
$Z$ be the closure of $\Spec({\sO_{X,y}}/{(x_1, \ldots , x_{t-1})})$
in $X$ under the inclusion $\Spec(\sO_{X,y}) \inj X$.
Then $Z$ is an integral closed subscheme of $X$ of codimension $t-1$
which contains $Y$ and which is regular in an open neighborhood of $y$.
Since $X$ is regular at $y$, we can find an open
neighborhood $V$ of $y$ in $X$ such that $V \cap Z$ is also regular. 
Note here that the regular locus of an excellent scheme is open.
By induction on $t$, we can find a closed subscheme $E' \subset Z$
which has the properties stated in the lemma for the closed immersion 
$Z \inj X$. 

Since $Y \subset Z$ is an integral closed subscheme of codimension one
such that $Z$ is regular at $y$, it follows that $\sO_{Z,y}$ is an
excellent discrete valuation ring. 
We let $\ov{\pi}$ be the image of $\pi$ under the surjection
$\fm_{X,y} \surj \fm_{Z,y}$. Then $\ov{\pi}$ is a uniformizer
of $\sO_{Z,y}$.
We claim that there exists a closed subscheme $E'' \subset Z$ satisfying the
following properties.
\begin{enumerate}
\item
$y \notin E''$.
\item
$\sI_{E'', x} \subseteq \sI_{E',x}$ for all $x \in Z^{(1)} \setminus \{y\}$.
\item
If $x \in Z^{(1)} \setminus \{y\}$ and $\ov{\pi} \notin \sO^{\times}_{Z,x}$, then
there exists non-zero elements $a \in \fm_{Z,x}, \  b \in \sO_{Z,x}$ such that 
$\ov{\pi} = ab^{-1}$ and $\sI_{E'',x} \subseteq
a\sI_{E',x} \cap b\sI_{E',x}$.
\end{enumerate}

To prove the claim, let $S$ be set of generic points
of $E'$. We consider $\ov{\pi}$ as a rational function on $Z$ and let
$(\{y\} \amalg S_1)$ and $S_2$ be the sets of
points in $Z^{(1)}$ where $\ov{\pi}$ has zeros and poles, respectively.
We let $T = \{y\} \cup S \cup S_1 \cup S_2$.
For any point $x \in T$,
let $\fm_x$ denote the maximal ideal whose vanishing locus is $x$.

For every $x \in S_1 \cup S_2$, we can find nonzero elements $a \in \fm_{Z,x}$
and $b \in \sO_{Z,x}$ such that $\ov{\pi} = ab^{-1}$ as an element of the function 
field of 
$\sO_{Z,x}$. Note that $b \in \fm_{Z,x}$ if $x \in S_2$.
Since $\sO_{Z,x}$ is a one-dimensional local integral
domain and $\sI_{E'} \neq 0$, it is easy to see that
$a\sI_{E',x} \cap b\sI_{E',x}$ is an $\fm_{Z,x}$-primary ideal.
In particular, $\fm^r_{Z,x} \subset a\sI_{E',x} \cap b\sI_{E',x}$ for
all $r \gg 0$. Since $S$ is the set of generic points of $E'$, we have
$\fm_{Z,x}^r \subset \sI_{E',x}$ for all $x \in S$ and $r \gg 0$.
It follows that for all $r \gg 0$, the ideal
$I = {\underset{x \in T \setminus \{y\}}\prod} \fm^r_{x}$ satisfies the
properties (1) to (3) above if we replace $Z$ by $\Spec(\sO_{Z,T})$. 
We now choose any ideal subsheaf (there exist many) $\sI \subset \sO_Z$ such 
that $\sI \sO_{Z,T} = I$ and let $E'' \subset Z$ be the closed subscheme
defined by $\sI$. It is then clear that $E''$ satisfies the
properties (1) to (3). This proves the claim.

\enlargethispage{10pt}

It follows from \cite[Lemma~1]{Kato86} that $\sI_{E''}$ satisfies the
following for any $x \in Z^{(1)} \setminus \{y\}$.

\vskip .2cm

\hspace*{.3cm} (4) Let $A$ be any local ring having an {\'e}tale local 
homomorphism
$\sO_{Z,x} \to A$, and $F$ the total quotient ring of $A$. Then, for any
$a \in K^M_1(A, \sI_{E''}A)$, the element $\{a, \ov{\pi}\}$ of
$K^M_2(F)$ belongs to $K^M_2(A, \sI_{E'}A)$.

\vskip .2cm

Let $E = E'' \times_Z Y$. We show that $E$ has the property asserted by the
lemma.
Let $f \colon X' \to X$ be any {\'e}tale map and let $y'$ be a pre-image of 
$y$ in $X$. Let $w \in Y' = \ov{\{y'\}}$ be any codimension one point.
Since an {\'e}tale map is open and $y \in Z$, we can find a point
$z' \in X'$ such that $f(z') = z$ and $y' \in \ov{\{z'\}}$.
In particular, we have the inclusions
of codimension one $\ov{\{w\}} \subsetneq \ov{\{y'\}} \subsetneq \ov{\{z'\}}$.
We let $B$ be the set of all points $x \in X'$ such that
$\ov{\{w\}} \subsetneq \ov{\{x\}} \subsetneq \ov{\{z'\}}$.
It easily follows from the definition of $C^{\local}(X'_w, D'_w)$ 
that the composition of the canonical maps
\begin{equation}\label{eqn:KS-2.7-0}
C^{\local}(X'_{(w,z')}, D'_{(w,z')}) \xrightarrow{\partial_1} 
{\underset{x \in B}\oplus}
C^{\local}(X'_{(w,x)}, D'_{(w,x)}) \xrightarrow{\partial_2} C^{\local}(X'_w, D'_w)
\end{equation}
is zero (see \cite[Theorem~8.1]{Kerz11}).

Let $a \in K^M_{d'}(\sO_{Y',w}, \sI_E \sO_{Y',w})$. We need to show that
$\partial_2(a) = 0$. Take a lift $\wt{a}$
of $a$ under the canonical surjection $K^M_{d'}(\sO_{Z',w}, \sI_{E''}\sO_{Z',w})
\surj K^M_{d'}(\sO_{Y',w}, \sI_E \sO_{Y',w})$. 
Consider the element $\{\wt{a}, \ov{\pi}\}$ of the group 
$K^M_{{d'}+1}(k(z')) \cong C^{\local}(X'_{z'}, D'_{z'}) \inj 
C^{\local}(X'_{(w,z')}, D'_{(w,z')})$.
We then have a commutative diagram

\begin{equation}\label{eqn:KS-2.7-1}
\xymatrix@C.8pc{
K^M_{d'+1}(k(z')) \ar[r]^-{\partial_{z', y'}} \ar[d] & K^M_{d'}(k(y')) 
\ar[d] \\
C^{\local}(X'_{(w,z')}, D'_{(w,z')}) \ar[r]^-{\partial_{1,y'}} & 
C^{\local}(X'_{(w,y')}, D'_{(w,y')}),}
\end{equation}
where $\partial_{1,y'}$ is the composition of $\partial_1$ with the
projection map ${\underset{x \in B}\oplus} C^{\local}(X'_{(w,x)}, D'_{(w,x)}) \surj
C^{\local}(X'_{(w,y')}, D'_{(w,y')})$.

Since $\sO_{Z, y}$ is a discrete valuation ring and $f$ is {\'e}tale,
it follows that $\sO_{Z',y'}$ is also a discrete valuation ring
with uniformizing parameter $\ov{\pi}$. This implies
that $\partial_{z', y'}(\{\wt{a}, \ov{\pi}\}) = a$.
It follows from the commutative diagram ~\eqref{eqn:KS-2.7-1}
that $\partial_{1,y'}(\{\wt{a}, \ov{\pi}\}) = a$.
On the other hand, (4) says that 
$\partial_{z', x}(\{\wt{a}, \ov{\pi}\}) = 0$ for all $x \in B \setminus \{y'\}$.
It follows that $\partial_1(\{\wt{a}, \ov{\pi}\}) = a$.
In particular, $\partial_2(a) = \partial_2 \circ 
\partial_1(\{\wt{a}, \ov{\pi}\}) = 0$.
This proves the lemma. 
\end{proof}

\begin{prop}\label{prop:Cont-main}
Let $X$ and $ U$ be as above. Then the following hold.
\begin{enumerate}
\item
The canonical map $K^M_{d_X(P)}(k(P)) \to C_{U/X}$ is continuous for
every Parshin chain $P = (p_0, \ldots , p_s)$ on $(U \subset X)$
such that $p_s$ lies in the regular locus of $X$.
\item
The push-forward  $f_* \colon C_{U'/X'} \to C_{U/X}$ is continuous for every
admissible morphism $f \colon (X',U') \to (X,U)$
such that the image of the generic point of $X'$ lies in the regular locus
of $X$.
\item
If $f \colon X' \to X$ is as in \propref{prop:PF-mod},
then the resulting map $f_* \colon C(X', D') \to C(X,D)$ is continuous.
\end{enumerate}
\end{prop}
\begin{proof}
We can use \lemref{lem:Cont-dominant} to reduce (2)
to the case when $f$ is a closed immersion and $U' = U \cap X'$.
We shall therefore assume this to be the case in the statement of (2).

We shall now prove the proposition by induction on
the codimension of Parshin chain in (1) and on the codimension of
$f(X')$ in (2).
If the codimension of the Parshin chain is zero, then the continuity is
clear. Similarly, if the codimension of $f(X')$ is zero, then (2) 
follows from \lemref{lem:Cont-dominant}. So we can assume both these
numbers to be positive.

To prove (1) in case the codimension of $P$ in $X$ is one,
Lemma~\ref{lem:Cont-gen-0} says that it is enough to show that the
composite map $K^M_{d_X(P)}(k(P)) \to C(X,D)$ is continuous for
every closed subscheme $D$ with support $C$. But this follows directly from
Lemmas~\ref{lem:Disc} and ~\ref{lem:KS-2.7}.
The codimension one case of (2) follows from that of (1) using
\lemref{lem:DST-univ}.

We now prove the general case of (1). Let $P = (p_0, \ldots , p_s)$.
%Let $\eta$ be the generic point of $X$ and $d' = d_X(\eta)$.
Let $Y \subset X$ be the closure of $p_s$ in $X$ with the integral
closed subscheme structure. Since $\sO_{X,p_s}$ is regular, we can
find inclusions of integral closed subschemes
$Y \subsetneq Z \subsetneq X$ such that the generic points of
$Y$ and $Z$ lie in the regular locus of $X$ and $Y$ has codimension one in $Z$.
Furthermore, $p_s$ lies in the regular locus of $Z$
(see the proof of \lemref{lem:KS-2.7}).
Let $V = Z \cap U$. 

By induction argument for (2), the map $C_{V/Z} \to C_{U/X}$ is continuous.
On the other hand, the codimension one case of (1) implies that the map
$K^M_{d_X(P)}(k(P)) \to C_{V/Z}$ is continuous. The factorization
$K^M_{d_X(P)}(k(P)) \to C_{V/Z} \to  C_{U/X}$ now finishes the proof of the
general case of (1). 

To prove the general case of (2), it suffices
using \lemref{lem:DST-univ} to show that
the map $K^M_{d_X(P)}(k(P)) \to C_{U/X}$ is continuous for every
Parshin chain $P$ on $(U' \subset X')$.
If $P$ is nonmaximal, then the direct sum topology on $I_{U'/X'}$
restricts to the discrete topology on $K^M_{d_X(P)}(k(P))$.
So the continuity is clear. If $P$ is a maximal
Parshin chain on $(U' \subset X')$, then its generic point is same
as that of $X'$. Hence, the map $K^M_{d_X(P)}(k(P)) \to C_{U/X}$ is continuous 
by (1). This finishes the proof of (2). The last part (3) is immediate from
\lemref{lem:Disc}.
\end{proof}

\begin{remk} \label{remk:Kerz-comp}
Part (2) of the above proposition (hence part (1)) is also proven in
\cite[Theorem~6.1]{Kerz11} when $X' \to X$ is a closed immersion. 
However, the proof of {\sl op. cit.} requires
modification. The reason is that 
in order to apply the induction on $d_X(\eta)$ (where $\eta$ is the generic 
point of $X$), one needs to find a closed immersion $X'' \subset X$
such that $X'' \cap U$ is regular and $X' \subset X''$. But this may not be 
possible. It is not enough to find such a closed immersion in a
neighborhood of the generic point of $X'$.
\end{remk}

\subsection{Idele class group of the function field}
\label{sec:Idel-F}
It is easy to see from the construction of the push-forward map between the
idele class groups that if $X'' \xrightarrow{f'} X' \xrightarrow{f} X$
are admissible morphisms, then $(f \circ f')_* = f_* \circ f'_*$, whenever 
these maps are defined.
In particular, the groups $C_{U/X}$ form a co-filtered system of
abelian groups if we let $U \subset X$ run through all dense open subschemes.
It follows from Remark~\ref{remk:PF-mod-0} that 
$\{C(X,D)\}$, where $D$ runs through all closed subschemes with 
support $C$, is a co-filtered system of abelian groups
whose structure maps are surjective.
All this allows us to define the idele and the idele class groups of the 
function field of $X$ as follows.
\begin{defn}\label{defn:Idele-fn-field}
Let $X$ be an excellent integral scheme of dimension $d$ endowed with a 
dimension function. Let $K$ denote the function field of $X$.
We let
\[
I_{K/X} = {\underset{U \subset X}\varprojlim} I_{U/X}, \ \
C_{K/X} = {\underset{U \subset X}\varprojlim} C_{U/X} \ \
\mbox{and} \ \ \wt{C}_{U/X} = {\underset{|D| = 
X \setminus U}\varprojlim} C(X,D).
\]
\end{defn}

These groups are equipped with their inverse limit topologies.
There are canonical continuous maps $I_{K/X} \to C_{K/X}$
and $C_{U/X} \to \wt{C}_{U/X}$. These maps may not be
surjective even though their images are dense.
However, if $d = 1$, then $I_{K/X}$ and $C_{K/X}$ coincide with the classical
idele group and the idele class group of $K$
(see \cite[Proposition~2.1]{Kerz11}). In particular,
$I_{K/X} \to C_{K/X}$ is a topological quotient in this case. 

For any open dense $U \subset X$ and closed subscheme $D \subset X$
with support $X \setminus U$, there are canonical continuous homomorphisms
\begin{equation}\label{eqn:Idele-fn-field-0}
C_{U/X} \to \wt{C}_{U/X} \stackrel{\lambda_D}{\surj} C(X,D).
\end{equation}

\subsection{Idele class group and 0-cycles}\label{sec:ICGC}
For $X \in \Sch_k$ where $k$ is a field, let $\CH^F_0(X)$ denote the
classical Chow group of 0-cycles on $X$ in the sense of \cite{Fulton}
(where it is denoted by $\CH_0(X)$). 
%The following will be used later in this paper.

\begin{prop}\label{prop:Degree}
Let $X$ be an integral scheme which is of 
finite type over a field $k$ and is 
endowed  with its canonical dimension function.
Let $D \subset X$ be a nowhere dense closed subscheme. 
Let $C = D_\red$ and $U = X \setminus C$.  
Then the following hold.
\begin{enumerate}
\item
There is a canonical isomorphism $C_{X/X} \cong C(X,D) \cong \CH^F_0(X)$ if 
$D = \emptyset$.
\item
If $X$ is projective over $k$, then there is a degree map
$\deg \colon C(X,D) \to \Z$ which coincides with the classical degree
map for $\CH^F_0(X)$ if $D = \emptyset$. In particular, for every
closed point $x \in U$, the composite map
$\lambda_x \colon K^M_0(k(x)) \to C(X,D) \xrightarrow{\deg} \Z$
has the property that $\lambda_x(1)  = [k(x): k]$.
\item
Suppose $X$ is projective over $k$.
If either $k$ is finite and $X$ geometrically irreducible
or $k$ is separably closed, 
then there is an exact sequence of discrete abelian groups
\[
0 \to C(X,D)_0 \to C(X,D) \xrightarrow{\deg} \Z \to 0.
\]
\item 
If $k$ is any field and  $X$ is projective but not necessarily
geometrically irreducible over $k$, then there exists an integer $n \ge 1$
which depends only on $U$ such that we have an 
exact sequence of discrete abelian groups
\[
0 \to C(X,D)_0 \to C(X,D) \xrightarrow{{\deg}/n} \Z \to 0.
\]
In particular, 
for every $m \geq 1$, we have a short exact sequence 
\[
0 \to C(X,D)_0/m \to C(X,D)/m \xrightarrow{{\deg}/n} \Z/m \to 0.
\]
\end{enumerate}
\end{prop}
\begin{proof}
If $D = \emptyset$, then there is no maximal chain on $(X \subset X)$ and hence
$C_{X/X} \cong C(X,D)$. On the other hand, it is immediate from the
definition that there is an exact sequence
\begin{equation}\label{eqn:Degree-00} 
{\underset{x \in X_{(1)}}\bigoplus} K^M_1(k(x)) \xrightarrow{\partial}
{\underset{x \in X_{(0)}}\bigoplus} K^M_0(k(x)) \to C_{X/X} \to 0.
\end{equation}
The desired isomorphism now follows because it is well known that
the boundary map $\partial$ coincides with the map that sends a rational
function on a curve to its associated divisor. This proves (1).

We now prove (2). 
By \propref{prop:PF-mod} and Remark~\ref{remk:PF-mod-0}, there is a 
canonical map $C(X,D) \to C(X, \emptyset) \cong 
C_{X/X}$. Using (1), we can
identify the latter with $\CH^F_0(X)$. Since $X$ is projective, we have the
degree map $\CH^F_0(X) \to \Z$ which takes a closed point to the degree
of its residue field over $k$. Hence, by composition, we get the
desired map $\deg \colon C(X,D) \to \Z$. It is clear from the construction 
that $\deg([x]) = [k(x):k]$ for every closed point $x \in U$. 

To prove (3), we only need to show that 
the degree map is surjective. For this, it suffices to show that the
composite map $C_{U/X} \to C(X,D) \xrightarrow{\deg} \Z$ is surjective.
We first assume that $k$ is finite.
We can now find a geometrically integral curve $Y' \subset X$ 
not contained in $D$ (see \cite[Theorem~7.5]{Ghosh-Krishna}).
We let $Y = Y'_n$ be the normalization of $Y'$. Then
$Y$ is also geometrically integral as $k$ is perfect.
We let $V = U \times_X Y$. It follows from
\propref{prop:Funtorial*} that there is a push-forward map
$f_* \colon C_{V/Y} \to C_{U/X}$, where $f \colon Y \to X$ is the
canonical finite map. We have a commutative diagram
\begin{equation}\label{eqn:Degree-1} 
\xymatrix@C.8pc{
C_{V/Y} \ar[r] \ar[d]_-{f_*} & C_{Y/Y} \ar[dr]^-{\deg} \ar[d]^-{f_*} &  \\
C_{U/X} \ar[r] & C_{X/X} \ar[r]_-{\deg} & \Z.}
\end{equation}
We can therefore assume that $X$ is a normal projective geometrically 
 integral  curve (the proposition is trivial if $X$ is a point).

If $k$ is separably closed, then $U$ has a rational point, so
we are done. If $k$ is finite, then an easy consequence of the 
Lang-Weil estimate (this uses geometric integrality)
shows that $X$ has a 0-cycle of degree one (i.e.,
$X$  has index one). It follows that $\deg \colon C_{X/X} \cong
\CH^F_0(X) \to \Z$ is surjective. On the other hand, the regularity of $X$
and \lemref{lem:Cont-dominant} together imply that the map $C_{U/X} \to C_{X/X}$ 
is surjective. The result follows.

We now prove (4). 
We define the degree map $\deg \colon C_{U/X} \to \Z$ to be the
composite map $C_{U/X} \to C_{X/X} \cong \CH^F_0(X) \to \Z$ as above.
By the definition of this map, we see that for every closed point
$x \in U$, the image of the composite map $K^M_0(k(x)) \to C_{U/X} \to C_{X/X} 
\to \CH^F_0(X) \to \Z$ is the subgroup generated by $[k(x):k]$.
It follows that there exists an integer $n \ge 1$ such that the image
of $\deg \colon C_{U/X} \to \Z$ is $n\Z$. 
Since there is a factorization $C_{U/X} \surj C(X,D) \to \CH^F_0(X)$
for all closed subschemes $D \subset X$ with support $C$
(see Remark~\ref{remk:PF-mod-0}), we see that
the degree map of $C_{U/X}$ induces such a map for all $C(X,D)$.
Moreover, ${\deg}(C(X,D)) = {\deg}(C_{U/X}) = n\Z$.
This proves the exact sequence of (4). 
\end{proof}

It follows from \lemref{lem:Disc} that all maps in
the exact sequences of \propref{prop:Degree} are continuous, where
$\Z$ is considered a discrete topological abelian group.
Taking the limit, we get a
surjective continuous homomorphism 
${\deg}/n \colon \wt{C}_{U/X} \surj \Z$.
We let $(\wt{C}_{U/X})_0$ denote the kernel of this map.
Since the transition maps of the projective system
$\{C(X,D)_0\}_{|D| = C}$ are surjective, we obtain the following.

\begin{cor}\label{cor:Degree-0}
Under the hypothesis of \propref{prop:Degree},
there is an isomorphism of topological abelian groups
\[
(\wt{C}_{U/X})_0 \xrightarrow{\cong} {\underset{|D| = C}\varprojlim} \
C(X,D)_0.
\]
\end{cor}

\subsection{Relation with Wiesend's idele class group}
\label{sec:Wiesend}
Let $k$ be a finite field and let $X \in \Sch_k$ be a complete integral
scheme. Let $U \subset X$ be dense open. Let $W(U)$ denote Wiesend's
idele class group of $U$ (see \cite[Definition~7.1]{Kerz11}).
It's definition is identical to that of $C_{U/X}$ except that we take only
Parshin chains on $(U \subset X)$ of length at most one as generators
and relations. The topology of $W(U)$ is induced by the canonical topology
on $K^M_1(k(P))$, where $P$ runs through the Parshin chains on $(U \subset X)$
of length one.
It follows from \propref{prop:Cont-main} that there is a canonical
continuous homomorphism 
\begin{equation}\label{eqn:Wiesend-0}
\alpha_{U/X} \colon W(U) \to C_{U/X}.
\end{equation}

We can define Wiesend's class group with modulus $W(X,D)$ as a quotient of 
$W(U)$ analogous to $C(X,D)$. The main question that one needs to answer
in order to identify the class field theories of Kato-Saito \cite{Kato-Saito-2}
and Kerz-Saito \cite{Kerz-Saito-2} is whether $\alpha_{U/X}$
descends to a homomorphism
$\alpha_{X|D} \colon W(X,D) \to C(X,D)$ for every closed subscheme $D \subset X$
with support $C$. 
In this context, we note that one can define the $C$-topology on $W(U)$
exactly as on $C_{U/X}$. It is however not a priori clear that 
$\alpha_{U/X}$ is continuous with respect to the $C$-topology on $W(U)$.
If it were, it would imply that $\alpha_{U/X}$ descends at least to a
morphism of pro-abelian groups
$\alpha^{\bullet}_{X|D} \colon \{W(X,D)\}_{|D| = C} \to \{C(X,D)\}_{|D| = C}$.
These questions are studied in \cite{Gupta-Krishna-CFT}.

\section{Reciprocity map for $C_{{U/X}}$}\label{sec:Rec}
In this section, we shall define the reciprocity map for the idele
class group $C_{U/X}$.
We begin by recalling Kato's reciprocity map for Henselian local fields
which we shall heavily use.

\subsection{Kato's reciprocity}\label{sec:KR}
Let $X$ be a Noetherian $k$-scheme, where $k$ is a field of 
characteristic $p > 0$.  
We consider the following {\'e}tale sheaves on $X$.

Let $\{W_r\Omega^{\bullet}_X\}_{r \ge 1}$ be the projective system of
Bloch-Deligne-Illusie de Rham-Witt complexes of $X$ as $r \ge 1$
(see \cite{Illusie}). 
Recall that $\{W_r\Omega^{\bullet}_X\}_{r \ge 1}$ is equipped with
three operators: the
restriction $\sR \colon W_{r+1}\Omega^{\bullet}_X \to  W_{r}\Omega^{\bullet}_X$,
the Frobenius $F \colon W_{r+1}\Omega^{\bullet}_X \to  W_{r}\Omega^{\bullet}_X$
and the Verschiebung $V \colon W_{r}\Omega^{\bullet}_X \to  
W_{r+1}\Omega^{\bullet}_X$. These operators satisfy a set of relations.
%which turn $\{W_r\Omega^{\bullet}_X\}_{r \ge 1}$ into a Witt complex (see \cite[\S~1]{R}).
There is Teichm{\"u}ller lift $[.] \colon \sO_X \to W_r\sO_X$, which is
locally given by $[x] = (x, 0, \ldots , 0)$. This is a multiplicative
homomorphism.

The Frobenius of $W_{r+1}\Omega^n_X$ descends to give a map
$W_{r}\Omega^{n}_X \to {W_r \Omega^n_X}/{d V^{r-1} \Omega^{n-1}_X}$, which we
shall also denote by $F$. We denote the canonical surjection between
these two groups by the identity map.
Hence we have a map
$(1 - F) \colon W_{r}\Omega^{n}_X \to {W_r \Omega^n_X}/{d V^{r-1} \Omega^{n-1}_X}$,
which is surjective in the {\'e}tale topology.
We let $W_r\Omega^n_{X, \log}$ denote the kernel of this map of
{\'e}tale sheaves. 
We let ${\Z}/{p^r}(n) = (W_r\Omega^n_{X, \log})[-n]$. 
There is a unique map $\underline{p} \colon
W_r\Omega^n_X \to W_{r+1}\Omega^n_X$ such that the composite
$W_{r+1}\Omega^n_X \xrightarrow{\sR}  W_r\Omega^n_X \to W_{r+1}\Omega^n_X$
is multiplication by $p$. This induces a map
$\underline{p} \colon W_r\Omega^n_{X, \log} \to W_{r+1}\Omega^n_{X, \log}$.
Hence, $\{{\Z}/{p^r}(n)\}_{r \ge 1}$ is a direct system of
{\'e}tale sheaves on $X$.
Note that ${\Z}/{p^r}(0) = W_r\sO_{X, \log} \cong {\Z}/{p^r}$.

If $m \ge 1$ is an integer prime to $p$, we let
${\Z}/m(n) = {\Z}/m$ if $n =0$ and $(\mu_m)^{\otimes n}$ if $n \ge 1$.
For $n < 0$, we let ${\Z}/m(n) = \sHom({\Z}/m(-n), {\Z}/m)$.
We let $H^p_{\et}(X, {\Q}/{\Z}(q)) = {\underset{m \ge 1}\varinjlim} \
H^p_{\et}(X, {\Z}/m(q))$.
We shall write $H^p_{\et}(X, {\Q}/{\Z}(q))$ often as $H^{p,q}(X)$.

Let us now assume that  $A$ is a 
Noetherian 
%local 
ring containing a field of 
characteristic $p > 0$ and let $X = \Spec(A)$. The Kummer 
sequence and the cup product together produce a
Norm-residue homomorphism ${K^M_n(A)}/m \to  H^{n}_{\et}(A, {\Z}/m(n))$
for any $m$ prime to $p$ and any integer $n \ge 0$.
Similarly, 
we have the dlog map between {\'e}tale sheaves on $X$:
%\W_r\Omega^n_{X, \log} has been changed to W_r\Omega^n_{X, \log}
\[
{\sK^M_{n,X}}/{p^r} \to W_r\Omega^n_{X, \log}; \ \{a_1, \ldots , a_n\} \mapsto
d\log([a_1])\wedge \cdots \wedge d\log([a_n]).
\]
Taking the global sections, we get the $p$-adic Norm-residue homomorphism
\[
{K^M_{n}(A)}/{p^r} \to H^0_{\et}(A, W_r\Omega^n_{X, \log}) \cong
H^n_{\et}(A, {\Z}/{p^r}(n)).\]
Using the cup product for {\'e}tale cohomology and taking the limit over
$r, m \ge 1$, we get a pairing

\begin{equation}\label{eqn:Pairing-2}
K^M_n(A) \times H^1_{\et}(A, {\Q}/{\Z}) \to 
H^{n+1}_{\et}(A, {\Q}/{\Z}(n)).
\end{equation}

Let us now assume that there is a sequence of fields
$\un{K} = \{K_0, \ldots , K_n\}$ of characteristic $p$ such that
$K_0$ is a finite field and each $K_i$ with $i \ge 1$ 
is the quotient field of a 
Henselian discrete valuation ring whose residue field is $K_{i-1}$. 
In this case, we shall say that $K_n$ is an $n$-dimensional
Henselian local field. 
It follows from \cite[Theorem~3.5]{Kato-Saito-2}
that there is a residue isomorphism
\begin{equation}\label{eqn:Residue-iso}
{\rm Res} \colon H^{i+1}_{\et}(K_{i}, {\Q}/{\Z}(i)) 
\xrightarrow{\cong} H^{i}_{\et}(K_{i-1}, {\Q}/{\Z}(i-1))
\end{equation}
for every $i \ge 1$. Since it is easy to check that
$H^{1}_{\et}(K_{0}, {\Q}/{\Z}) \cong {\Q}/{\Z}$, we conclude that there
is a canonical isomorphism
${\rm Res} \colon H^{n+1}_{\et}(K_{n}, {\Q}/{\Z}(n))
\xrightarrow{\cong}  {\Q}/{\Z}$.
It follows from this and ~\eqref{eqn:Pairing-2} that
there is a pairing
\begin{equation}\label{eqn:Pairing-3}
K^M_n(K_n) \times H^1_{\et}(K_n, {\Q}/{\Z}) \xrightarrow{\psi_{K_n}} {\Q}/{\Z}.
\end{equation} 

On the other hand, for every integer $m \ge 1$ and every scheme $X$, there is a
natural isomorphism $H^1_{\et}(X, {\Z}/m) \cong \Hom_{\cont}(\pi_1(X), {\Z}/m)
\cong \Hom_{\cont}(\pi^{\ab}_1(X), {\Z}/m)$
which characterizes finite {\'e}tale covers $Y \to X$ together with a
morphism $\Aut (Y/X) \to {\Z}/m$.
It follows that there is a natural isomorphism
$H^1_{\et}(X, {\Q}/{\Z}) \cong \Hom_{\cont}(\pi_1(X), {\Q}/{\Z})
\cong (\pi^{\ab}_1(X))^{\vee}$, where ${\Q}/{\Z}$ has discrete topology,
$\pi^{\ab}_1(X)$ has profinite topology and $(\pi^{\ab}_1(X))^{\vee}$
denotes the Pontryagin dual (see \S~\ref{sec:PD*}). 
By the Pontryagin duality for 
profinite groups, we get a continuous isomorphism
${H^1_{\et}(X, {\Q}/{\Z})}^{\vee} \cong \pi^{\ab}_1(X)$.

Combining all of the above, we conclude the following.
\begin{prop}\label{prop:Reciprocity*}
Let $K$ be an $n$-dimensional Henselian local field. Then there is a
canonical reciprocity map
\[
\rho_{K} \colon K^M_n(K) \to \Gal(K^{\ab}/K).
\]
\end{prop}

It is easy to see that a finite extension of 
an $n$-dimensional Henselian local field is of the same kind.
Recall also that if $L/K$ is a finite abelian extension, and
if $G$ (resp. $H$) denotes the abelianized Galois group of $K$
(resp. $L$), then $H \subset G$ is a normal subgroup of finite index.
Hence, there is a co-restriction map
${\rm Cores}_{H/G} \colon H^*(H, M) \to H^*(G,M)$ for $G$-module $M$.
In particular, we have the co-restriction map
${\rm Cores}_{H/G} \colon H^1(H, {\Q}/{\Z}) \to  H^1(G, {\Q}/{\Z})$.
Taking the Pontryagin duals, we get a `transfer map'
${\rm tr}_{L/K} \colon G \to H$.

\vskip .2cm

The following result is due to Kato \cite{Kato80}.

\begin{prop}\label{prop:Kato-**}
Let $K$ be as in \propref{prop:Reciprocity*}. 
Then $\rho_K$ has following properties.
\begin{enumerate}
\item
For each finite extension $L/K$, there is a commutative diagram
\begin{equation}\label{eqn:Kato-**-0}
\xymatrix@C.8pc{
K^M_n(L) \ar[r]^-{\rho_L} \ar[d]_-{N_{L/K}} & \Gal(L^{\ab}/L) \ar[d]^-{\rm can} \\
K^M_n(K) \ar[r]^-{\rho_K} & \Gal(K^{\ab}/K).}
\end{equation}
In particular, there is a canonical homomorphism 
\[
\ov{\rho}_{L/K} \colon \frac{K^M_n(K)}{N_{L/K}(K^M_n(L))} \to \Gal(L/K)
\]
if $L/K$ is a Galois extension.
This map is an isomorphism if $K$ is furthermore complete.

%In particular, the reciprocity map $\rho_K$ is continuous with respect to
%the canonical topology of $K^M_n(K)$ and profinite topology of
%$\Gal(K^{\ab}/K)$.

\item
For each finite extension $L/K$, there is a commutative diagram
\begin{equation}\label{eqn:Kato-**-1}
\xymatrix@C.8pc{
K^M_n(K) \ar[r]^-{\rho_K} \ar[d]_-{\rm can} & \Gal(K^{\ab}/K) 
\ar[d]^-{{\rm tr}_{L/K}} \\
K^M_n(L) \ar[r]^-{\rho_L} & \Gal(L^{\ab}/L).}
\end{equation}

\item
If $\ff$ denotes the residue field of $K$, then there is a commutative diagram
\begin{equation}\label{eqn:Kato-**-2}
\xymatrix@C.8pc{
K^M_n(K) \ar[r]^-{\rho_K} \ar[d]_-{\partial} & \Gal(K^{\ab}/K) \ar[d] \\
K^M_{n-1}(\ff) \ar[r]^-{\rho_K} & \Gal(\ff^{\ab}/\ff),}
\end{equation}
where the left vertical arrow is the residue map and the
right vertical arrow is the canonical surjection if we
identify $\Gal(\ff^{\ab}/\ff)$ with $\Gal(K^{\nr}/K)$, the Galois
 group of the maximal unramified subextension of ${K^{\ab}}/K$.
\end{enumerate}
\end{prop}
\begin{proof}
This proposition is proven in \cite[\S~3]{Kato80}.
The only thing one needs to remark here is that even though Kato
states this proposition for complete fields, his proofs of
the commutativity of ~\eqref{eqn:Kato-**-0}, ~\eqref{eqn:Kato-**-1} and 
~\eqref{eqn:Kato-**-2}, which only depend on the compatibility of the
Norm-residue homomorphisms with Norm, transfer and residue maps, 
do not use this assumption
(see \cite[Lemma~4.5]{Kato-Saito-2}).
\end{proof}

\subsection{Openness of Norm subgroup}\label{sec:N-open}
We shall show in this subsection that the norm subgroups associated to finite
abelian extensions of a Henselian discrete valuation field $K$ are often open.
This can not be deduced from the last part of
\propref{prop:Kato-**}(1) even if the Galois theory of $K$ 
coincides with that of its completion $\wh{K}$. The reason is that the
quotient of the norm subgroup of the Milnor $K$-theory associated to a finite
abelian extension $L/K$ may not coincide with the corresponding quotient
associated to ${\wh{L}}/{\wh{K}}$. 

\vskip .2cm

We begin with the following elementary observations.

\begin{lem}\label{lem:Decomposition}
Let $L/K$ be a finite abelian extension. Then
it is the inductive limit of a tower of abelian extensions
each of which is cyclic of prime order.
\end{lem}
\begin{proof}
It is an easy exercise using Galois theory of fields.
\end{proof}

\begin{lem}\label{lem:Brauer-tor}
For any field $K$, there is a functorial isomorphism
\[
\Br(K) \xrightarrow{\cong} H^{2}(K, {\Q}/{\Z}(1)),
\]
where $\Br(K)$ is the Brauer group.
\end{lem}
\begin{proof}
This is an easy exercise using the fact that $\Br(K)$ is a torsion group
and then using the Norm-residue homomorphisms.
\end{proof}

\begin{lem}\label{lem:Hensel-Open-1}
Let $L/K$ be a finite cyclic extension of a Henselian discrete 
valuation field. Then $N_{L/K}(L^{\times})$ is
an open subgroup of $K^{\times}$.
\end{lem}
\begin{proof}
Let $G$ and $\ov{G}$ denote the Galois groups of $L/K$ and ${\ov{K}}/K$,
respectively so that there is a surjection $\ov{G} \surj G$.
We consider the diagram
\begin{equation}\label{eqn:Hensel-Open-1-0}
\xymatrix@C.8pc{
H^0(G, L^{\times}) \times H^1(G, {\Q}/{\Z}) 
\ar[r]^{\cong} \ar[d] &
H^0(G, L^{\times}) \times H^2(G, \Z) \ar[r]^-{\cup} &
H^2(G, L^{\times}) \ar[r]^{\cong} \ar[d] & \Br(L/K) \ar@{^{(}->}[d] \\ 
K^{\times} \times H^{1,0}(K) \ar[r]^-{BK \times id}  &
H^{1,1}(K) \times  H^{1,0}(K) \ar[r]^-{\cup} &
H^{2,1}(K) \ar[r]^{\cong}  & \Br(K),} 
\end{equation}
where $BK$ is the Norm-residue homomorphism of Bloch-Kato
and the isomorphism on the bottom right is by \lemref{lem:Brauer-tor}.
The map 
\[
H^1(G, {\Q}/{\Z}) \cong \Hom_{\cont}(G, {\Q}/{\Z}) 
\to H^{1,0}(K) \cong H^1(\ov{G}, {\Q}/{\Z}) \cong 
\Hom_{\cont}(\ov{G}, {\Q}/{\Z})
\]
is the canonical inclusion induced by $\ov{G} \surj G$.
It is easy to check that ~\eqref{eqn:Hensel-Open-1-0} is commutative.

Since $G$ is cyclic, $\Hom_{\cont}(G, {\Q}/{\Z}) = G^{\vee}$ is
also cyclic of the same order. Let $\eta$ be a generator of
$G^{\vee}$. Then the top cup product in
~\eqref{eqn:Hensel-Open-1-0} induces a
map $K^{\times} \xrightarrow{(-)\cup \eta}  \Br(L/K)$.
Moreover, one knows that this induces an isomorphism
\begin{equation}\label{eqn:Hensel-Open-1-1}
{K^{\times}}/{N_{L/K}(L^{\times})} \xrightarrow{(-)\cup \eta}  \Br(L/K)
\end{equation}
(see \cite[Chapter~XIV, Proposition~2, Corollary~1]{Serre-LF}).

On the other hand, it follows from \cite[Lemma~2.2]{Kato89}
that as an element of $H^{1,0}(K)$, there exists an
integer $n \ge 0$ such that $\eta \in \Fil_n H^{1,0}(K)$, where
$\Fil_n H^{1,0}(K)$ is a certain subgroup of $H^{1,0}(K)$.
What is important for us is that every element $a \in \Fil_n H^{1,0}(K)$ has 
the property that the map
$K^{\times} \xrightarrow{(-)\cup a}  \Br(K)$ under the bottom cup product in
~\eqref{eqn:Hensel-Open-1-0} annihilates $U_{n+1}(K)$
(see \cite[Definition~2.1]{Kato89}).
It follows that $U_{n+1}(K) \cup \eta = 0$ in $\Br(L/K)$.
We conclude from ~\eqref{eqn:Hensel-Open-1-1} that
$U_{n+1}(K) \subset N_{L/K}(L^{\times})$.
This finishes the proof.
\end{proof}

\begin{cor}\label{cor:Milnor-norm-open}
Let $L$ be a finite abelian extension of a Henselian discrete valuation 
field $K$ and $n \ge 0$ an integer. 
Then $N_{L/K} \colon K^M_n(L) \to K^M_n(K)$ is continuous.
If $L/K$ is cyclic, then the image of $N_{L/K} \colon K^M_n(L) \to K^M_n(K)$
is an open subgroup of $K^M_n(K)$.
\end{cor}
\begin{proof}
The $n = 0$ case is obvious, so we assume $n \ge 1$.
The first assertion now follows from \lemref{lem:Kato-Kerz} and
\cite[Lemma~4.3]{Kato-Saito-2}.
To prove the second assertion, we note that \lemref{lem:Kato-Kerz} implies
that $U = \{U_1, \sO^{\times}_K, \ldots , \sO^{\times}_K\}$ is open
in $K^M_n(K)$ for every open subgroup $U_1 \subset \sO^{\times}_K$.
Furthermore, the projection formula for the norm map shows that
if $U_1 \subset N_{L/K}(L^{\times})$, then
$U \subset N_{L/K}(K^M_n(L))$. Hence, it suffices to show that
$N_{L/K}(L^{\times})$ contains some open subgroup of $K^{\times}$.
But this follows from \lemref{lem:Hensel-Open-1}.
\end{proof}

\begin{remk}\label{remk:Openness}
We expect the statement of \lemref{lem:Hensel-Open-1} to be true
for any finite Galois extension of a Henselian discrete 
valuation field $K$ but we are not aware of a proof.
%If $K$ is complete, then this can be proven using the results of \cite[Chapter~V]{Serre-LF}. 
If one assumes that the field extension is 
unramified, an answer is given by the following stronger assertion.
We shall not need this for proving the main results of this paper.
\end{remk}

\begin{lem}\label{lem:Hensel-open-0}
Let $L/K$ be a finite unramified abelian extension of a Henselian discrete 
valuation field. Then $N_{L/K}(U_n(L)) = U_n(K)$ for all $n \ge 1$.
In particular, $N_{L/K}(L^{\times})$ is
an open subgroup of $K^{\times}$.
\end{lem}
\begin{proof}
It follows from \cite[Chapter~V, Proposition~1]{Serre-LF} 
that $N_{L/K}(U_n(L)) \subset U_n(K)$ for every $n \ge 1$.
We only need to show that this inclusion is an equality.
Note that although $K$ is assumed to be complete in {\sl ibid.},
proof of this inclusion does not require the completeness assumption.

Let $\fk$ and $\fl$ denote the residue fields of $K$ and $L$,
respectively. For any $g(X) \in \sO_K[X]$ and $n \ge 1$, we shall denote
its image in ${\sO_K}/{\fm^n_K}[X]$ by $\ov{g}_n(X)$. For any $a \in \sO_K$,
we shall let $\ov{a}_n$ be its image in ${\sO_K}/{\fm^n_K}$. 
We shall use similar
notations for elements of $\sO_L$ and $\sO_L[X]$.

Since $L/K$ is unramified, $\fl/\fk$ is an abelian extension of
degree $[L:K]$. Let this degree be $d$.
We choose $\alpha' \in \fl$ such that $\fl = \fk(\alpha')$
and let $f(X) \in \fk[X]$ be the minimal polynomial of $\alpha'$ over $\fk$.
Let $g(X) = X^d + a_{d-1}X^{d-1} + \cdots + a_1 X + a_0 \in \sO_K[X]$ be
a polynomial such that $\ov{g}_1(X) = f(X)$. By the Hensel lemma
(applied to $\sO_L$), there exists $\alpha \in \sO_L$ such that
$g(\alpha) = 0$ and $\ov{\alpha}_1 = \alpha'$. Note that $\alpha'$ is
a simple root of $f(X)$ over $\fl$.

We now let $n \ge 1$ be any integer and $u \in U_n(K)$.
We let $h(X) = X^d + a_{d-1}X^{d-1} + \cdots + a_1 X + ua_0 \in \sO_K[X]$.
It is then clear that $\ov{h}_n(X) = \ov{g}_n(X) \in {\sO_K}/{\fm^n_K}[X]$.
Hence, this equality holds in ${\sO_L}/{\fm^n_L}[X]$ too.

We write $g(X) = (X - \alpha) g'(X)$ in $\sO_L[X]$.
We then get
\[
\ov{h}_n(X) = \ov{g}_n(X) = (X - \ov{\alpha}_n)\ov{g'}_n(X)
\]
in ${\sO_L}/{\fm^n_L}[X]$.
Furthermore, since the images of $X - \ov{\alpha}_n$ and $\ov{g'}_n(X)$ in 
$\fl[X]$ generate the unit ideal, they must generate the unit ideal in
${\sO_L}/{\fm^n_L}[X]$ too.

We now use the fact that $(\sO_L, \fm^n_L)$ is a Henselian pair
because $(\sO_L, \fm_L)$ is.
It follows that there exists $\beta \in \sO_L$ such that
$h(\beta) = 0$ and $\ov{\beta}_n = \ov{\alpha}_n$.
Since $\ov{h}_1(X) = \ov{g}_1(X)$ is irreducible, it follows
that $h(X)$ is also irreducible and 
 $h(X)$ is therefore  the minimal polynomial of $\beta$ over 
$K$.
In particular, $N_{L/K}(\beta) = (-1)^dua_0$.
Similarly, we have $N_{L/K}(\alpha) = (-1)^da_0$.
We therefore get $N_{L/K}(\alpha^{-1} \cdot \beta)
= u$. Note that $\alpha, \beta \in U_0(L)$ since their images in
$\fl$ are units. Since $\ov{(\alpha^{-1} \cdot \beta)}_n = 1$,
it follows that $\alpha^{-1} \cdot \beta \in U_n(L)$.
This proves the lemma. 
\end{proof}

\subsection{Continuity of Kato's reciprocity}
\label{sec:Rec-Cont}
We shall now show that Kato's reciprocity map is continuous.

\begin{prop}\label{prop:Kato-rec-cont}
Let $K$ be an $n$-dimensional Henselian local field. Then
the reciprocity map $\rho_K$ is continuous with respect to the canonical 
topology on $K^M_n(K)$ and profinite topology on $\Gal({K^{\ab}}/K)$.
\end{prop}
\begin{proof}
Using the profinite topology on $\Gal({K^{\ab}}/K)$, the proposition is 
equivalent to saying that for every finite abelian extension $L/K$,
the composite map $\rho_{K,L} \colon K^M_n(K) \to \Gal(L/K)$
annihilates some open subgroup of $K^M_n(K)$ in its canonical topology.
We let $G = \Gal(L/K)$.  By \lemref{lem:Decomposition}, we can then 
 write $G = G_1 \times \cdots \times G_r$, 
where each $G_i$
is a finite cyclic group. Let $p_i \colon G \to G_i$ be the projection.
Since the intersection of finitely many nonempty open subgroups in
a topological group is a nonempty open subgroup,
it suffices to show that $p_i \circ \rho_{K,L} \colon K^M_n(K) \to G_i$
annihilates some open subgroup of $K^M_n(K)$ for each $i$.

We fix $1 \le i \le r$. Let  $K \subset E_i \subset L$ 
denote the intermediate
extension such that $\Gal(E_i/K) = G_i$.
We have to then show that the composite map 
$p_i \circ \rho_{K,L} = \rho_{K,E_i}$ annihilates some open subgroup of 
$K^M_n(K)$.
We have therefore reduced the proof of the proposition to showing that
for every finite cyclic extension $L/K$,
the composite map $\rho_{K,L} \colon K^M_n(K) \to \Gal(L/K)$
annihilates some open subgroup of $K^M_n(K)$.
Using \propref{prop:Kato-**}(1), it suffices to show 
that $N_{L/K}(K^M_n(L))$ contains some open subgroup of
$K^M_n(K)$. But this follows from \corref{cor:Milnor-norm-open}.
\end{proof}

\subsection{Reciprocity for $C_{U/X}$}\label{sec:Rec-U}
Let $k$ be a finite field. In this subsection, we shall
assume all schemes to be endowed with their canonical dimension functions.
Let $X \in \Sch_k$ be an integral scheme of 
dimension $d$.
Let $C \subset X$ be a nowhere dense reduced 
closed subscheme with complement $U$. Let $\eta$ denote the
generic point of $X$ and let $K = k(\eta)$.  
%Let $\ov{K}$ denote a fixed separable closure of $K$.

We shall define a continuous reciprocity map
from $C_{U/X}$ to $\pi^{\ab}_1(U)$.
We begin with a more general construction.
We let $P = (p_0, \ldots , p_s)$ be any Parshin chain on $X$ with
the condition that $p_s \in U$.
Let $X(P) = \ov{\{p_s\}}$ be the integral closed subscheme of $X$
and let $U(P) = U \cap X(P)$. We let $X(P)_n$ denote the normalization of
$X(P)$ and let $f^P \colon X(P)_n \to X$ denote the canonical map.
Let $U(P)_n = (f^P)^{-1}(U)$. We then get canonical maps of pairs
\[
({X(P)}_n, {U(P)}_n) \xrightarrow{f^P} (X, U) \ \mbox{and} \
(X(P), U(P)) \xrightarrow{\iota^P} (X, U),
\]
such that $f^P$ factors through $\iota^P$. Note that 
${U(P)}_n \neq \emptyset$ since $p_s \in U$.

Now, $P$ is a maximal
Parshin chain on $X(P)$ whose generic point lies in $U(P)$.
Moreover, \lemref{lem:Inv-ambient} says that there is a factorization 
$\Spec(k(P)) \to X(P) \xrightarrow{\iota^P} X$, 
where the first map is dominant. 
On the other hand, \lemref{lem:P-valuation} says that the map
$\Spec(k(P)) \to X(P)$ factors through $\Spec(k(P)) \to X(P)_n$.
We thus have the factorization $\Spec(k(P)) \to X(P)_n \xrightarrow{f^P} X$.
We can therefore consider the push-forward maps
\begin{equation}\label{eqn:Rec-0}
\pi^{\ab}_1(\Spec(k(P))) 
\to \pi^{\ab}_1(U(P)_n) \to \pi^{\ab}_1(U(P)) 
\xrightarrow{\iota^P_*} \pi^{\ab}_1(U),
\end{equation}
where the composition of the last two arrows is $f^P_*$.

We now let $V \subset k(p_s)$ be an $s$-DV dominating $P$.
We have seen in \S~\ref{sec:valuation} that $Q(V^h)$ is an
$s$-dimensional Henselian local field. Hence, we have 
the reciprocity map
\begin{equation}\label{eqn:Kato-rec}
\rho_{Q(V^h)} \colon K^M_s(Q(V^h)) \to \Gal({{Q(V^h)}^{\ab}}/{Q(V^h)})
\cong \pi^{\ab}_1(\Spec(Q(V^h))).
\end{equation}
Taking the sum of these maps over $\sV(P)$ and using 
\lemref{lem:P-valuation}, we get a reciprocity map
\begin{equation}\label{eqn:Kato-rec-0}
\rho_{k(P)} \colon K^M_s(k(P)) \to \pi^{\ab}_1(\Spec(k(P))).
\end{equation}
Composing this with ~\eqref{eqn:Rec-0}, we get the following.

\begin{lem}\label{lem:Kato-rec-1}
For every Parshin chain $P$ on $X$ whose generic point lies in $U$, there
is a reciprocity map
\begin{equation}\label{eqn:Kato-rec-2}
\rho_{{k(P)}/U} \colon K^M_{d_X(P)}(k(P)) \to \pi^{\ab}_1(U),
\end{equation}
which factors through the reciprocity map
$K^M_{d_X(P)}(k(P)) \xrightarrow{\rho_{{{k(P)}}/{U(P)_n}}} \pi^{\ab}_1(U(P)_n)$.
\end{lem}

We let $\wt{I}_{U/X}$ denote the direct sum of the
groups $K^M_{d_X(P)}(k(P))$, where $P$ runs through the set of
all Parshin chains on $X$ whose generic points lie in $U$.
Taking the sum of the maps in ~\eqref{eqn:Kato-rec-2},
we get a map
\begin{equation}\label{eqn:Kato-rec-3}
\wt{\rho}_{U/X} \colon \wt{I}_{U/X} \to \pi^{\ab}_1(U).
\end{equation}

The restriction of $\wt{\rho}_{U/X}$ to $I_{U/X}$ yields a map
\begin{equation}\label{eqn:Kato-rec-4}
\rho_{U/X} \colon I_{U/X} \to 
\pi^{\ab}_1(U).
\end{equation}

We shall now show that $\rho_{U/X}$ factors through the idele class group.
Before this, we prove some lemmas.

\begin{lem}\label{lem:Gen-commute}
Let $f \colon X' \to X$ be a finite dominant morphism between 
integral schemes in $\Sch_k$. Let $U \subset X$ be a dense open 
subscheme and let $U' \subset f^{-1}(U)$. Let 
$P_1 = (p'_0, \ldots , p'_s)$ be a Parshin chain on $X'$ such that 
$p'_s \in U'$ and let $P = f(P_1)$. Then the diagram
\begin{equation}\label{eqn:Gen-commute-0}  
\xymatrix@C.8pc{
K^M_s(k(P_1)) \ar[r]^-{\rho_{k(P_1)}} \ar[d]_-{N} & \pi^{\ab}_1(U') 
\ar[d]^-{f_*} \\
K^M_s(k(P))  \ar[r]^-{\rho_{k(P)}} & \pi^{\ab}_1(U)}
\end{equation}
is commutative.
\end{lem}
\begin{proof}
Let $P = (p_0, \ldots , p_s)$, where $p_i = f(p'_i)$.
Let $Y$ (resp. $Y'$) be the closure of $\{p_s\}$ (resp. $\{p'_s\}$) in
$X$ (resp. $X'$) with integral closed subscheme structure.

By ~\eqref{eqn:Kato-rec-2}, $\rho_{{k(P)}/U}$ has a factorization
$K^M_s(k(P)) \to \pi^{\ab}_1(Y \cap U) \to \pi^{\ab}_1(U)$ and
similarly for $\rho_{{k(P_1)}/U'}$.
We thus have to show that the outer rectangle in 
\begin{equation}\label{eqn:Gen-commute-1}  
\xymatrix@C.8pc{
K^M_s(k(P_1)) \ar[r] \ar[d] & \pi^{\ab}_1(Y' \cap U') \ar[r] \ar[d] &
\pi^{\ab}_1(U') \ar[d] \\
K^M_s(k(P)) \ar[r] & \pi^{\ab}_1(Y \cap U) \ar[r] & \pi^{\ab}_1(U)}
\end{equation}
is commutative.
Since  the right square in \eqref{eqn:Gen-commute-1} commutes,
it suffices to show that the left square commutes.
Since the induced map $f: Y' \to Y $ is also finite dominant, we can 
assume furthermore that the Parshin chain $P_1$ (resp. $P$) is maximal on $X'$ 
(resp. $X$). 
Let $K'$ (resp. $K$) denote the function field of $X'$ (resp. $X$). 
Let $V' \subset K'$ be a $d$-DV dominating $P_1$ and let
$V' = V'_0 \subset V'_1 \subset \cdots \subset V'_{d-1} \subset V'_d = K'$
be the induced chain of valuation rings in $K'$.
Let $V = V'\cap K$ and $V_i = V'_i \cap K$. Then it is well known that
$V = V_0 \subset V_1 \subset \cdots \subset V_{d-1} 
\subset V_d = K$ is a $d$-DV in $K$ dominating $P$.
Furthermore, $Q(V'^h)$ is a finite separable extension of $Q(V^h)$.
This uses the fact that any finite algebra over $V^h$ is a
product of Henselian local finite algebras.
We write $V = f_*(V')$.

Let $V' \subset K'$ be a $d$-DV dominating $P_1$ and let $V = f_*(V)
\subset K$. By the definition of the reciprocity map in
~\eqref{eqn:Kato-rec-0}, we need to show that the outer
rectangle in the diagram
\begin{equation}\label{eqn:Normal-commute-1}  
\xymatrix@C.8pc{
K^M_d(Q(V'^h)) \ar[r]^-{\rho_{Q(V'^h)}} \ar[d]_-{N} & \pi^{\ab}_1(Q(V'^h)) 
\ar[d]^-{f_*} \ar[r] & \pi^{\ab}_1(U') \ar[d]^-{f_*} \\
K^M_d(Q(V^h)) \ar[r]^-{\rho_{Q(V^h)}} & \pi^{\ab}_1(Q(V^h)) \ar[r] & \pi^{\ab}_1(U)}
\end{equation}
is commutative.
Since the right square is clearly commutative, we only have to show that the
left square commutes. But this follows from \propref{prop:Kato-**}.
\end{proof}

\begin{lem}\label{lem:Via-residue}
Let $X$ be a normal integral scheme in $\Sch_k$ of dimension $d \ge 2$ 
and let $U \subset X$ be a dense open subscheme. Let 
$P = (p_0, \ldots , p_{d})$ be a maximal Parshin chain on $X$ such that $p_{d-1} 
\in U$. Then there is a commutative diagram
\begin{equation}\label{eqn:Via-residue-01}
\xymatrix@C1pc{
K^M_d(k(P)) \ar[dr]_-{\rho_{k(P)/U}} \ar[r]^-{\partial} &  
K^M_{d-1}(k(P')) \ar[d]^-{\rho_{k(P')/U}} \\
& \pi^{\ab}_1(U).}
\end{equation}
\end{lem}
\begin{proof}
Let $K$ be the function field of $X$.
Let $Y \subset X$ be the closure  of $p_{d-1}$ with the integral
closed subscheme structure.  
Let $f^P \colon \Spec(k(P)) \to X$ and $f^{P'} \colon \Spec(k(P')) \to Y$
be the canonical maps. Let $\iota \colon Y \cap U \inj U$ be the
inclusion.

Let us now pick a $d$-DV $V \in \sV(P)$ and let 
$V = V_0 \subset \cdots \subset V_{d-1} \subset V_d = K$ be the
corresponding chain of valuation rings in $K$. Let $V'$ be the image of
$V$ in $k_{d-1}$, where $k_i$ is the residue field of $V_i$.
 Since $X$ is normal, we have that $V_{d-1} = \sO_{X, p_{d-1}} = \sO_{U, p_{d-1}}$,
where the second equality holds because $p_{d-1} \in U$.
It follows that $k(Y)  =k_{d-1}$ and  $V'$ is a $(d-1)$-DV in $k(Y)$ dominating $P'$.

In the notations of \S~\ref{sec:valuation}, $Q(V^h)$ is
a Henselian discrete valuation field with ring of
integers $\wt{V}_{d-1}$ and residue field $Q({V'}^h)$
(see the proof of \cite[Proposition~3.3]{Kato-Saito-2}).
Moreover, there is a commutative diagram of schemes
\begin{equation}\label{eqn:Via-residue-02}  
\xymatrix@C.8pc{
\Spec(Q(V^h)) \ar@{^{(}->}[r] \ar[d] & \Spec(\wt{V}_{d-1}) \ar[d] &
\Spec(Q({V'}^h)) \ar[r] \ar@{^{(}->}[l] & \Spec(k(Y)) \ar@{^{(}->}[r] & 
Y \cap U \ar@{^{(}->}[d]^-{\iota_*} \\
\Spec(K) \ar@{^{(}->}[r] & \Spec(V_{d-1}) \ar[rrr] & & & U.}
\end{equation} 

This gives rise to the corresponding commutative diagram of the
abelianized {\'e}tale fundamental groups
\begin{equation}\label{eqn:Via-residue-03}  
\xymatrix@C.8pc{
& & & \pi^{\ab}_1(k(P')) \ar[dr]^-{f^{P'}_*} & \\
& \pi^{\ab}_1(\wt{V}_{d-1}) \ar[dr] &
\pi^{\ab}_1(Q({V'}^h)) \ar[r] \ar[l]_-{\cong} 
\ar@{^{(}->}[ur] & \pi^{\ab}_1(k(Y)) \ar@{->>}[r] & 
\pi^{\ab}_1(Y \cap U) \ar[d]^-{\iota_*} \\
\pi^{\ab}_1(Q(V^h)) \ar@{->>}[ur] \ar[r] \ar@{^{(}->}[drr] &
\pi^{\ab}_1(K) \ar@{->>}[r] & \pi^{\ab}_1(V_{d-1}) \ar@{->>}[rr] & & 
\pi^{\ab}_1(U) \\
& & \pi^{\ab}_1(k(P)) \ar[urr]_-{f^{P}_*}, & &}
\end{equation} 
where we have abbreviated the notation $\pi^{\ab}_1(\Spec(A))$ to 
$\pi^{\ab}_1(A)$ for a ring $A$. We let $\iota' \colon
\Spec(Q({V'}^h)) \inj \Spec(\wt{V}_{d-1})$ be the inclusion and let
$\partial'$ be the composition
$\pi^{\ab}_1(Q(V^h)) \surj  \pi^{\ab}_1(\wt{V}_{d-1}) 
\xrightarrow{(\iota'_*)^{-1}} \pi^{\ab}_1(Q({V'}^h))$.
Note that $\partial'$ coincides (by definition) 
with the right vertical arrow in
~\eqref{eqn:Kato-**-2} (if we let $K = Q(V^h)$ in that diagram).
The upshot of ~\eqref{eqn:Via-residue-03} 
is that $f^P_*$ has a factorization
\begin{equation}\label{eqn:Via-residue-04}  
\xymatrix@C.8pc{
\pi^{\ab}_1(Q(V^h)) \ar[r]^-{\partial'} \ar[dr]_-{f^P_*} &
\pi^{\ab}_1(Q({V'}^h)) \ar[d]^-{(\iota \circ f^{P'})_*} \\
& \pi^{\ab}_1(U).}
\end{equation}

On the other hand, it follows from \propref{prop:Kato-**} that 
there is a commutative diagram

\begin{equation}\label{eqn:Rec-Q-1}
\xymatrix@C1pc{
K^M_d(Q(V^h)) \ar[r]^-{\rho_{Q(V^h)}} \ar@{->>}[d]_-{\partial} & 
\pi^{\ab}_1(Q(V^h)) \ar@{->>}[d]^-{\partial'} \\
K^M_{d-1}(Q(V'^h)) \ar[r]^-{\rho_{Q(V'^h)}} & \pi^{\ab}_1(Q(V'^h)).}
\end{equation}

Since $f^P_* \circ \rho_{Q(V^h)}$ is the restriction of $\rho_{{k(P)}/U}$
on the factor $K^M_d(Q(V^h))$ and $(\iota \circ f^{P'})_* \circ \rho_{Q(V'^h)}$
is the restriction of $\rho_{{k(P')}/U}$ on the factor
$K^M_{d-1}(Q(V'^h))$, the lemma follows by combining
~\eqref{eqn:Via-residue-04} and ~\eqref{eqn:Rec-Q-1} and then summing
over $\sV(P)$.
\end{proof}

We can now prove the main result of this section.

\begin{thm}\label{thm:Rec-main}
Let $X$ be an integral scheme in $\Sch_k$ of dimension $d \ge 1$ 
and let $U \subset X$ be a dense open subscheme.
Then the map $\rho_{U/X}$ on the idele group descends to a reciprocity map
\[
\rho_{U/X} \colon C_{U/X} \to \pi^{\ab}_1(U).
\]
\end{thm}
\begin{proof}
We let $Q = (p_0, \ldots , p_{s-2}, p_s)$ be a $Q$-chain on
$(U \subset X)$. 
We let $X' = \ov{\{p_s\}} \subset X$ with the integral closed subscheme 
structure, and set $U' = U \cap X'$. Note that $U' \neq \emptyset$
because $p_s \in U$. Let $f \colon (X', U') \to (X,U)$ be the
inclusion.

It follows from the definition of $\rho_{U/X}$ (see ~\eqref{eqn:Rec-0})
that there is a commutative diagram
\begin{equation}\label{eqn:Rec-PF}
\xymatrix@C.8pc{
I_{U'/X'} \ar[r]^-{\rho_{U'/X'}} \ar[d]_-{f_*} & \pi^{\ab}_1(U') 
\ar[d]^-{f_*} \\
I_{U/X} \ar[r]^-{\rho_{U/X}} & \pi^{\ab}_1(U),}
\end{equation}
where the vertical arrow on the left is from \propref{prop:Funtorial*}.
Since $Q$ is a $Q$-chain on $X'$ and the map
$\partial \colon K^M_{d_X(Q)}(k(Q)) \to I_{U/X}$ factors through
$K^M_{d_X(Q)}(k(Q)) \to I_{U'/X'}$ by \lemref{lem:Inv-ambient},
it suffices to show that $\rho_{U'/X'}$ annihilates the image of
$K^M_{d_X(Q)}(k(Q)) = K^M_{d_{X'}(Q)}(k(Q))$.
We can therefore assume that $Q$ is maximal.

Let $\nu \colon X_n \to X$ be the normalization map and let 
$\nu^{-1}(Q)$ be the set of all $Q$-chains $Q'$ on $X_n$ such that
$\nu(Q') = Q$. We now consider the diagram
\begin{equation}\label{eqn:Rec-main-ex}
\xymatrix@C.8pc{
{{\underset{Q' \in \nu^{-1}(Q)}\bigoplus} K^M_{d}(k(Q'))}
\ar[r]^-{\partial} \ar[d]_-{\cong} & I_{{U_n}/{X_n}} 
\ar[r]^-{{\rho}_{{U_n}/{X_n}}} \ar[d]^-{\nu_*} & \pi^{\ab}_1(U_n)
\ar[d]^-{\nu_*} \\
K^M_d(k(Q)) \ar[r]^-{\partial} &  I_{{U}/{X}} 
\ar[r]^-{{\rho}_{{U}/{X}}} & \pi^{\ab}_1(U).}
\end{equation}

We have seen in the proof of \propref{prop:Funtorial*} that the
left square is commutative and the left vertical arrow is an
isomorphism by \cite[Lemma~3.3.1]{Kato-Saito-2} 
(see \lemref{lem:Q-chain-nor}). \lemref{lem:Gen-commute} says that
the right square is commutative. It suffices therefore to show that the
top composite arrow in ~\eqref{eqn:Rec-main-ex} is zero. This allows us to
assume that $X$ is normal (and $Q$ is maximal).

We first assume that $d \ge 2$.
Let $B(Q)$ denote the set of all Parshin chains
of the form $P = (p_0, \ldots , p_{d-2}, p_{d-1}, p_d)$ on $X$.
For every $P \in B(Q)$, let $\iota_{Q,P} \colon
K^M_d(k(Q)) \to K^M_d(k(P))$ be the
map induced by the inclusion $k(Q) \inj k(P)$.
It is clear that for any $\alpha \in K^M_d(k(Q))$,
the element $\iota_{Q,P}(\alpha)$ lies in $K^M_d(\sO^h_{X,P'})
\subset K^M_d(k(P))$ for all but finitely many $P$'s.
Letting
\[
J_Q = \{(a_P)_{P \in B(Q)} \in {\underset{P \in B(Q)}\prod} K^M_{d}(k(P))|
a_P \in K^M_d(\sO^h_{X,P'}) \ \mbox{for \ almost \ all} \
P\},
\]
we thus get a canonical map
\begin{equation}\label{eqn:Rec-Q-0}
\wt{\partial} \colon K^M_d(k(Q)) \to J_Q,
\end{equation}
whose composition with the projection to any $K^M_d(k(P))$ is $\iota_{Q,P}$.
We let
\[
J'_Q = \{(a_P)_{P \in B(Q)} \in {\underset{P \in B(Q), p_{d-1} \in U}\prod} 
K^M_{d}(k(P))|
a_P \in K^M_d(\sO^h_{X,P'}) \ \mbox{for \ almost \ all} \
P\},
\]
\[
J''_Q = {\underset{P \in B(Q), p_{d-1} \in C}\prod} 
K^M_{d}(k(P)) \ \mbox{and} \ 
R_Q  =  {\underset{P \in B(Q), p_{d-1} \in U}\bigoplus} 
K^M_{d-1}(k(P')). 
\]
It is then clear that the projection map
$J_Q \to J'_Q \times J''_Q$ is an isomorphism.

Since $X$ is normal, it follows from \lemref{lem:Via-residue}
that there is a reciprocity map $\rho_{{B(Q)}/L} \colon J_Q \to \pi^{\ab}_1(U)$
and the composite $J'_Q \inj J_Q \xrightarrow{\rho_{{B(Q)}/L}} \pi^{\ab}_1(U)$
factors through $R_Q$.

We thus get a commutative diagram
\begin{equation}\label{eqn:Rec-Q-2}
\xymatrix@C.8pc{
K^M_{d}(k(Q)) \ar[r]^-{\wt{\partial}} \ar@{=}[d] & 
J_Q \ar[r]^-{{\rho}_{{B(Q)}/L}} \ar[d]^-{\cong} & 
\pi^{\ab}_1(U) \ar@{=}[d] \\
K^M_{d}(k(Q)) \ar[r]^-{\wt{\partial}} \ar@{=}[d] & 
J''_Q \times J'_Q \ar[r]^-{{\rho}_{{B(Q)}/L}} \ar[d]^-{{\rm id} \times
\partial} & \pi^{\ab}_1(U) \ar@{=}[d] \\
K^M_{d}(k(Q)) \ar[r]^-{\partial_{U/X}} &
J''_Q \times R_Q \ar[r]^-{\rho_{U/X}} & \pi^{\ab}_1(U) ,}
\end{equation}
where $\partial_{U/X}$ is from ~\eqref{eqn:IC-0}.
It suffices therefore to show that ${\rho}_{{B(Q)}/L} \circ
\wt{\partial} = 0$.
But this is shown in 
\cite[\S~3.7.4]{Kato-Saito-2} (see the third paragraph on
page~281). Note that the map $\wt{\partial}$ and
${\rho}_{{B(Q)}/L}$ are same as the maps considered in 
\cite[\S~3.7.4]{Kato-Saito-2} because $Q$ is a maximal chain.
This proves the desired assertion and finishes the proof of the 
proposition when $d \ge 2$.

If $d = 1$, we can use a 1-dimensional variant of \lemref{lem:Q-chain-nor}
(whose proof is straightforward) to assume again that $X$ is normal. 
The identity $\rho_{U/X} \circ \partial = 0$ is then classical.
%Kerz-Schmidt covering data paper has a proof.
\end{proof}

\begin{remk}\label{remk:KS-Kerz-rec}
If we assume $U$ to be regular in \thmref{thm:Rec-main}, then
we can apply \thmref{thm:Kerz-global}. In this case, a reciprocity map 
$\wt{C}_{U/X} \to \pi^{\ab}_1(U)$ was constructed by Kato-Saito
\cite[\S~3]{Kato-Saito-2}. By composing with $C_{U/X} \to \wt{C}_{U/X}$,
we obtain another reciprocity map
$\rho^{KS}_{U/X} \colon C_{U/X} \to  \pi^{\ab}_1(U)$.
We shall show in \lemref{lem:Agreement*} that this map agrees with
$\rho_{U/X}$ of \thmref{thm:Rec-main}.
\end{remk}

Our next goal is to show that the reciprocity map is continuous.

\begin{prop}\label{prop:Rec-main-cont}
The map $\rho_{U/X} \colon C_{U/X} \to \pi^{\ab}_1(U)$ is continuous with
respect to the profinite topology on $\pi^{\ab}_1(U)$.
\end{prop}
\begin{proof}
Since $C_{U/X}$ has the quotient topology induced by the direct sum
topology of $I_{U/X}$, it suffices to show using \lemref{lem:DST-univ}
that for every Parshin chain $P$ on $(U \subset X)$, the map
$K^M_{d_X(P)}(k(P)) \to \pi^{\ab}_1(U)$ is continuous.
Since $Y \mapsto \pi^{\ab}_1(Y)$ is a functor from schemes to
profinite topological groups together with continuous homomorphisms
(e.g., see \cite[Remark~5.5.3]{Szamuely}), it suffices to show
using ~\eqref{eqn:Rec-0} that the map
$K^M_d(F) \to \Gal(F^{\ab}/F)$ is continuous for every $d$-dimensional
Henselian local field $F$. But this follows from \propref{prop:Kato-rec-cont}.
\end{proof}

\subsection{Reciprocity map for $C_{K/X}$}\label{sec:Rec-K}
Using \thmref{thm:Rec-main}, we can define the reciprocity for the
function field of $X$ as follows.
Recall that $K$ denotes the function field of $X$.
Since the map $\rho_{U/X} \colon C_{U/X} \to \pi^{\ab}_1(U)$ is clearly
compatible with the inclusions $U' \subset U$ of open subsets of $X$,
taking the projective limit of $\rho_{U/X}$ over all open subsets of $X$,
we obtain the following reciprocity for $K$.

\begin{cor}\label{cor:Rec-K-main}
There exists a continuous reciprocity homomorphism
\[
\rho_{K/X} \colon C_{K/X} \to \Gal(K^{\ab}/K).
\]
\end{cor}  

Let $\wt{C}_{K/X} = {\underset{U \subset X}\varprojlim}
\wt{C}_{U/X} = {\underset{D \subset X}\varprojlim} C(X,D)$,
where the first limit is over all open subschemes of $X$ and the second is
over all closed subschemes of $X$.
We then get a sequence of maps
\[
C_{K/X} \to \wt{C}_{K/X} \to \wt{C}_{U/X} \to  C(X,D),
\]
where all except the first arrow are surjective if we let $U = X \setminus D$.
Furthermore, there is a factorization
\begin{equation}\label{eqn:Rec-K-main-0}
C_{K/X} \to \wt{C}_{K/X} \xrightarrow{\wt{\rho}_{K/X}} \Gal(K^{\ab}/K)
\end{equation}
of $\rho_{K/X}$. Moreover, $\wt{\rho}_{K/X}$ is continuous by
\lemref{lem:Disc}. It is also easy to see that $\rho_{K/X}$ coincides
with the classical reciprocity for the function field of a curve if $d =1$.

\begin{remk}\label{remk:Independence}
It will be interesting to know if $C_{K/X}$ and $\rho_{K/X}$ are independent
of $X$ when $d \ge 2$.
\end{remk}

\section{Ramification filtrations}\label{sec:Filt*}
In this section, we recall some filtrations on the absolute Galois
group of a Henselian discrete valuation field and prove some properties of 
these filtrations. The set-up is the following.

Let $K$ be a Henselian discrete valuation field of 
characteristic $p > 0$ with ring of integers
$\sO_K$, maximal ideal $\fm_K \neq 0$ and residue field $\ff$.
Note that our assumption implies that $K$ is infinite.
%We fix a separable closure $\ov{K}$ of $K$ and 
%assume all separable algebraic extensions of $K$ to be
%contained in $\ov{K}$. 
Let $K^{\ab}$ denote the maximal
abelian extension of $K$ in $\ov{K}$. 
Let $G_K$ denote the absolute Galois group of $K$ and 
$G^{\ab}_K$ its abelianization. 
Let $K^{\tr}$ denote the maximal abelian extension of $K$ which is tamely
ramified (the inertia group of every finite subextension 
has order prime to the residue characteristic). 
We shall write $H^{q}(A) := H^{q}_\et(X, {\Q}/{\Z}(q-1))$ for $q \ge 1$
and Noetherian scheme $X = \Spec(A)$.
In particular, $H^1(K)$  is the Pontryagin dual of the 
profinite group $G_K$.

\subsection{The Abbes-Saito filtration}\label{sec:AS-0**}
By the work of Abbes-Saito \cite{Abbes-Saito} (see \cite[Theorem~3.1]{Saito20}
for property (3) below), there is
an upper numbering non-logarithmic decreasing filtration
$\{G^{(r)}_K\}_{r \in \Q_{\ge -1}}$ of $G_K$ by closed normal subgroups
which satisfies the following.
\begin{enumerate}
\item
$G^{(-1)}_K = G_K$.
\item
$G^{(0)}_K$ is the inertia subgroup $\Gal({\ov{K}}/{K^{\nr}})$ of $G_K$.
\item
$G^{(1)}_K = \Gal({\ov{K}}/{K^{\tr}})$ is the wild inertia subgroup of $G_K$.
\item 
If $\ff$ is perfect, then $G^{(i)}_K$ coincides with $G^{(i)}_{K, \rm cl}$
for every $i \ge 0$, where the latter is the classical upper numbering 
ramification filtration of $G_K$ \cite[Chapter~IV]{Serre-LF}.
\end{enumerate}
We note here that we have shifted the original filtration of \cite{Abbes-Saito}
by one to the left in order to make it compatible 
with the classical filtration. 

For any integer $n \ge -1$, we let
$\Fil^{\as}_n H^{1}(K) = \{\chi \in H^{1}(K)| \chi(a) = 0 \ \forall \
a \in G^{(n)}_K\}$.
This yields an increasing filtration
\begin{equation}\label{eqn:AS-0}
0 = \Fil^{\as}_{-1} H^{1}(K) \subset \Fil^{\as}_0 H^{1}(K) \subset 
\Fil^{\as}_1 H^{1}(K) \subset \cdots
\end{equation}
of $H^{1}(K)$. We shall denote this filtration by $\Fil^{\as}_\bullet H^{1}(K)$. 
Clearly, we have $\Fil^{\as}_0 H^{1}(K) = H^{1}(\sO_K)$ is the subgroup
of unramified characters and $\Fil^{\as}_1 H^{1}(K)$ is the subgroup of
tamely ramified characters.

\subsection{The Matsuda filtration}\label{sec:Mat-0}
For a Noetherian $K$-scheme $X$, let $\{W_r\Omega^{\bullet}_X\}_{r \ge 1}$
be the de Rham-Witt complex recalled in \S~\ref{sec:KR}.
For any integers $r \ge 1$ and $q \ge 0$, there is a commutative diagram of 
short exact sequences of {\'e}tale sheaves
\begin{equation}\label{eqn:DRC}
\xymatrix@C.8pc{
0 \ar[r] & W_r\Omega^q_{X, \log} \ar[r] \ar[d]_-{\un{p}} &
W_{r}\Omega^{q}_X \ar[d]_-{V} \ar[r]^-{1 - F} & 
\frac{W_r\Omega^{q}_X}{d V^{r-1}\Omega^{q-1}_X} \ar[d]^-{V} \ar[r] & 0 \\
0 \ar[r] & W_{r+1}\Omega^q_{X, \log} \ar[r] &
W_{r+1}\Omega^{q}_X \ar[r]^-{1 - F} & 
\frac{W_{r+1}\Omega^{q}_X}{d V^{r}\Omega^{q-1}_X} \ar[r] & 0.}
\end{equation}

Suppose now that $X = \Spec(K)$. Then $W_{r}\Omega^{q}_X$ is acyclic and
the long cohomology exact sequence associated to the above is of the
form
\begin{equation}\label{eqn:DRC-0}
0 \to  H^{q}(K, {\Z}/{p^r}(q)) \to W_r\Omega^q_K 
\xrightarrow{1 - F} \frac{W_r\Omega^{q}_K}{d V^{r-1}\Omega^{q-1}_K} 
\xrightarrow{\partial} H^{q+1}(K, {\Z}/{p^r}(q)) \to 0.
\end{equation}
Note that we have used here the fact that the kernel
$B_{r-1}\Omega^q_K$ of $V^{r-1} \colon \Omega^q_K \to W_{r}\Omega^q_K$
is a $K$-module and the differential 
$d \colon W_r\Omega^q_K \to W_{r}\Omega^{q+1}_K$ is $W(K)$-linear via the
Frobenius.

Let $\pi_r \colon W_r\Omega^q_K \surj 
\frac{W_r\Omega^{q}_K}{d V^{r-1}\Omega^{q-1}_K}$ denote the quotient map,
and let $\tau_r \colon H^{q+1}(K, {\Z}/{p^r}(q)) \to H^{q+1}(K)$ denote the 
canonical map into the inductive limit via the structure map of
the projective system $\{(W_r\Omega^q_{K, \log}, \un{p})\}_{r \ge 0}$.
We then get the maps
\begin{equation}\label{eqn:DRC-1}
W_r\Omega^q_K \xrightarrow{\partial \circ \pi_r} 
H^{q+1}(K, {\Z}/{p^r}(q)) \xrightarrow{\tau_r} H^{q+1}(K).
\end{equation}
We denote both $\partial \circ \pi_r$ and $\tau_r \circ
\partial \circ \pi_r$ by $\delta_r$.
It follows from ~\eqref{eqn:DRC} that the diagram
\begin{equation}\label{eqn:DRC-2}
\xymatrix@C.8pc{
W_r\Omega^q_K \ar[r]^-{V} \ar[dr]_-{\delta_r} & W_{r+1}\Omega^q_K 
\ar[d]^-{\delta_{r+1}} \\
& H^{q+1}(K)}
\end{equation}
is commutative.

We now specialize to the case $q = 0$. Then
~\eqref{eqn:DRC-0} is of the form
\begin{equation}\label{eqn:DRC-3}
0 \to {\Z}/{p^r} \to W_r(K) \xrightarrow{1 - F} W_r(K)
\xrightarrow{\partial} H^1_{\et}(K, {\Z}/{p^r}) \to 0,
\end{equation}
where $F((a_{r-1}, \ldots , a_{0})) = (a^p_{r-1}, \ldots , a^p_{0})$.
We write the Witt vector $(a_{r-1}, \ldots , a_{0})$ as $\un{a}$ in short.
Let $v_K \colon K^{\times} \to \Z$ be the normalized valuation.
For an integer $m \ge 1$, let
${\rm ord}_p(m)$ denote the $p$-adic order of $m$ and
let $r' = \min(r, {\rm ord}_p(m))$. We let ${\rm ord}_p(0) = - \infty$.

For $m \ge 0$, we let
\begin{equation}\label{eqn:DRC-4}
\Fil^{\bk}_m W_r(K) = \{\un{a}| p^{i}v_K(a_i) \ge - m\}; \ \mbox{and}
\end{equation}
\begin{equation}\label{eqn:DRC-5}
\Fil^{\ms}_m W_r(K) = \Fil^{\bk}_{m-1} W_r(K) + 
V^{r-r'}(\Fil^{\bk}_{m} W_{r'}(K)).
\end{equation}
Clearly, $\Fil^{\bk}_0 W_r(K) = \Fil^{\ms}_0 W_r(K) = W_r(\sO_K)$.
We let $\Fil^{\bk}_{-1} W_r(K) = \Fil^{\ms}_{-1} W_r(K) = 0$.

%the following definition is changed.
%We let
%\begin{equation}\label{eqn:DRC-6}
%\Fil^{\bk}_m H^1(K) = H^1(K)\{p'\} \bigoplus {\underset{r \ge 1}\bigcup}
%\delta_r(\Fil^{\bk}_m W_r(K)) \ \ \mbox{and}
%\end{equation}
%\[
%\Fil^{\ms}_m H^1(K) = H^1(K)\{p'\} \bigoplus {\underset{r \ge 1}\bigcup}
%\delta_r(\Fil^{\ms}_m W_r(K)).
%\]

For $m \geq 1$, we let
\begin{equation}\label{eqn:DRC-6}
\Fil^{\bk}_{m-1} H^1(K) = H^1(K)\{p'\} \bigoplus {\underset{r \ge 1}\bigcup}
\delta_r(\Fil^{\bk}_{m-1} W_r(K)); \ \ \mbox{and}
\end{equation}
\[
\Fil^{\ms}_m H^1(K) = H^1(K)\{p'\} \bigoplus {\underset{r \ge 1}\bigcup}
\delta_r(\Fil^{\ms}_m W_r(K)).
\]
Moreover, we let 
$\Fil^{\ms}_0 H^1(K) = \Fil^{\bk}_{-1} H^1(K) = H^1(\sO_K)$,
 the subgroup
of unramified characters.
The filtrations $\Fil^{\bk}_\bullet H^1(K)$ and $\Fil^{\ms}_\bullet H^1(K)$
are due to Brylinski-Kato \cite{Kato89} and Matsuda \cite{Matsuda},
respectively. We shall use the following results related to various
filtrations of $H^1(K)$.

\begin{thm}\label{thm:Fil-main}
The three filtrations defined above satisfy the following relations.
\begin{enumerate}
\item
$\Fil^{\as}_m H^1(K) = \Fil^{\ms}_m H^1(K)$ for all $m \ge -1$. 
\item
$\Fil^{\ms}_m H^1(K) \subset  \Fil^{\bk}_m H^1(K) \subset \Fil^{\ms}_{m+1} H^1(K)$
for all $m \ge -1$. 
\item
$H^1(K) = {\underset{m \ge 0}\bigcup} \Fil^{\bk}_m H^1(K) 
= {\underset{m \ge 0}\bigcup} \Fil^{\ms}_m H^1(K) =
{\underset{m \ge 0}\bigcup} \Fil^{\as}_m H^1(K)$.
\end{enumerate}
\end{thm}
\begin{proof}
The first part is independently due to Abbes-Saito 
\cite[Th{\'e}or{\`e}me~9.10]{Abbes-Saito-1}
(for $m \ge 2$), Yatagawa \cite[Theorem~0.1]{Yatagawa} and
Saito \cite[Corollary~3.3]{Saito20}. The second part is clear. The first
equality in part (3) is due to Kato \cite[Lemma~2.2]{Kato89} and the
remaining ones follow from (1) and (2).
\end{proof}

\subsection{Filtrations of $H^1(K)$ via Milnor $K$-theory}
\label{sec:Fil-M}
Let $K$ be a Henselian discrete valuation field as above.
The above filtrations can be described in terms of Milnor $K$-theory
which we shall use. In order to explain this, 
we need to recall some filtrations of $K^M_*(K)$.
For any integers $m \ge 0$ and $n \ge 1$, let $U_m K^M_n(K)$ be
the subgroup $\{U_m(K), K^{\times}, \ldots , K^{\times}\}$ of $K^M_n(K)$.
We let $U'_mK^M_n(K)$ be the subgroup $\{U_m(K), \sO^{\times}_K, \ldots ,
\sO^{\times}_K\}$ of $K^M_n(K)$. 

\begin{lem}\label{lem:Milnor-K-filt}
For every $m \ge 0$ and $n \ge 1$, there are inclusions
\[
U_{m+1}K^M_n(K) \subseteq U'_mK^M_n(K) \subseteq U_{m}K^M_n(K).
\]
\end{lem}
\begin{proof}
We only need to prove the first inclusion. We shall prove it by induction
on $n$. Since this inclusion is obvious for $n =1$, we can assume that 
$n \ge 2$. Let $\pi$ be a uniformizer of $K$.
We let $\alpha = \{1 + u_1\pi^{m+1}, u_2 \pi^{m_2}, \ldots , u_n \pi^{m_n}\}$,
where $u_1, \ldots , u_n \in \sO^{\times}_K$ and $m_2, \ldots , m_n \in \Z$.

Using the bilinearity and additivity properties of the Milnor $K$-theory,
we can assume that $u_2 = \cdots = u_n = 1$. By the same token,
we can also assume $m_2 = \cdots = m_n = 1$.
We can therefore write $\alpha = \{1 + u_1\pi^{m+1}, \pi, \ldots , \pi\}$.
By \cite[Chapter 1, \S~1]{Bass-Tate}, we have the relation
$\{\pi, \pi\} = \{\pi, (-1)(-\pi)\} = \{\pi, -1\} + \{\pi,  - \pi\} =
\{\pi, -1\}$.
Applying this relation and another relation $\{a, b\} = - \{b,a\}$ repeatedly,
we see that $\alpha = \{1 + u_1\pi^{m+1}, \pi\}  \cdot \beta$, where
$\beta \in K^M_{n-2}(\sO_K)$. Since $U'_mK^M_2(K) \cdot K^M_{n-2}(\sO_K)
\subset U'_mK^M_n(K)$, we can assume $n = 2$ and
$\alpha = \{1 + u_1\pi^{m+1}, \pi\}$.

Now, we let $t = -(u_1\pi^m)$, $v' = (1 + t(-1- \pi))^{-1}$ and
$v'' = -1 - \pi$. It is then clear that
$1 + v't, 1 + v''t \in U_m(K)$ and $(1 + v't)(1 + v''t) = 1 - \pi t$.
We therefore get
\[
\begin{array}{lll}
\{1 + u_1\pi^{m+1}, \pi\} & = & \{1 + (u_1\pi^{m})\pi, \pi\} 
= \{1 + (u_1\pi^{m})\pi, -(u_1\pi^m)\} = \{1 - \pi t, t\} \\
& = & \{(1 + v't)(1 + v''t), t\} 
=  \{1 + v't, t\} +   \{1 + v''t, t\} \\
& = &  \{1 + v't, -v'\} + \{1 + v''t, - v''\} \\
& \in & U'_mK^M_2(K).
\end{array}
\]
The lemma is now proven.
\end{proof}

Let us now assume that $K$ is a $d$-dimensional Henselian local field.
Recall from ~\eqref{eqn:Pairing-2} that there is a cup product pairing
\begin{equation}\label{eqn:Pairing-local}
K^M_d(K) \times H^1(K) \xrightarrow{\cup} H^{d+1}(K); \ \
(\alpha, \beta) \mapsto \{\alpha, \beta\}.
\end{equation}

We have the following characterization of $\Fil^{\bk}_\bullet H^1(K)$ and 
$\Fil^{\ms}_\bullet H^1(K)$ in terms of Milnor $K$-theory.

\begin{thm}\label{thm:Filt-Milnor-*}
Let $\chi \in H^1(K)$ be a character. Then the following hold.
\begin{enumerate}
\item
For every integer $m \ge 0$, we have that 
$\chi \in \Fil^{\bk}_m H^1(K)$ if and only if $\{\alpha, \chi\} = 0$
for all $\alpha \in U_{m+1} K^M_d(K)$.
\item
For every integer $m \ge 1$, we have that $\chi \in \Fil^{\ms}_m H^1(K)$ if and 
only if $\{\alpha, \chi\} = 0$ for all $\alpha \in U'_{m} K^M_d(K)$.
\end{enumerate}
\end{thm}
\begin{proof}
The first part is due to Kato \cite[Proposition~6.5]{Kato89}.
If $p \neq 2$, the second part follows from 
\lemref{lem:Milnor-K-filt} and \cite[Proposition~3.2.7]{Matsuda}
(see also \cite[Lemma~2.6]{Kerz-Saito-2}). We give a proof which works
in all positive characteristics.

Let $m \ge 1$ and $\Fil^\dagger_m H^1(K)$ be the subgroup of
characters $\chi \in H^1(K)$ such that $\{\alpha, \chi\} = 0$ for all 
$\alpha \in U'_{m} K^M_d(K)$. First of all, it follows from
\lemref{lem:Milnor-K-filt} and part (1) of the theorem that

\begin{equation}\label{eqn:Filt-Milnor-*-0}
\Fil^{\bk}_{m-1} H^1(K) \subset \Fil^\dagger_m H^1(K) \subset \Fil^{\bk}_m H^1(K).
\end{equation}

We let $A$ denote the henselization of the polynomial ring $\sO_K[T]$ along
the prime ideal $\fm_K[T]$. We let $A_K = A \otimes_{\sO_K} K$ and
$A_{\ff} = A \otimes_{\sO_K} \ff$. For any integers $q, r \ge 1$,
let $V^q_{r}(A) = H^q_{\et}(A, i^* \circ {\bf R}j_*({\Z}/{p^r}(q-1)))$,
where $i \colon A \surj A_{\ff}$ and $j \colon A \to A_K$ are canonical
maps. 
%Note that $V^q_{r}(\sO_K) = H^q_r(K)$.
We let $V^q(A) = {\underset{r \ge 1}\varinjlim} V^q_r(A)$.
Since $(A, A_{\ff})$ is a Henselian pair, one checks that
$1 + \pi A \in (A_K)^{\times}$. Hence, $\{\chi, 1 + \pi A\}$ is
well defined in $V^2(A)$.

By \cite[\S~(1.9)]{Kato89}, there is a canonical map
$\iota^q_A \colon H^q_r(A_{\ff}) \to V^q_{r}(A)$. We consider the map
$\lambda^q_{\pi} \colon H^q_r(A_{\ff}) \bigoplus H^{q-1}_r(A_{\ff}) \to V^q_{r}(A)$
given by $\lambda^q_{\pi}(a, b) = \iota^q_A(a) + \{\iota^{q-1}_A(b), \pi\}$,
where $\pi \in \sO_K$ is our chosen uniformizer of $K$. The map
$\lambda^q_\pi$ is injective as Kato shows.
By using the exact sequence of sheaves given in ~\eqref{eqn:DRC}
for $A_{\ff}$ and using the associated cohomology groups, we get the
canonical maps
$\Omega^{i}_{A_{\ff}} \xrightarrow{V^{r-1}} W_r\Omega^{i}_{A_{\ff}} \to
H^{i+1}_r(A_{\ff})$. 
Taking the direct sum of these maps for $i = q-1, q-2$,
composing the resulting map with $\lambda^q_{\pi}$, and taking limit
as $r \mapsto \infty$, we thus get a canonical map
\begin{equation}\label{eqn:Filt-Milnor-*-1}
\ov{\lambda}^q_{\pi} \colon \Omega^{q-1}_{A_{\ff}} \bigoplus \Omega^{q-2}_{A_{\ff}}
\to V^q(A).
\end{equation}

It follows from \cite[Theorem~3.2(1)]{Kato89} that
$\Fil^{\bk}_m H^1(K)$ coincides with the subgroup $~'\Fil^{\bk}_m H^1(K)$ 
of $H^1(K)$ consisting of characters
$\chi$ such that $\{\chi, 1 + \pi^{m+1}T\} = 0$ in $V^{2}(A)$.
We let $~'\Fil^{\ms}_m H^1(K)$ be the subgroup of $H^1(K)$ consisting of 
characters $\chi$ such that $\{\chi, 1 + \pi^{m+1}T\} = 0$ 
and $\{\chi, 1 + \pi^{m}T\} = \ov{\lambda}^2_{\pi}(T\alpha_\chi, 0)$ in 
$V^{2}(A)$. It suffices to show for every $m \ge 1$ that

\begin{equation}\label{eqn:Filt-Milnor-*-2}
\Fil^\dagger_m H^1(K) = ~'\Fil^{\ms}_m H^1(K) = \Fil^{\ms}_m H^1(K).
\end{equation}

The proof of the first equality given in \cite[Proposition~3.2.7]{Matsuda}
is unconditional. So we shall prove the second equality.
We begin by showing that $\Fil^{\ms}_m H^1(K) \subseteq ~'\Fil^{\ms}_m H^1(K)$.
To show this, we let $m_0 = \ord_p(m)$ and $n = mp^{-m_0}$. 
Then  using ~\eqref{eqn:DRC-5}
and taking limit as $r \mapsto \infty$, we see that

\begin{equation}\label{eqn:Filt-Milnor-*-3}
\Fil^{\ms}_m H^1(K) = \Fil^{\bk}_{m-1} H^1(K) + 
\delta_{m_0}(\Fil^{\bk}_{m}(W_{m_0}(K))).
\end{equation}

Using the fact that every element $\un{a} \in W_{m_0}(K)$ is of the
form $\stackrel{m_0-1}{\underset{i = 0}\sum} V^{m_0-1-i}([a_i])$ and the
commutative diagram ~\eqref{eqn:DRC-2}, it follows by looking at
~\eqref{eqn:DRC-4} and ~\eqref{eqn:Filt-Milnor-*-3} that
$\frac{\Fil^{\ms}_m H^1(K)}{\Fil^{\bk}_{m-1} H^1(K)}$ is generated by
elements of the form $\delta_{j+1}([x\pi^{-mp^j}])$ with $x \in \sO_K$ and
$0 \le j \le  m_0 -1$.
If $\chi \in \Fil^{\bk}_{m-1} H^1(K) = ~'\Fil^{\bk}_{m-1} H^1(K)$, then
$\{\chi, 1 + \pi^mT\} = 0$. Since $~'\Fil^{\bk}_{m-1} H^1(K) \subseteq
~'\Fil^{\bk}_{m} H^1(K)$, we also have $\{\chi, 1 + \pi^{m+1}T\} = 0$.
We conclude from the definition of $~'\Fil^{\ms}_m H^1(K)$
that $\Fil^{\bk}_{m-1} H^1(K) \subseteq ~'\Fil^{\ms}_m H^1(K)$.
In order to show therefore that $\Fil^{\ms}_m H^1(K) \subseteq 
~'\Fil^{\ms}_m H^1(K)$, it suffices to show that
for $\chi = \delta_{j+1}([x\pi^{-mp^j}])$ with $x \in \sO_K$ and
$0 \le j \le  m_0 -1$, one has $\{\chi, 1 + \pi^mT\} =
\ov{\lambda}^2_{\pi}(T\alpha_\chi, 0)$ in $V^{2}(A)$.
But this follows immediately from \cite[Lemma~3.7(1)]{Kato89}.

To finish the proof of the theorem, we are now left with showing
that the induced inclusion map

\begin{equation}\label{eqn:Filt-Milnor-*-4}
\frac{\Fil^{\ms}_m H^1(K)}{\Fil^{\bk}_{m-1} H^1(K)} \inj
\frac{~'\Fil^{\ms}_m H^1(K)}{\Fil^{\bk}_{m-1} H^1(K)} 
\end{equation}
is surjective.

We let $E_m = \frac{\Fil^{\ms}_m H^1(K)}{\Fil^{\bk}_{m-1} H^1(K)}$
and $F_m = \frac{~'\Fil^{\ms}_m H^1(K)}{\Fil^{\bk}_{m-1} H^1(K)}$.
We let $\gr^{\bk}_\bullet H^1(K)$ denote the associated graded 
group of $\Fil^{\bk}_\bullet H^1(K)$.
We then have inclusions $E_m \subset F_m \subset \gr^{\bk}_m H^1(K)$. 
For any integers $q, r \ge 0$, let $B_r\Omega^q_{\ff}$ be the subgroup of
$\Omega^q_{\ff}$ generated by elements of the form 
$x^{p^i}d\log(x) \wedge d\log(y_1) \wedge \cdots \wedge d\log(y_{q-1})$
for $0 \le i \le r$.
We let $Z_r\Omega^q_{\ff}$ be the subgroup  of $\Omega^q_{\ff}$ generated by 
$B_r\Omega^q_{\ff}$ and elements of the form
$x^{p^r}d\log(y_1) \wedge \cdots \wedge d\log(y_{q})$.
Let $C \colon Z_1\Omega^q_{\ff} \to \Omega^q_{\ff}$ denote the Cartier
homomorphism \cite[Chapter~0, \S~2]{Illusie}.

Let 
\[
N = \{(\alpha, \beta) \in B_{m_0 +1}\Omega^1_{\ff} \bigoplus 
(\ff)^{p^{m_0}}|nC^{m_0}(\alpha) = - dC^{m_0}(\beta)\}
\]
be a submodule of $\Omega^1_{\ff} \bigoplus \ff$.
It is shown in \cite[\S~(3.9)]{Kato89} that there is an isomorphism of
abelian groups
\begin{equation}\label{eqn:Filt-Milnor-*-5}
\phi_m \colon \gr^{\bk}_m H^1(K) \xrightarrow{\cong} N; \ \ 
\phi_m(\chi) = (\alpha, \beta)
\end{equation}
such that $\{\chi, 1 + \pi^mT\} = \ov{\lambda}^2_{\pi}((T \alpha, T \beta))$. 

It follows from the definition of $~'\Fil^{\ms}_m H^1(K)$ that
$\phi_m(F_m) = N \bigcap \Omega^1_{\ff}$.
Since the kernel of $C^{m_0} \colon Z_{m_0}\Omega^1_{\ff} \to \Omega^1_{\ff}$ is
$B_{m_0}\Omega^1_{\ff}$ (see \cite[Chap.~0, \S~2, p.~520]{Illusie}),
we get $\phi_m(F_m) = B_{m_0}\Omega^1_{\ff}$.

Using the definition of $\phi_m$ and \cite[Lemma~3.7(1)]{Kato89},
we conclude that $F_m$ is generated by elements of the form 
$\delta_{j+1}([x\pi^{-mp^j}])$ with $x \in \sO_K$ and
$0 \le j \le  m_0 -1$. But we have already seen earlier in the proof
that such elements
lie in $E_m$. This shows that the inclusion ~\eqref{eqn:Filt-Milnor-*-4}
is bijective. The proof of the theorem is now complete.
\end{proof}

\begin{remk}\label{remk:Filt-Milnor-*-ex}
We remark that even in the case $p \neq 2$, part (2) of
\thmref{thm:Filt-Milnor-*} is an 
improvement of \cite[Proposition~3.2.7(iii)]{Matsuda}
and \cite[Lemma~2.6]{Kerz-Saito-2}. This is possible by virtue of
\lemref{lem:Milnor-K-filt}.
\end{remk}

The following definitions are due to Kato \cite{Kato89} and
Matsuda \cite{Matsuda}. 

\begin{defn}\label{defn:Conductor}
Let $K$ be a Henselian discrete valuation field and let $\chi \in H^1(K)$
be a character of $\Gal(\ov{K}/K)$. The Swan conductor of $\chi$ is the 
smallest integer $m \ge 0$ such that $\chi \in \Fil^{\bk}_m H^1(K)$.
It is denoted by $\Sw(\chi)$.
The Artin conductor of $\chi$  is the 
smallest integer $m \ge 0$ such that $\chi \in \Fil^{\ms}_m H^1(K)$.
It is denoted by $\Ar(\chi)$.
It follows from \thmref{thm:Fil-main} that $\Sw(\chi) \le \Ar(\chi)$.
They are the same and agree with the classical definition 
\cite[Proposition~6.8]{Kato89} if the residue field of $K$ is perfect. 
\end{defn}

\section{The fundamental group with modulus}
\label{sec:EFG-div}
In this section, we shall introduce our key object, the
{\'e}tale fundamental group with modulus $\pi^{\adiv}_1(X,D)$ of an
integral normal scheme $X$ over a field relative to 
an effective Weil divisor $D$. This will be the target of our reciprocity map
for the Kato-Saito idele class group with modulus.
This fundamental group is a priori different from the already known
fundamental group with modulus $\pi^{\ab}_1(X,D)$, introduced by
Deligne and Laumon \cite{Laumon} and used extensively in the class field
theory of Wiesend and others (see \cite{Kerz-Saito-2}).

\subsection{Ramification of field extensions}\label{sec:RFE}
Let $K$ be a Henselian discrete valuation field with ring of integers
$\sO_K$, maximal ideal $\fm_K$ and residue field $\ff$.
%We shall assume that $K$ is infinite.
%We fix a separable closure $\ov{K}$ of $K$ and 
%assume all separable algebraic extensions of $K$ to be
%contained in $\ov{K}$.
 Let $G^{(\bullet)}_K$ denote the Abbes-Saito
ramification filtration of $G_K$.

Let $L/K$ be a finite separable extension and let $n \ge 0$ be an integer.
We shall say that the ramification of $L/K$ is bounded by $n$ if
$G^{(n)}_K$ is contained in $\Gal(\ov{K}/{L})$ under the inclusions 
$G^{(n)}_K \subset G_K \supset \Gal(\ov{K}/{L})$. Note that this
definition coincides with \cite[Definition~6.3]{Abbes-Saito}.
Note also that since $G^{(n)}_K$ is normal in $G_K$, the above condition is same
as the condition that $G^{(n)}_K$ is contained in $\Gal(\ov{K}/{\wt{L}})$, where
$\wt{L}$ is the Galois closure of $L$ in $\ov{K}$.
The following is straightforward.

\begin{lem}\label{lem:Field-Galois}
Let $L_1, L_2$ be two finite separable extensions of $K$ whose ramifications 
are bounded by $n$. Let $L'_1 \subset L_1$ be a subextension.
Then $L'_1$ and the compositum $L_1L_2$ have ramifications bounded by $n$.
\end{lem}

\subsection{Morphisms with bounded ramification along a 
divisor}\label{sec:GC}
Let us now assume that $k$ is a field and $X \in \Sch_k$ is an
integral normal scheme. 
Let $K$ denote the function field of $X$
and
% $\ov{K}$ a chosen separable closure of $K$.
let $D \subset X$ be an effective 
Weil divisor with support $|D| = C$. We consider $C$ as a reduced
closed subscheme of $X$ with complement $U$.

We write $D = {\underset{\lambda}\sum} \ n_\lambda 
\ov{\{\lambda\}}$,
where $\lambda$ runs through the set of generic points of $C$.
We allow $D$ to be empty in which case we 
write $D = 0$.
For any generic point $\lambda$ of $C$, we let $K_\lambda$ denote the
fraction field of the henselization of the discrete valuation ring 
$\sO_{X,\lambda}$.
If $D'' \subset X$ is another effective Weil divisor, then we shall say
that $D'' \ge D'$ if $D'' - D'$ is effective in the free abelian
group of Weil divisors on $X$. 

Let $\sU_{\et}(U)$ be the category of finite {\'e}tale covers of $U$.
We denote an object of $\sU_\et(U)$ by $(U', f)$, where $f \colon
U' \to U$ is the underlying finite {\'e}tale map.
For $(U', f) \in \sU_\et(U)$, let $X'$ denote the normalization of
$X$ in $k(U')$, where the latter denotes the ring of total quotients of $U'$.
Note that $U'$ is a disjoint union of integral normal schemes each of whose
components lies in $\sU_\et(U)$. In particular, $X'$ is also a disjoint 
union of integral normal schemes each of which is finite and surjective over 
$X$ and the induced map $X' \to X$ extends $f$. We denote this extension by $f$
itself. Let $K'$ denote the ring of total
quotients of $X'$. For any $\lambda \in X'^{(1)} \cap C$, we let 
$K'_{\lambda}$ denote the
quotient field of $\sO^h_{X', \lambda}$. Note that the latter is a Henselian
discrete valuation ring. 

\begin{defn}\label{defn:R-bound}
We shall say that $(U',f)$ has ramification bounded by $D$ if the map
$f \colon X' \to X$ has the property that for every generic point 
$\lambda$ of $D$ 
and every point $\lambda' \in f^{-1}(\lambda)$, the extension of fields 
$K_\lambda \inj K'_{\lambda'}$ has ramification bounded by $n_\lambda$.

Let $L/K$ be a finite separable field extension. Let $X_L$ denote the
normalization of $X$ in $L$ and let $f \colon X_L \to X$ denote the
induced finite surjective map. We shall say that $L/K$ has ramification 
bounded by $D$ if $f$ is {\'e}tale over $U$ and the resulting cover
$(f^{-1}(U), f)$ has ramification bounded by $D$.
We shall say that $(U',f)$ (resp. $L/K$) is tamely ramified if its
ramification is bounded by $C$.
\end{defn}

The following is easily deduced from definitions.

\begin{prop}\label{prop:Unramified}
If $(U',f)$ has ramification bounded by $D$ and $D' \ge D$ with $|D'|
= |D|$, then it has ramification bounded by $D'$. If $X$ is regular, then
$(U',f)$ has ramification bounded by $D = 0$  if and only if 
 $f \colon X' \to X$ is {\'e}tale.
If $X$ is regular along $C$ and the latter is a normal crossing divisor, 
then $(U',f)$ is tamely ramified if and only
if it is tamely ramified in the sense of 
\cite[Definition~2.2.2]{Grothendieck-Murre}.
\end{prop}
\begin{proof}
The first part is clear. The third part follows from
the third property of the Abbes-Saito filtration.
One way implication of the second part is also clear.
So we only need to show that if $X$ is regular and
$(U',f)$ has ramification bounded by $D = 0$, then  
$f \colon X' \to X$ is {\'e}tale.

It follows from the second property of the Abbes-Saito filtration and
the fppf descent for unramified morphisms that
$f$ is unramified over all codimension one points of $X$. 
Hence, it must be unramified.
Since a dominant morphism over the spectrum of a discrete valuation ring
is flat, it follows that $f$ is flat outside a closed subscheme
$B \subset X$ of codimension more than one. It follows that
$f$ is {\'e}tale outside $B$. The Zariski-Nagata purity theorem then
implies that $B$ must be empty.
\end{proof}

\subsection{The underlying Galois category}\label{sec:GC-0}
We fix a point $u \in U$ together with the canonical map 
$\ov{u} \to \{u\} = \Spec(k(u))$, where $\ov{u}$ is the
spectrum of a separable closure of $k(u)$.
Let $F \colon \sU_\et(U) \to \Sets$ be the functor
which assigns to any $(U', f) \in \sU_{\et}(U)$, the finite set
$f_{\ov{u}} := \Hom_U(\ov{u}, U')$.
Let $\sU_\et(X,D)$ be the full subcategory of objects in $\sU_\et(U)$
which have ramification bounded by $D$.
There are inclusions of categories with fiber functors
\begin{equation}\label{eqn:Fiber-functor}
(\sU_\et(X), F) \inj (\sU_\et(X,D), F) \inj (\sU_\et(U), F).
\end{equation}

It is well known that $(\sU_\et(X), F)$ and $(\sU_\et(U), F)$
are Galois categories in the sense of \cite[\S~4, 5, Expos{\'e}~V]{SGA-1}.
The following lemma asserts a similar property of $(\sU_\et(X,D), F)$.

\begin{lem}\label{lem:Galois-cat}
$(\sU_\et(X,D), F)$ is a Galois category.
\end{lem}
\begin{proof}
The proof of the lemma is routine following the
strategy of the proof of \cite[Theorem~2.4.2]{Grothendieck-Murre}.
Since $(\sU_\et(U), F)$ is known to be a Galois category and
$(\sU_\et(X,D), F)$ is its full subcategory, the axioms
(G3), (G4) and (G5) of a Galois category  from 
\cite[\S~4, 5, Expos{\'e}~V]{SGA-1} are immediately verified.
The axioms (G1) and (G2) follow from \lemref{lem:Field-Galois} and
the fact that the category of normal
varieties admits quotients by finite group actions.
The axiom (G6) follows because it holds for
$(\sU_\et(U), F)$ and a finite birational morphism between normal varieties 
is an isomorphism. We leave out the details.
\end{proof}

Since a Galois category is equivalent to the category of finite sets with 
continuous action of a uniquely determined profinite group (e.g.,
see \cite[Expos{\'e}~V, \S~4]{SGA-1}), we are led to the following.

\begin{defn}\label{defn:Fun-D}
The co-1-skeleton (or divisorial) {\'e}tale fundamental group of $X$ with 
modulus $D$ is the profinite group $\pi^{\divf}_1(X,D, \ov{u})$ such that
$(\sU_\et(X,D), F)$ is equivalent to the category of finite sets with
continuous action of $\pi^{\divf}_1(X,D, \ov{u})$.
We let $\pi^{\adiv}_1(X,D)$ denote the abelianization of 
$\pi^{\divf}_1(X,D, \ov{u})$ and call it the (divisorial) abelianized 
{\'e}tale fundamental group of $X$ with modulus $D$.
Since $\pi^{\divf}_1(X,D, \ov{u})$ varies from $\pi^{\divf}_1(X,D, \ov{u'})$
only by inner automorphisms, the group $\pi^{\adiv}_1(X,D)$
is well defined.
\end{defn}

\begin{remk}\label{remk:Curve-defn}
  In \cite{Kerz-Saito-2}, Kerz and Saito use another notion of abelianized fundamental
  group with modulus, which they denote by $\pi^{\ab}_1(X,D)$, and where the bound on the
  ramification is measured by pulling back the {\'e}tale covers to curves.
  This notion of ramification was originally introduced by Deligne and Laumon
  \cite{Laumon}. We shall see in \corref{cor:SNCD} that
  $\pi^{\adiv}_1(X,D)$ coincides with $\pi^{\ab}_1(X,D)$ if $X$ is regular and
  $C$ is a simple normal crossing divisor on $X$.
  \end{remk}

It follows from \cite[Expos{\'e}~V, Proposition~6.9]{SGA-1} that there are
surjective continuous homomorphisms
\begin{equation}\label{eqn:Fiber-functor-0}
\Gal(K^{\ab}/K) \surj \pi^{\ab}_1(U) \surj \pi^{\adiv}_1(X,D) \surj
\pi^{\ab}_1(X).
\end{equation}

Recall that an object $(U', f)$ of $\sU_\et(X,D)$ is called connected if
$U'$ is not empty and for any inclusion $(V', g) \inj (U',f)$,
the scheme $V'$ is either empty or the inclusion is an isomorphism.
Recall also that if $(U', f)$ is a connected finite {\'e}tale cover of $U$,
then it is called a Galois cover if $\Aut_U(U')$ acts transitively on
$F(U')$. This is equivalent to saying that $U = {U'}/{\Aut_U(U')}$ and
$U' \times_U U' \cong U \times_{\Spec(k)} \Aut_U(U')$ if we consider
$\Aut_U(U')$ a finite constant group scheme over $\Spec(k)$.

\begin{lem}\label{lem:Galois-cover}
For any $(U',f) \in \sU_\et(X,D)$ connected, there exists a 
Galois cover $(V', g) \in \sU_\et(X,D)$ dominating $(U',f)$.
\end{lem}
\begin{proof}
This is standard in $\sU_\et(U)$. One constructs $(V',g)$ in
$\sU_\et(U)$ as a connected component of the $n$-fold self fiber product
$U'_n := U' \times_U \cdots \times_U U'$, where $n = |F(U')|$.
If $(U',f) \in \sU_\et(X,D)$, then 
we have seen in \lemref{lem:Galois-cat} that $U'_n$ has
ramification bounded by $D$. Hence, so does $(V',g)$.
It follows that $(V',g)$ is a Galois cover which 
lies in $\sU_\et(X,D)$ and dominates $(U',f)$.
\end{proof}

\begin{lem}\label{lem:Function-field}
Let $(U',f) \in \sU_\et(X,D)$ be a Galois cover and let
$K'$ be the function field of $U'$. Then ${K'}/K$ is Galois with 
Galois group $\Aut_U(U')$ and its ramification is bounded by $D$.
\end{lem}
\begin{proof}
This is an easy consequence of 
the fact that for a Galois cover $(U',f)$, the action of
$\Aut_U(U')$ on $F(U')$ defines a bijection 
$\Aut_U(U') \xrightarrow{\cong} \Hom_{U}(\xi, U')$ for every
geometric point $\xi \to U$. We can take $\xi$ to be $\Spec(\ov{K})$
and the lemma follows.
\end{proof}

We say that $(U', f)$ is abelian if it is a Galois cover with
Galois group $\Aut_U(U')$ abelian.
If $(U', f)$ is a Galois cover and we let $H = [\Aut_U(U'), \Aut_U(U')]$,
then $H$ acts freely on $U'$ and $V' = {U'}/H$ is a Galois cover of
$U$ with Galois group $\Aut_U(V') = {\Aut_U(U')}^{\ab}$.
We call $V'$ the abelianization of $U'$.
If $(U',f) \in \sU_\et(X,D)$, then we have seen in the proof of
\lemref{lem:Galois-cat} that $V' \in \sU_\et(X,D)$.

Let $\sU_{\ab}(U)$ denote the category of finite {\'e}tale covers 
$f \colon U' \to U$ such that every connected component of $U'$
is an abelian cover of $U$.
Let $\sU_{\ab}(X,D)$ denote the full subcategory of 
$\sU_{\ab}(U)$  whose objects have ramification bounded by $D$. 
We denote the restrictions of the functor
$F \colon  \sU_\et(U) \to \Sets$ to $\sU_{\ab}(U)$ and $\sU_{\ab}(X,D)$ 
by the same notation $F$.

\begin{prop}\label{prop:Abelian-cover}
$(\sU_{\ab}(X,D), F)$ is a Galois category such that
$\pi_1((\sU_{\ab}(X,D), F)) = \pi^{\adiv}_1(X,D)$.
\end{prop}
\begin{proof}
It is well known (and easy to show)
that $(\sU_{\ab}(U), F)$ is a Galois category with
$\pi_1((\sU_{\ab}(U), F)) = \pi^{\ab}_1(U)$.
Since $(\sU_{\ab}(X,D), F)$ is a full subcategory of
$(\sU_{\ab}(U), F)$, the proof that the former category is Galois
is same as that of \lemref{lem:Galois-cat}.
We shall prove the second assertion.

Let $(\sU_{\Gal}(X,D), F)$ denote the subcategory of
Galois covers in $(\sU_\et(X,D), F)$.
Let $I$  be the set indexing the connected objects of 
$(\sU_\et(X,D), F)$. 
Let $I^{\Gal}$ (resp. $I^{\ab}$) be the set indexing the objects of 
$(\sU_{\Gal}(X,D), F)$ (resp. $(\sU_{\ab}(X,D), F)$).
If $P_a, P_b \in (\sU_{\ab}(X,D), F)$ corresponding to $a, b \in I^{\ab}$,
then we let $a \le b$ if $P_b$ dominates $P_a$. Then $I^{\ab}$ is a partially 
ordered set. Moreover, it is directed. Indeed, we have seen above that
given $a, b \in I^{\ab}$, every connected component $Z$ of $P_a \times_U P_b$
lies in $(\sU_\et(X,D), F)$. Furthermore, it follows from
\lemref{lem:Galois-cover} that there is an object 
$(Y,g) \in (\sU_\et(X,D), F)$ which is Galois and dominates $Z$.

We let $P_c$ denote the abelianization of $Y$ and let $g' \colon P_c \to X$
be the induced map.  We have seen above that
$(P_c, g') \in (\sU_{\ab}(X,D), F)$.  To show that
$(P_c, g')$ dominates $P_a, P_b$, it suffices to show that
the function fields of $P_a$ and $P_b$ are dominated by the function field
of $P_c$. But this follows because the function field of 
$Z$ dominates those of $P_a$ and $P_b$ and \lemref{lem:Function-field} 
says that the Galois groups of the function fields 
of $P_a$ and $P_b$ are abelian.  
We have thus shown that $I^{\ab}$ is a directed set. An easier argument
shows that the same holds for $I$ and $I^{\Gal}$. Clearly, 
there are inclusions of directed sets $I^{\ab} \inj I^{\Gal} \inj I$. 

By definition of the automorphism group of a fiber functor,
we have an isomorphism of profinite groups
$\pi^{\divf}_1(X,D) \xrightarrow{\cong} {\underset{a \in I}\varprojlim} \
\Aut_U(P_a)$. It follows from \lemref{lem:Galois-cover} that

\begin{equation}\label{eqn:Abelian-cover-0}
\pi^{\divf}_1(X,D) \xrightarrow{\cong} {\underset{a \in I^{\Gal}}\varprojlim} \
\Aut_U(P_a).
\end{equation}

Now, we know that the transition maps of the
projective system $\{\Aut_U(P_a)\}_{a \in I^{\Gal}}$ are surjective.
This implies that the same holds for the projective system of 
commutator subgroups
$\{[\Aut_U(P_a), \Aut_U(P_a)]\}_{a \in I^{\Gal}}$. 
It follows that the sequence
\begin{equation}\label{eqn:Abelian-cover-3}
1 \to {\underset{a \in I^{\Gal}}\varprojlim}
[\Aut_U(P_a), \Aut_U(P_a)] \to
\pi^{\divf}_1(X,D) \to {\underset{a \in I^{\Gal}}\varprojlim} (\Aut_U(P_a))^{\ab}
\to 1
\end{equation}
is exact.

Since the projection 
$[\pi^{\divf}_1(X,D), \pi^{\divf}_1(X,D)] \to [\Aut_U(P_a), \Aut_U(P_a)]$
is surjective for every $a \in I^{\Gal}$, this implies by
 \cite[Corollary~1.1.8]{Pro-fin} that
\[
\ov{[\pi^{\divf}_1(X,D), \pi^{\divf}_1(X,D)]} =
{\underset{a \in I^{\Gal}}\varprojlim}
[\Aut_U(P_a), \Aut_U(P_a)],
\]
where the left hand side is the closure taken in 
$\pi^{\divf}_1(X,D)$.
On the other hand, there is an exact sequence of topological groups
(by definition of $\pi^{\adiv}_1(X,D)$)
\begin{equation}\label{eqn:Abelian-cover-2}
1 \to \ov{[\pi^{\divf}_1(X,D), \pi^{\divf}_1(X,D)]} \to
\pi^{\divf}_1(X,D) \to \pi^{\adiv}_1(X,D) \to 1.
\end{equation}
Comparing ~\eqref{eqn:Abelian-cover-3} and ~\eqref{eqn:Abelian-cover-2},
we get $\pi^{\adiv}_1(X,D) \xrightarrow{\cong} 
{\underset{a \in I^{\Gal}}\varprojlim} \ (\Aut_U(P_a))^{\ab}$.

Finally, we note that the map $I^{\Gal} \to I^{\ab}$ induced by 
$P_a \mapsto (P_a)^{\ab}$ is cofinal in $I^{\ab}$. It follows that
\begin{equation}\label{eqn:Abelian-cover-1}
\pi_1((\sU_{\ab}(X,D), F)) = {\underset{a \in I^{\ab}}\varprojlim} \
\Aut_U(P_a) \cong {\underset{a \in I^{\Gal}}\varprojlim} \
(\Aut_U(P_a))^{\ab} =  \pi^{\adiv}_1(X,D).
\end{equation}
This proves the second part of the proposition.
\end{proof}

\begin{lem}\label{lem:Fun-modulus}
Let $I^{\Gal}_{D/K}$ (resp. $I^{\ab}_{D/K}$) be the directed system of finite 
Galois (resp. abelian) extensions $L/K$ which have 
ramifications bounded by $D$.
Then there are canonical isomorphisms of profinite groups
\[
\pi^{\divf}_1(X,D) \xrightarrow{\cong} {\underset{L \in I^{\Gal}_{D/K}}\varprojlim}
\Gal(L/K); \ \ \pi^{\adiv}_1(X,D) \xrightarrow{\cong} 
{\underset{L \in I^{\ab}_{D/K}}\varprojlim} \Gal(L/K).
\]
\end{lem}
\begin{proof}
We follow the notations of \propref{prop:Abelian-cover}.
Using \lemref{lem:Function-field} together with 
~\eqref{eqn:Abelian-cover-0} and ~\eqref{eqn:Abelian-cover-1}, we only need to
show that the correspondence $(U',f) \mapsto k(U')$
defines isomorphisms of directed systems
$I^{\Gal} \cong I^{\Gal}_{D/K}$ and $I^{\ab} \cong I^{\ab}_{D/K}$.
But this is immediate.
\end{proof}

\subsection{Pontryagin duality}\label{sec:PD*}
Our next goal is to identify the Pontryagin dual of $\pi^{\adiv}_1(X,D)$.
Before doing so, we briefly recall this duality.
Let $\pfd$ denote the full subcategory of the category of topological
abelian groups consisting of those objects which are either profinite or 
discrete abelian groups.
We shall consider ${\Q}/{\Z}$ as a discrete topological abelian group.
For any $G \in \pfd$, we let
$G^{\vee} = \Hom_{\cont}(G, {\Q}/{\Z})$ and consider it as a
topological abelian group with respect to the topology of pointwise
convergence.

Suppose that $G$ is either profinite or
discrete torsion. Then $G^{\vee} \in \pfd$.
%in general the dual functor take values in \Ab?? We should 
%mention this in the following lemma. Similarly for double dual map. 
In this case, the Pontryagin duality theorem  (see \cite[\S~2.9]{Pro-fin}) 
says that $G^{\vee}$ is the classical Pontryagin dual of $G$.
Moreover, the evaluation map ${\rm ev}_G \colon 
G \to (G^{\vee})^{\vee}$ is an isomorphism of 
topological abelian groups.
Note that this duality does not hold otherwise. For example,
$\Z^{\vee} \cong {\Q}/{\Z}$ and ${(\Z^{\vee})}^{\vee} \cong \wh{\Z}$.
The evaluation map is the profinite completion map
$\Z \to \wh{\Z}$.
In general, we can however say the following.

\begin{lem}\label{lem:Pont-dual}
The functor $G \mapsto G^{\vee}$ from $\pfd$ to the category of
topological abelian groups is exact and the
map ${\rm ev}_G$ is injective.
%statement suggest that double dual is an Endomorphism of \pfd. 
%Is it true?
\end{lem}
\begin{proof}
Let 
\[
0 \to G' \xrightarrow{\alpha} G \xrightarrow{\beta} G'' \to 0
\]
be an exact sequence in $\pfd$.
One can check by an easy inspection that 
\[
0 \to G''^{\vee} \to G^{\vee} \to G'^{\vee}
\]
is exact. 
We prove that the last arrow is surjective.

If $G$ is discrete, then other groups are also discrete
and the assertion follows because ${\Q}/{\Z}$ is injective
as an abelian group.
We suppose therefore that $G$ is profinite. Then the other groups
are also profinite.
Since profinite groups are compact Hausdorff, we see that $G'$ is closed
in $G$. If $H = \coker(G^{\vee} \to G'^{\vee}) \neq 0$, then
we have an exact sequence
\[
0 \to H^\vee \to G' \xrightarrow{\alpha} G
\] 
using the Pontryagin duality theorem. Since $H \neq 0$, one knows that
$H^\vee$ can not be zero, again by the Pontryagin duality. But this
contradicts the assumption that $\alpha$ is injective.

The second part of the lemma follows easily
from the first part and the example before the lemma,
using the fact that every torsion-free discrete
abelian group is a direct limit of finitely generated free abelian groups.
We leave out the details.
\end{proof}

\begin{lem}\label{lem:Density}
Let $G \in \pfd$ and $H \subset G$ a subgroup. Then $H$ is dense
in $G$ if and only if every element of $G^\vee$ which annihilates $H$ is
zero.
\end{lem}
\begin{proof}
Suppose $H$ is dense in $G$ and let $\chi \in G^\vee = 
\Hom_{\cont}(G, {\Q}/{\Z})$
is such that $\chi(H) = 0$. Since $\Ker(\chi)$ is closed in $G$ and it contains
$H$, it must contain $G$. Equivalently, $\chi = 0$.
Conversely, suppose every element of $G^\vee$ which annihilates $H$ is
zero and $\ov{H} \neq G$. Let $G' = G/{\ov{H}}$. 
Then $G' \in \pfd$ and it follows
from \lemref{lem:Pont-dual} that there is an exact sequence of
continuous homomorphisms
\[
0 \to G'^{\vee} \to G^\vee \to \ov{H}^{\vee} \to 0.
\]
Since $G' \neq 0$, we have that $G'^\vee \neq 0$.
We now choose any $\chi \neq 0$ in $G'^\vee$. Then $\chi$ is a nonzero 
character of $G$ which vanishes on $H$. This contradicts our hypothesis. 
\end{proof}

\subsection{The Pontryagin dual of $\pi^{\adiv}_1(X,D)$}
\label{sec:PD}
In this subsection, we shall identify the Pontryagin dual of $\pi^{\adiv}_1(X,D)$
with a special subgroup of $H^1(U)$.

Let $X$ be as in \S~\ref{sec:GC}.
For any $x \in X^{(1)} \cap C$, we let $\wh{K}_x$ denote the quotient field of 
the $\fm_x$-adic completion $\wh{\sO_{X,x}}$ of $\sO_{X,x}$. 
Let $\sO^{sh}_{X,x}$ denote the strict henselization of $\sO_{X,x}$
and let $K^{sh}_x$ denote its quotient field.
Then it is clear from the definitions that there are
inclusions 
\begin{equation}\label{eqn:Field-incln}
K \inj K_x \inj K^{sh}_x \inj \ov{K} \ \mbox{and} \ \ K \inj K_x \inj \wh{K}_x. 
\end{equation}

Let $\Irr_C$ denote the set of all generic points of $C$ and
let $C_\lambda$ denote the closure of an element $\lambda \in \Irr_C$.
Let $D \subset X$ be an effective Weil divisor with support $C$.
We write $D = {\underset{\lambda \in \Irr_C}\sum} n_\lambda C_\lambda$.
%
%this part has been changed 
%We can also write $D = {\underset{x \in X^{(1)}}\sum} n_x \ov{\{x\}}$,
%where $n_x = 0$ for all $x \in U$.
%
%\begin{defn}\label{defn:Fid_D}
%Let $\Fil_D H^1(K)$ denote the subgroup of characters
%$\chi \in H^1(K)$ such that for every $x \in X^{(1)}$, the image $\chi_x$ of
%$\chi$ under the canonical surjection $H^1(K) \surj H^1(K_x)$ 
%(the surjectivity follows from ~\eqref{eqn:Field-incln})
%lies in $\Fil^{\ms}_{n_x} H^1(K_x)$.
%By \thmref{thm:Fil-main}, this is equivalent to the condition that
%$\chi_x \in \Fil^{\as}_{n_x} H^1(K_x)$.
%\end{defn}
%
%
%Since $n_x = 0$ for all $x \in U^{(1)}$, it follows from the property (2)
%of the Abbes-Saito filtration that $\Fil_D H^1(K)$ lies 
%inside $H^1(U)$ under the canonical inclusion 
%$H^1(U) \inj H^1(K)$. 

\begin{defn}\label{defn:Fid_D}
Let $\Fil_D H^1(K)$ denote the subgroup of characters
$\chi \in H^1(U) \subset H^1(K)$ such that for every $x \in \Irr_C$, the image $\chi_x$ of
$\chi$ under the canonical surjection $H^1(K) \surj H^1(K_x)$ 
(the surjectivity follows from ~\eqref{eqn:Field-incln})
lies in $\Fil^{\ms}_{n_x} H^1(K_x)$.
By \thmref{thm:Fil-main}, this is equivalent to the condition that
$\chi_x \in \Fil^{\as}_{n_x} H^1(K_x)$.
\end{defn}

We shall write $\Fil_D H^1(K)$ also
as $\Fil_D H^1(U)$. In fact, $H^1(U)$ can be completely described
in terms of $\Fil_D H^1(U)$ if we allow $D$ to vary.
To see this, let 
$\Div_C(X)$ denote the set of all closed subschemes $D \subset X$
such that $D_\red = C$. Then $\Div_C(X)$ is clearly a directed set
with respect to inclusion.
Let $\Div^W_C(X)$ denote the set of all effective Weil divisors on $X$
with support $C$. Then $\Div^W_C(X)$ is a directed set
with respect to the order: $D_1 \le D_2$ if $D_2$ dominates $D_1$ (i.e., 
$D_2 - D_1$ is effective). There is a canonical morphism of directed sets
$\Div_C(X) \to \Div^W_C(X)$ which has the property that
given any $D \in \Div^W_C(X)$, there exists $D' \in \Div_C(X)$ such that
$D \le D'$ as Weil divisors.

\begin{prop}\label{prop:Fil-Exhaustive}
The canonical map 
\[
{\underset{D \in \Div_C(X)}\varinjlim} \Fil_D H^1(K) \to H^1(U)
\]
is an isomorphism.
\end{prop}
\begin{proof}
We only need to show that the map in question is surjective.
Let $\chi \in H^1(U)$. Let $\lambda \in \Irr_C$ and let $\chi_\lambda$
denote its image under the composite map $H^1(U) \inj H^1(K) \to 
H^1(K_\lambda)$.
By \thmref{thm:Fil-main}, there exists an integer $n_\lambda \ge 1$
such that $\chi_\lambda \in \Fil^{\ms}_{n_\lambda} H^1(K_\lambda)$.
We let $D = {\underset{\lambda \in \Irr_C}\sum} n_\lambda C_\lambda$.
Since $\chi \in H^1(U)$, 
%one has that $\chi_x \in \Fil^{\ms}_0 H^1(K_x)$
%for all $x \in U^{(1)}$. 
it follows immediately that
$\chi \in \Fil_D H^1(K)$. If we choose any element
$D' \in \Div_C(X)$ dominating $D$, we get
$\chi \in \Fil_{D'} H^1(K)$.
\end{proof}

To give a similar description of $\Fil_D H^1(K)$
as a subgroup of $H^1(K)$,
let $I^{\Gal}_{D/K}$ and $I^{\ab}_{D/K}$ be as in \lemref{lem:Fun-modulus}.
For $L/K$ a Galois extension, let $H^1(L/K)$ denote the Pontryagin 
dual of $\Gal(L/K)$. In other words, it is the kernel of
the canonical surjection $H^1(K) \surj H^1(L)$.
Note that the map $H^1({L_{\ab}}/K) \to H^1(L/K)$ is an
isomorphism, where $L_{\ab}$ is the maximal abelian extension of $K$
inside $L$.

Let $I^{\ab}_{U/K}$ be the direct system of finite abelian extensions
$L/K$ which are {\'e}tale along $U$.
Since the map $\pi^{\ab}_1(U) \to \Gal(L/K)$ is surjective for
every $L \in I^{\ab}_{U/K}$ and 
$\pi^{\ab}_1(U) \xrightarrow{\cong} 
{\underset{L \in I^{\ab}_{U/K}}\varprojlim} \Gal(L/K)$
(analogous to ~\eqref{eqn:Abelian-cover-0}),
it follows from the Pontryagin duality of profinite groups 
(see \cite[Lemma~2.9.3]{Pro-fin}) that
the map ${\underset{L \in I^{\ab}_{U/K}}\varinjlim} H^1(L/K) \to
H^1(U)$ is an isomorphism.

\begin{lem}\label{lem:Fil-D-char-0}
Let $L/K$ be a finite abelian extension. Then $L \in I^{\ab}_{D/K}$
if and only if $H^1(L/K) \subset \Fil_D H^1(K)$.
\end{lem}
\begin{proof}
Suppose $L \in I^{\ab}_{D/K}$ and $\chi \in H^1(L/K)$.
Then $\chi(\Gal(K^{\ab}/L)) = 0$. 
%Let $x \in X^{(1)}$.
%If $x \in U$, then $\chi$ is unramified at $x$ since $L$ is
%unramified along $U$. Hence, $\chi_x(G^{(0)}_{K_x}) = 0$.
%Equivalently, $\chi \in \Fil^{\as}_0 H^1(K_x)$. 
Since $H^1(L/K) \subset H^1(U)$, we have $\chi \in H^1(U)$.
Let $f:X_L \to X$ 
denote the normalization of $X$ in  $L$. 
If $x \in \Irr_C$, then it follows from the definition of $L$ being
in $I^{\ab}_{D/K}$ that $G^{(n_x)}_{K_x} \subset \Gal({K^{\ab}_x}/L_{x'})$
under the inclusions $G^{(n_x)}_{K_x} \inj G_{K_x} \supset 
\Gal({K^{\ab}_x}/L_{x'})$,
where $x' \in f^{-1}(x)$.
It follows that $\chi_x(G^{(n_x)}_{K_x}) = 0$. Equivalently,
$\chi_x \in \Fil^{\as}_{n_x} H^1(K_x)$.
We have therefore shown that $\chi \in \Fil_D H^1(K)$.

Conversely, suppose that $H^1(L/K) \subset \Fil_D H^1(K)$. Then 
$H^1(L/K) \subset H^1(U)$ and taking their Pontryagin duals, we get
$\pi^{\ab}_1(U) \surj \Gal(L/K)$. Since $\pi^{\ab}_1(U)$ is the Galois group
of the maximal abelian extension of $K$ which is {\'e}tale over $U$,
Galois theory says that $L$ is {\'e}tale over $U$.
We now show that the ramification of $L/K$ is bounded by $D$.

Let $x \in \Irr_C$ and let $x' \in f^{-1}(x)$.
We need to show that $G^{(n_x)}_{K_x}$ is annihilated under the map
$\Gal({K^{\ab}_x}/{K_x}) \to \Gal(L_{x'}/K_{x})$.
For this, we consider the commutative diagram (see ~\eqref{eqn:Field-incln})
\begin{equation}\label{eqn:Fil-D-char-0-1}
\xymatrix@C.8pc{
G^{(n_x)}_{K_x} \ar@{^{(}->}[r] \ar[dr] & 
\Gal({K^{\ab}_x}/{K_x}) \ar@{->>}[r] \ar@{^{(}->}[d] & \Gal(L_{x'}/K_{x}) \ar[d] 
\ar@{^{(}->}[r] & \Gal({(K_x \otimes_K L)}/{K_x}) \ar[dl]^-{\cong} \\
& \Gal(K^{\ab}/K) \ar@{->>}[r] & \Gal(L/K). & }
\end{equation}
Note that the slanted arrow on the right is an isomorphism because
$\Spec(L) \to \Spec(K)$ is a $\Gal(L/K)$-torsor and 
$\Spec(K_x \otimes_K L) \to \Spec(K_x)$ is a base change of this torsor.
It suffices therefore to show that $G^{(n_x)}_{K_x} \subset \Gal(K^{\ab}/K)$
is annihilated in $\Gal(L/K)$.

To prove this last claim, it suffices using the Pontryagin duality to show
that every character of $\Gal(L/K)$ annihilates the image of
$G^{(n_x)}_{K_x}$. The latter is equivalent to the statement that every
element $\chi \in H^1(L/K)$ has the property that $\chi_x$ lies
in $\Fil^{\as}_{n_x}H^1(K_x)$. But this is ensured by our hypothesis that
$H^1(L/K) \subset \Fil_D H^1(K)$.  
\end{proof}

\begin{prop}\label{prop:Fil-D-lim}
The inclusion $\Fil_D H^1(K) \inj H^1(K)$ defines an isomorphism
\[
{\underset{L \in I^{\ab}_{D/K}}\varinjlim} H^1(L/K) 
\xrightarrow{\cong} \Fil_D H^1(K).
\]
\end{prop}
\begin{proof}
%In view of \lemref{lem:Fil-D-char-0}, we only need to show that
%the map given in the proposition is surjective.
Recall that  $H^1(K)$ is a torsion abelian group. We consider it a
topological abelian group with discrete topology.
Let $F = {H^1(U)}/{\Fil_D H^1(K)}$ so that there is an exact 
sequence of discrete torsion abelian groups
\begin{equation}\label{eqn:Fil-D-lim-0}
0 \to \Fil_D H^1(K) \to H^1(U) \to F \to 0.
\end{equation}

Since $\Hom_{\cont}(G, {\Q}/{\Z}) \cong \Hom_{\Ab}(G, {\Q}/{\Z})$
for a discrete torsion group $G$, it follows from \lemref{lem:Pont-dual}
that
\begin{equation}\label{eqn:Fil-D-lim-1}
0 \to F^{\vee} \to \pi^{\ab}_1(U) \to (\Fil_D H^1(K))^{\vee} \to 0
\end{equation}
is an exact sequence of continuous homomorphisms of profinite abelian groups
(see \cite[Theorem~2.9.6]{Pro-fin}).
Note that $F^{\vee}$ is closed in $\pi^{\ab}_1(U)$ as $(\Fil_D H^1(K))^{\vee}$
is Hausdorff.

It follows from \cite[Proposition~2.1.4(d)]{Pro-fin} that 
$F^{\vee}$ is an intersection of open (hence closed) 
%We should also mention that closed subgroup of finite index 
%(follows because \pi_1^{ab}(U) is compact).
subgroups
of $\pi^{\ab}_1(U)$.
Let $J$ be the set of all open subgroups of $\pi^{\ab}_1(U)$
which contain $F^{\vee}$. Then $J$ is a directed set whose
order is given by declaring $H_1 \le H_2$ if $H_1 \subset H_2$.
The above statement then says that $F^{\vee} = 
{\underset{H \in J}\varprojlim} H$.
Hence, we get a short exact sequence of inverse systems of
profinite groups  
\begin{equation}\label{eqn:Fil-D-lim-3}
0 \to \{H\}_{H \in J} \to \{\pi^{\ab}_1(U)\}_{H \in J} \to 
\{{\pi^{\ab}_1(U)}/H\}_{H \in J} \to 0,
\end{equation}
where the middle term is a constant inverse system.
We conclude from \cite[Proposition~2.2.4]{Pro-fin} that this
sequence remains exact after taking the inverse limits.
But this is equivalent by ~\eqref{eqn:Fil-D-lim-1} to saying that the map
\begin{equation}\label{eqn:Fil-D-lim-4}
(\Fil_D H^1(K))^{\vee} \to {\underset{H \in J}\varprojlim} \ {\pi^{\ab}_1(U)}/H
\end{equation}
is an isomorphism of profinite groups.

For every $H \in J$, let $H' = p^{-1}_U(H)$, where $p_U \colon \Gal(K^{\ab}/K) 
\surj \pi^{\ab}_1(U)$ is the canonical surjection. Then Galois theory of 
field extensions tells us that there is a unique finite abelian 
subextension $L_H/K$ such that $\Gal(L_H/K) = {\Gal(K^{\ab}/K)}/{H'} = 
{\pi^{\ab}_1(U)}/{H}$. Any such $L_H$ must lie in $I^{\ab}_{U/K}$.
Moreover, $(\Fil_D H^1(K))^{\vee} \xrightarrow{\cong} 
{\underset{H \in J}\varprojlim} \Gal(L_H/K)$.
Using the Pontryagin duality and \cite[Lemma~2.9.3]{Pro-fin},
we get
\begin{equation}\label{eqn:Fil-D-lim-5}
{\underset{H \in J}\varinjlim} H^1(L_H/K) \xrightarrow{\cong} 
\Fil_D H^1(K).
\end{equation}
Since $L_H \in I^{\ab}_{D/K}$ for every $H \in J$ by \lemref{lem:Fil-D-char-0},
the desired surjectivity follows.
\end{proof}

\begin{thm}\label{thm:Dual-D}
By taking the Pontryagin duals, the canonical homomorphism of profinite 
groups
$\pi_1(U) \to \pi^{\divf}_1(X,D)$
defines an isomorphism 
\[
(\pi^{\divf}_1(X,D))^{\vee} \cong (\pi^{\adiv}_1(X,D))^{\vee} \xrightarrow{\cong} 
\Fil_D H^1(K).
\]
\end{thm}
\begin{proof}
The first isomorphism is obvious. We only need to prove the second
isomorphism.
By \lemref{lem:Fun-modulus}, we have
$\pi^{\adiv}_1(X,D) \xrightarrow{\cong} 
{\underset{L \in I^{\ab}_{D/K}}\varprojlim} \Gal(L/K)$.
Taking the Pontryagin duals, this yields (see \cite[Lemma~2.9.3]{Pro-fin})
${\underset{L \in I^{\ab}_{D/K}}\varinjlim} H^1(L/K) \xrightarrow{\cong} 
(\pi^{\adiv}_1(X,D))^{\vee}$. We now apply \propref{prop:Fil-D-lim}
to conclude the proof.
\end{proof}

\begin{cor}\label{cor:Dual-D-0}
The maps $\pi^{\ab}_1(U) \to \pi^{\adiv}_1(X,D)$ define an
isomorphism of profinite groups
\[
\pi^{\ab}_1(U) \xrightarrow{\cong} 
{\underset{D \in \Div_C(X)}\varprojlim} \pi^{\adiv}_1(X,D).
\]
\end{cor}
\begin{proof}
Combine Proposition~\ref{prop:Fil-Exhaustive}, \thmref{thm:Dual-D}
and the Pontryagin duality.
\end{proof}

\begin{remk}\label{remk:No-normal-Div}
The reader may note that $\Fil_D H^1(K)$ is meaningful as long as $X$
is integral and is regular at the generic points
of $|D|$. One can therefore define $\pi^{\adiv}_1(X,D)$ in this generality
by letting it be the Pontryagin dual of $\Fil_D H^1(K)$.
From this perspective, \thmref{thm:Dual-D} provides a Tannakian
interpretation of $\Fil_D H^1(K)$ if $X$ is normal.
\end{remk}

Using Remark~\ref{remk:Curve-defn}, \thmref{thm:Dual-D} and
\cite[Corollary~2.8]{Kerz-Saito-2}), we get the following.

\begin{cor}\label{cor:SNCD}
  If $X$ is smooth and $D_\red$ is a simple normal crossing divisor on $X$, then
  there is an isomorphism
  \[
    \pi^{\adiv}_1(X,D) \xrightarrow{\cong}   \pi^{\ab}_1(X,D).
  \]
\end{cor}

\subsection{Zariski-Nagata purity theorem for 
$\pi^{\adiv}_1(X,D)$}\label{sec:purity-D}
Using \thmref{thm:Dual-D}, we can extend the Zariski-Nagata purity theorem
for $\pi^{\ab}_1(U)$ to $\pi^{\adiv}_1(X,D)$ as follows.
This is not known for $\pi^{\ab}_1(X,D)$.
Let $A \subset X$ be a closed subset 
such that $\dim(A) \le \dim(X) - 2$. Let $X' = X \setminus A$.
Set $U' = U \cap X'$ and $D' = j^*(D)$, where $j \colon X' \inj X$
is the inclusion.

\begin{thm}\label{thm:ZN-purity}
Assume that $U$ is regular. Then the canonical map
$j_* \colon \pi^{\adiv}_1(X',D') \to \pi^{\adiv}_1(X,D)$ is an isomorphism.
\end{thm}
\begin{proof}
By \thmref{thm:Dual-D}, it suffices to show that the inclusion
$\Fil_D H^1(K) \subset \Fil_{D'} H^1(K)$ is a bijection.
So let $\chi \in \Fil_{D'} H^1(K)$. Since $j^* \colon H^1(U) \to 
H^1(U')$ is an isomorphism by the Zariski-Nagata purity theorem for
{\'e}tale fundamental group and the Pontryagin duality, it follows that 
$\chi \in H^1(U)$. Since no generic point of $C = |D|$ can lie in $A$,
it is now clear that $\chi \in \Fil_D H^1(K)$.
\end{proof}

\section{The reciprocity map with modulus}\label{sec:R-fixed-D}
In this section, we shall prove \thmref{thm:Main-1}. We work with the following set-up.
Let $k$ be a finite field and $X \in \Sch_k$ an integral 
normal scheme of dimension $d \ge 1$. Let $D \subset X$ be a closed subscheme
of pure codimension one and $C = D_\red$. We let $U = X \setminus C$.
Let $\eta$ denote the generic point of $X$ and
let $K = k(\eta)$. 
%We fix a separable closure $\ov{K}$ of $K$ and
%assume all separable algebraic extensions of $K$ to be inside $\ov{K}$.
The main result of this paper is the following theorem, which is more general
than \thmref{thm:Main-1} because it  does not assume $U$ to be regular. 
Moreover, it implies \thmref{thm:Main-1} in view of \thmref{thm:Kerz-global}.

\begin{thm}\label{thm:Rec-D-map}
There exists a continuous reciprocity homomorphism $\rho_{X|D} \colon C(X,D) 
\to \pi^{\adiv}_1(X,D)$ such that the diagram
\begin{equation}\label{eqn:Rec-D-map-0}
\xymatrix@C.8pc{
{C}_{U/X} \ar[r]^-{\rho_{U/X}} \ar@{->>}[d] & \pi^{\ab}_1(U) \ar@{->>}[d] \\
C(X,D) \ar[r]^-{\rho_{X|D}} & \pi^{\adiv}_1(X,D)}
\end{equation}
is commutative.
\end{thm}
\begin{proof}
Let $p_{X|D} \colon {C}_{U/X} \to C(X,D)$ and $q_{X|D} \colon
\pi^{\ab}_1(U) \to \pi^{\adiv}_1(X,D)$ denote the quotient maps.
Since $C(X,D)$ is a quotient of ${C}_{U/X}$ with quotient topology,
we only need to show that $q_{X|D} \circ \rho_{U/X}$ annihilates the
relative Milnor $K$-groups of maximal Parshin chains (with respect to the
canonical dimension function) on $(U \subset X)$
relative to the ideal of $D$, where 
$\rho_{U/X} \colon C_{U/X} \to \pi^{\ab}_1(U)$ is the reciprocity map
from \thmref{thm:Rec-main}.

Let $P = (p_0, \ldots , p_d)$ be a maximal Parshin chain on $(U \subset X)$. 
Let $I_D$ be the image of
the ideal sheaf $\sI_D$ in $\sO^h_{X,P'}$, where
recall that $P' = (p_0, \ldots , p_{d-1})$.
We let $x = \Spec(k(p_{d-1}))$ and $C_x = \ov{\{x\}}$ so that $C_x$ is an 
irreducible component of $C$. Let $n_x$ be the multiplicity of $C_x$ in $D$.
We then know that $I_D = \fm^{n_x}$, where $\fm$ is the Jacobson 
radical of $\sO^h_{X,P'}$ (see the last paragraph of \S~\ref{sec:Pchain}). 
We let $V \subset K$ be a $d$-DV dominating $P$ and let
$K'$ be the quotient field of $V^h$.
We have seen in \S~\ref{sec:valuation} that $K'$ is a $d$-dimensional Henselian
local field whose ring of integers $\sO_{K'}$ dominates $\sO^h_{X,x}$.
We have $I_D \sO_{K'} = \fm^{n_x}_{K'}$, where $\fm_{K'}$
is the maximal ideal of $\sO_{K'}$.
It suffices to show that the image of 
$K^M_d(\sO_{K'}, \fm^{n_x}_{K'}) = U'_{n_x}K^M_d(K')$ under the composite map
$K^M_d(K') \to I_{U/X} \xrightarrow{\rho_{U/X}} \pi^{\ab}_1(U)$
dies in $\pi^{\adiv}_1(X,D)$ (see \S~\ref{sec:Mat-0}).

Now, it follows from the construction of $\rho_{U/X}$ (see \S~\ref{sec:Rec-U})
that there is a diagram

\begin{equation}\label{eqn:Rec-D-map-3}
\xymatrix@C.8pc{
K^M_d(K') \ar[d] \ar[r]^-{\rho_{K'}} & \Gal({{K'}^{\ab}}/{K'}) \ar[d]^-{q_V} & \\ 
I_{U/X} \ar[r]_-{\rho_{U/X}} & \pi^{\ab}_1(U) \ar[r]^-{q_{X|D}} & 
\pi^{\adiv}_1(X,D),} 
\end{equation}
where the square on the left is commutative.
We need to show that $q_{X|D} \circ q_V \circ \rho_{K'}$ annihilates
$U'_{n_x}K^M_d(K')$. Using the Pontryagin duality for profinite groups,
this is equivalent to showing that every continuous character of
$\pi^{\adiv}_1(X,D)$ annihilates the image of $U'_{n_x}K^M_d(K')$.
By \thmref{thm:Dual-D}, this is equivalent to showing that
for every element $\chi \in \Fil_D H^1(K)$, the composite character 
$\chi \circ q_V$ annihilates the image of $U'_{n_x}K^M_d(K')$ in
$\Gal({{K'}^{\ab}}/{K'})$.

However, $\chi \in \Fil_D H^1(K) \subset H^1(U)$ implies that 
$\chi_x \in \Fil^{\ms}_{n_x} H^1(K_x)$. The factorization
$H^1(U) \to H^1(K_x) \to H^1(K')$
then tells us that $\chi \circ q_V \in \Fil^{\ms}_{n_x} H^1(K')$. 
%Use \thmref{thm:Filt-Milnor-*}
But this forces $\chi \circ q_V$ to annihilate
$U'_{n_x}K^M_d(K')$ by \thmref{thm:Filt-Milnor-*}.
\end{proof}

\begin{cor}\label{cor:Rec-refined}
The reciprocity map $\rho_{U/X}$ factorizes into the composition of
continuous homomorphisms making the diagram 
\begin{equation}\label{eqn:Rec-refined-0}
\xymatrix@C.8pc{
C_{U/X} \ar[r] \ar@{->>}[dr] \ar@/^2pc/[rr]^-{\rho_{U/X}} & \wt{C}_{U/X}
\ar[r]^-{\wt{\rho}_{U/X}} \ar@{->>}[d]^-{p_{X|D}} & \pi^{\ab}_1(U) 
\ar@{->>}[d]^-{q_{X|D}} \\
& C(X,D) \ar[r]^-{\rho_{X|D}} & \pi^{\adiv}_1(X,D)}
\end{equation}
commutative.
\end{cor}
\begin{proof}
The factorization follows from the construction of $\rho_{U/X}$,
definition of $\wt{C}_{U/X}$, \corref{cor:Dual-D-0} and
\thmref{thm:Rec-D-map}.
\end{proof}

\subsection{Agreement with the Kato-Saito reciprocity map}
\label{sec:Agreement}
For every $L \in I^{\ab}_{D/K}$, it was shown by Kato-Saito in
\cite[\S~3.7]{Kato-Saito-2} that there exists a closed subscheme
$D' \subset X$ with support $C$ such that $D' \ge D$, and a reciprocity
map $\rho'_{D'|L} \colon C_{KS}(X,D') \to \Gal(L/K)$.

\begin{lem}\label{lem:Agreement*}
Assume that $U$ is regular. Then the diagram
\begin{equation}\label{eqn:Agreement*-0}
\xymatrix@C.8pc{
C(X,D') \ar@{->>}[r] \ar[d]_-{\psi_{X|D'}} & C(X,D) \ar[r]^-{\rho_{X|D}}
& \pi^{\adiv}_1(X,D) \ar@{->>}[d] \\
C_{KS}(X,D') \ar[rr]^-{\rho'_{D'|L}} & & \Gal(L/K)}
\end{equation}
is commutative. In particular, $\wt{\rho}_{U/X}$ agrees with
the Kato-Saito reciprocity map given by \cite[Theorem~9.1(3)]{Kato-Saito-2}.
\end{lem}
\begin{proof}
It follows from the construction of the isomorphism $\psi_{X|D'}$ in
\thmref{thm:Kerz-global} that for every closed point $x \in U$, the diagram
\begin{equation}\label{eqn:Agreement*-0*}
\xymatrix@C.8pc{
K^M_0(k(x)) \ar[r] \ar[dr] & C(X,D') \ar[d]^-{\psi_{X|D'}}_-{\cong} \\
& C_{KS}(X,D')}
\end{equation}
is commutative,
where the top horizontal arrow is the canonical map from the 
Milnor $K$-theory of the Parshin chain $\{x\}$ and the
diagonal arrow is the forget support map
$H^d_x(X, \sK^M_{d, (X,D')}) \to H^d_\nis(X, \sK^M_{d, (X,D')})$.
It suffices therefore to show using \cite[Theorem~2.5]{Kato-Saito-2}
that for every closed point $x \in U$, 
the diagram ~\eqref{eqn:Agreement*-0} commutes if we replace
both $C(X,D')$ and $C_{KS}(X,D')$ by $K^M_0(k(x))$.
But both maps are simply the Frobenius substitution by 
the construction of $\rho_{U/X}$ in \thmref{thm:Rec-main}
and by \cite[Proposition~3.8]{Kato-Saito-2}.
\end{proof}

%\enlargethispage{20pt}

\subsection{The reciprocity map for degree zero subgroups}
\label{sec:Deg-0}
Assume that $X$ is proper over $k$ and $U$ is regular.
The structure map $X \to \Spec(k)$ defines natural maps
$\pi^{\ab}_1(U) \surj \pi^{\adiv}_1(X,D) \surj \pi^{\ab}_1(X) \to \Gal(\ov{k}/k) 
\cong \wh{\Z}$.
The latter map is surjective if $X$ is geometrically connected.
Let $\deg'$ denote the composite map.
We denote the composite 
map $\pi^{\adiv}_1(X,D) \surj \pi^{\ab}_1(X) \to \Gal(\ov{k}/k)$
also by $\deg'$.
We let $\pi^{\ab}_1(U)_0 = \Ker(\deg' \colon \pi^{\ab}_1(U) \to \wh{\Z})$
and $\pi^{\adiv}_1(X,D)_0 = \Ker(\deg' \colon \pi^{\adiv}_1(X,D) \to \wh{\Z})$. 
We also have the degree map
$\deg \colon C(X,D) \to \Z$ by \propref{prop:Degree}.

\begin{lem}\label{lem:Rec-proj}
There exists an integer $n \ge 1$
depending only on $U$ and a commutative diagram of exact 
sequences of topological abelian groups
\begin{equation}\label{eqn:Rec-proj-0}
\xymatrix@C1pc{
0 \ar[r] & C(X,D)_0 \ar[r] \ar[d] & C(X,D) \ar[r]^-{{\deg}/n} 
\ar[d]^-{\rho_{X|D}} & \Z  \ar@{^{(}->}[d]  \ar[r] & 0 \\
0 \ar[r] & \pi^{\adiv}_1(X,D)_0 \ar[r] & \pi^{\adiv}_1(X,D) \ar[r]^-{{\deg'}/n} & 
\wh{\Z} \ar[r] & 0,}
\end{equation}
where the right vertical arrow is the profinite completion map.
\end{lem}
\begin{proof}
We need to show that the right square commutes.
We first show that this square commutes without dividing the
degree maps by any integer $n$.
Since the map $Z_0(U) \to C(X,D)$ is surjective
by \thmref{thm:Kerz-global} and \cite[Theorem~2.5]{Kato-Saito-2},
it suffices to show that for every closed point $x \in U$,
the diagram 
\begin{equation}\label{eqn:Rec-proj-1}
\xymatrix@C.8pc{
K_0(k(x)) \ar[d]_-{\rho_{k(x)}} \ar[r]^-{\deg} & \Z \ar[d] \\
\pi^{\ab}_1(k(x)) \ar[r]^-{q_*} & \wh{\Z}}
\end{equation}
commutes, where $q \colon \Spec(k(x)) \to \Spec(k)$ is the projection.  
But this is classical.

We now let $M$ and $M'$ be the images of the maps
$\deg \colon C_{U/X} \to \Z$ and $\deg' \colon \pi^{\ab}_1(U) \to \wh{\Z}$, 
respectively.
We must have $M = n\Z$ for some integer $n \ge 1$.
Moreover, $M$ and $M'$ must also be the images of
$\deg \colon C(X,D) \to \Z$ and 
$\pi^{\adiv}_1(X,D) \to \wh{\Z}$, respectively.
We thus get a commutative diagram of exact sequences
\begin{equation}\label{eqn:Rec-proj-2}
\xymatrix@C.8pc{
0 \ar[r] & M \ar[r] \ar[d] & \Z \ar[r] \ar[d] & {\Z}/n \ar[r] \ar[d] & 0 \\
0 \ar[r] & M' \ar[r] & \wh{\Z} \ar[r] & M'' \ar[r] & 0.}
\end{equation}

We identify $M$ with $\Z$ via the map $\Z\xrightarrow{n} M$.
By the functoriality of the profinite completion, there is a
commutative diagram
\[
\xymatrix@C.8pc{
\Z \ar[r]^-{n} \ar[d] & {\Z} \ar[d] \\
\wh{\Z} \ar[r]^-{n} & \wh{\Z},}
\]
where the vertical arrows are profinite completion maps.
Since ${\Z}/m \xrightarrow{n} {\Z}/m$ is continuous for all $m \ge 1$,
it follows that the bottom horizontal arrow is 
continuous. It follows from the universal property of profinite
completion and ~\eqref{eqn:Rec-proj-2}
that the image of this bottom horizontal arrow lies inside $M'$.
Since the image $n \wh{\Z}$ of this map is closed, we 
see that $\wh{\Z} \xrightarrow{n} \wh{\Z}$ is an embedding into
a closed subgroup contained in $M'$. 

On the other hand, the Chebotarev-Lang density theorem
(see \cite[Theorem~5.8.16]{Szamuely}) says that the image of
$\rho_{U/X}$ is dense. Hence,  the image of 
$C(X,D)$ is dense in $\pi^{\adiv}_1(X,D)$. Since the latter maps
onto $M'$ via $\deg'$, it follows that $M$ is dense in $M'$. 
Since $n \wh{\Z}$ contains $M$ and is closed,
it must contain $M'$.
We therefore conclude that $M' = n \wh{\Z}$. 
It also follows from this that the right vertical arrow of
~\eqref{eqn:Rec-proj-2} is an isomorphism.
\end{proof}

\begin{cor}\label{cor:Rec-proj-3}
The diagram ~\eqref{eqn:Rec-proj-0} induces a
commutative diagram of short exact sequences of discrete abelian groups
\begin{equation}\label{eqn:Rec-proj-4}
\xymatrix@C.8pc{
0 \ar[r] & (\wh{\Z})^\vee \ar[r] \ar[d]_-{\cong} & 
(\pi^{\adiv}_1(X,D))^\vee \ar[r] 
\ar[d]^-{\rho^{\vee}_{X|D}} & (\pi^{\adiv}_1(X,D)_0)^\vee \ar[r] 
\ar[d]^-{\rho^{\vee}_{X|D}} & 0 \\
0 \ar[r] & \Z^{\vee} \ar[r] & C(X,D)^\vee \ar[r] & (C(X,D)_0)^\vee \ar[r] & 0,}
\end{equation}
where the left vertical arrow is an isomorphism.
\end{cor}
\begin{proof}
We only need to show that the  left vertical arrow is an isomorphism as 
everything else follows from Lemmas~\ref{lem:Pont-dual} and
~\ref{lem:Rec-proj}.
However, it is well known that the profinite completion map
$\Z \inj \wh{\Z}$ induces the map
$\Hom_{\cont}(\wh{\Z}, {\Q}/{\Z}) \to \Hom_\Ab(\Z, {\Q}/{\Z})$ which is an
isomorphism.
\end{proof}

%\vskip .4cm

%\enlargethispage{10pt}

\noindent\emph{Acknowledgements.}
Gupta was supported by the
SFB 1085 \emph{Higher Invariants} (Universit\"at Regensburg).
He would like to thank TIFR, Mumbai for
invitation in Spring 2020, and Moritz Kerz for some fruitful 
discussion on this paper The authors would like to thank the referee for helpful
comments and suggestions.

\end{document}